%% file: MaximalWordsOfBoundedDegree.tex
\documentclass[orivec,runningheads]{llncs}

\usepackage[utf8]{inputenc}
\usepackage[T1]{fontenc}
\usepackage[export]{adjustbox}
\usepackage[small]{caption}
\input{macros}

\author{Alexandre Blondin Massé and Alain Goupil and Raphael L'Heureux and Louis Marin }

\begin{document}
	
	\title{Maximal $2$-dimensional binary words of bounded degree}
	
	\author{Alexandre Blondin Massé\inst{1,3} \and
		Alain Goupil\inst{2,3} \and
		Raphael L'Heureux\inst{2} \and 
		Louis Marin\inst{4}}
	
	\authorrunning{A. Blondin Massé et al.}
	
	\institute{Université du Québec à Montréal, Montréal, Canada \and
		Université du Québec à Trois-Rivières, Trois-Rivières, Canada \and
		LACIM, Montréal, Canada \and
		Université Gustave Eiffel, LIGM, Champs-sur-Marne, France}
	
	\maketitle

	\begin{abstract}
		Let $d \in \{0, 1, 2, 3, 4\}$ and $W$ be a $2$-dimensional word of dimensions $h \times w$ on the binary alphabet $\{\e,\o\}$, where $h,w \in \Zp$.
		Assume that each occurrence of the letter $\o$ in $W$ is adjacent to at most $d$ letters $\o$ and let $|W|_{\o}$ be the  number of letters $\o$ in $W$.
		We provide an exact formula for the maximum value of $|W|_{\o}$ for fixed $(h, w)$.
		As a byproduct, we deduce an upper bound on the length of maximum snake polyominoes contained in a $h \times w$ rectangle. 
	\end{abstract}
	
	\section{Introduction}
	We adapt the terminology for $2$-dimensional words from~\cite{giammarresi1997two,morita2004two}.
	For $a,b \in \Zp$, let $\DI{a,b} = \{n \in \Zp \mid a \leq n \leq b\}$.
	Let $A$ be a finite alphabet and $h,w \in \Zp$.
	A \emph{$2$-dimensional word $W$ of dimension $h \times w$ on $A$} is a matrix of $h$ rows and $w$ columns with entries in $A$.
	The set of all $2$-dimensional words of dimensions $h \times w$ on $A$ is denoted by $\WordshwA$.
	Given $W\in \WordshwA$ and $a \in A$, we denote by $|W|_a$ the number of occurrences of $a$ in $W$. 
	For $i \in \DI{1,h}$ and $j \in \DI{1,w}$, the entry $a$ of $W$ at the intersection of the $i$-th row and $j$-th column is written $W[i,j]$ and is called \emph{the $a$-cell of $W$ at $(i,j)$}. 
	The \emph{horizontal concatenation} of $U \in \WordshwAvar{h}{w_1}$ and $V \in \WordshwAvar{h}{w_2}$  is the word $U \hcat V \in \WordshwAvar{h}{(w_1+w_2)}$ defined by
	\[(U \hcat V)[i,j] = \begin{cases}
		U[i,j], & \mbox{if $j \leq w_1;$} \\
		V[i,j - w_1], & \mbox{if $j > w_1.$}
	\end{cases}\]
	The \emph{vertical concatenation} $U\vcat V$ of $U \in \WordshwAvar{h_1}{w}$ and $V \in \WordshwAvar{h_2}{w}$, is defined similarly.

		Given $W \in \WordshwAvar{h}{w}$ and $m,n \in \Q$ such that $mh, nw \in \Z$, the \emph{$(m,n)$-th power of $W$} is the word $W^{m \times n} \in \WordshwAvar{mh}{nw}$ such that $W^{m \times n}[i,j] = W[(i - 1\bmod h) + 1, (j - 1\bmod w) + 1]$.
		The \emph{neighborhood of $W[i,j]$}, denoted by $N_W(i,j)$, is defined as
	
		\[N_W(i,j) = \{(i-1,j), (i+1,j), (i,j-1), (i,j+1)\} \cap (\DI{1,h} \times \DI{1,w})\]
		
		The $i$-th row of $W$ is denoted by $W[i]$ and the factor obtained by selecting all rows between $i$ and $i'$, both included, is denoted by $W\DI{i,i'}$.	
		For any $a \in A$, the \emph{$a$-degree} of  $W[i,j]$ is given by
		\[
		\deg_{a,W}(i,j) = \begin{cases}
			\sum_{(i',j') \in N_W(i,j)} \indic(W[i',j'] = a) & \mbox{if $W[i,j] = a$;} \\
			0, & \mbox{if $W[i,j] \neq a$,}
		\end{cases}
		\]
		
		where $\indic$ is the usual indicator function.
		\begin{figure}[t]
			\centering
			\input{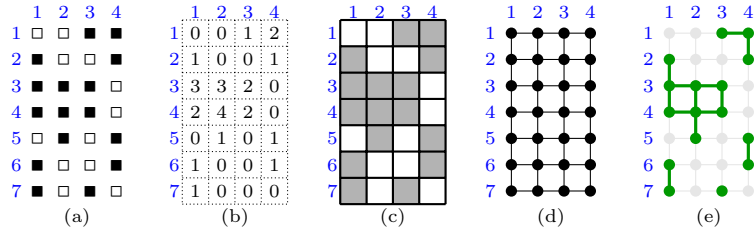}
			\caption{(a) A $7 \times 4$ word $W$  on the alphabet $\{\e,\o\}$.
				(b) The $\o$-degree word 
				of $W$.
				(c) The geometric representation of $W$.
				(d) The grid graph $G_{7,4}$. 
				(e) The subgraph (in green) of $G_{7,4}$ (in pale gray) induced by $W$. 
				Row and column indices of $W$ appear in blue.}\label{fig:2d-word}
		\end{figure}	
		For example, if $W$ is the word illustrated in Figure \ref{fig:2d-word}(a), then $\deg_{\o,W}(2,3) = 0$ and $\deg_{\o,W}(3,3) = 2$. 
		The \emph{$a$-degree word} of $W$, denoted by $\deg_a(W)$, is the $2$-dimensional word of dimensions $h \times w$ on the alphabet $\{0,1,2,3,4\}$ with entry at $(i,j)$ given by  $\deg_{a,W}(i,j)$ (Figure~\ref{fig:2d-word}(b)).
		In this paper, all binary words are taken on the alphabet $\{\e,\o\}$ and for brevity, we write $\Wordshwshort$ instead of $\Wordshw$ and ``degree'' instead of ``$\o$-degree''.  $|W|_{\o}$ is also called the \emph{area} of $W$.
		Given $d \in \DI{1,4}$, the set of all $2$-dimensional words of dimensions $h \times w$ with $\o$-cells of degree
		bounded by $d$ is denoted by $\Wordsdhwvarshort{d}{h}{w}$:
		\[\Wordsdhwvarshort{d}{h}{w} = \{W \in \Wordshwshort \mid \deg_{\o}(W)[i,j] \leq d, \mbox{for $(i,j) \in \DI{1,h} \times \DI{1,w}$}\}.\]
	
		For any $d \in \DI{1,4}$, $h, w \in \Zp$, let
		$\maxdhw{d}{h}{w} = \max\{|W|_{\o} : W \in \Wordsdhwshort\},$
		be the maximum number of filled cells of degree at most $d$ in a $h \times w$ 2D word.
		Clearly, $\maxdhw{d}{h}{w}$ is symmetric with respect to $(h,w)$ i.e. $\maxdhw{d}{h}{w} = \maxdhw{d}{w}{h}$.
		Moreover, let $G = (V,E)$ be a simple graph and $U \subseteq V$.
		Then $U$ is called a \emph{dominating set} of $G$ if, for every vertex $v \in V\backslash U$, there exists $u \in U$ such that $\{u,v\} \in E$, i.e. $v$ has at least one neighbor in $U$.
		The \emph{domination number of $G$}, denoted by $\gamma(G)$, is the minimal cardinality of a dominating set of $G$.
		
		The main result of this document is the following theorem.	
		\begin{figure}[t]
			\centering
			\input{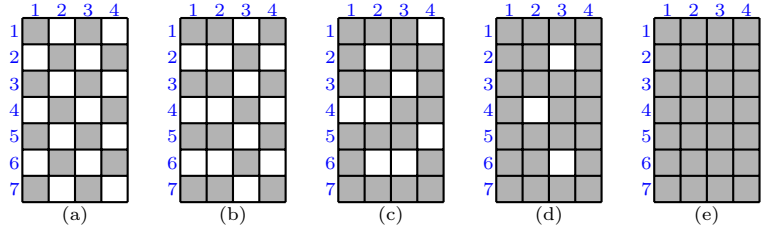}
			\caption{$7 \times 4$ $d$-full words, for (a) $d = 0$, (b) $d = 1$, (c) $d = 2$, (d) $d = 3$ and (e) $d = 4$.}\label{fig:full-words}
		\end{figure}
		
		\begin{theorem}\label{thm:main}
			Let $d \in \DI{1,4}$, $h, w \in \Zp$, $h \geq w$ and $\mdhw{d}{h}{w}$ be defined by
			
			\footnotesize
			\begin{eqnarray*}
				\mdhw{0}{h}{w} & = & \lceil hw / 2 \rceil, \\
				\mdhw{1}{h}{w} & = & \begin{cases}
					hw/2, & \mbox{if $h, w \equiv_2 0$;} \\
					(h - 1)w/2 + \lceil 2w / 3 \rceil,
					& \mbox{if $h \equiv_2 1$ and $w \equiv_2 0$;} \\
					h(w - 1)/2 + \lceil 2h / 3 \rceil,
					& \mbox{otherwise,} 
				\end{cases} \\
				\mdhw{2}{h}{w} & = & \begin{cases}\label{m(2,h,w)}
					hw, & \mbox{if $w = 1$ or $h = w = 2$;} \\
					3hw/4+1/2, & \mbox{if $h \equiv_2 1$, $h \geq 3$ and $w = 2$;} \\
					3hw/4, & \mbox{if $h \equiv_2 0$, $h \geq 4$ and $w = 2$;} \\
					2hw/3 + 2, & \mbox{if $w = 3$ or $h \equiv_3 w \equiv_3 0$;} \\
					2hw/3 + 4/3, & \mbox{if $w \geq 4$ and $h \equiv_3 w \nequiv_3 0$;} \\
					2hw/3 +  1, & \mbox{if $w \geq 4$,  $w \equiv_3 0$ and $h\nequiv_3 w$;} \\
					2hw/3 +  2/3, & \mbox{otherwise,}
				\end{cases} \\
				\mdhw{3}{h}{w} & = & \begin{cases}
					hw, & \mbox{if $1 \leq w \leq 2$;} \\
					hw - \gamma(G_{h-2,w-2}), & \mbox{otherwise,}
				\end{cases} \\
				\mdhw{4}{h}{w} & = & hw.
			\end{eqnarray*}
			\normalsize
			
			Then $\maxdhw{d}{h}{w} = \mdhw{d}{h}{w}$.
		\end{theorem}
		
		A $2$-dimensional word $W \in \Wordsdhwshort$ is called \emph{$d$-full} or maximal  if $|W|_{\o} = \maxdhw{d}{h}{w}$. The set of  {$d$-full words of degree at most $d$ is denoted $\MWordsdhwvarshort{d}{h}{w}$.
			In Figure~\ref{fig:full-words}, examples of $d$-full words of dimensions $7 \times 4$ are illustrated, for $d \in \DI{1,4}$.

			\paragraph{\bf Motivation.}
			
			The study of multidimensional words has generated  interest in the combinatorics on words community~\cite{carpi1988multidimensional,giammarresi1997two,morita2004two}.
			A famous conjecture, presented by Nivat in his invited talk at ICALP \cite{nivat}, stating that $2$-dimensional words with low rectangular complexity are  periodic, has generated extensive literature \cite{cyr2015nonexpansive,epifanio2003conjecture,sander2002rectangle,szabados2018nivat}.
			Recently, Mahalingam and Pandoh have considered  $2D$ generalizations of palindromes~\cite{mahalingam2022hv}.
			Theorem~\ref{thm:main} can be seen as a result on pattern avoidance \cite{berstel2009combinatorics} as words in $\Wordsdhwvarshort{d}{h}{w}$ avoiding specific $3 \times 3$ rectangles.
			Also, beyond its intrinsic interest, Theorem~\ref{thm:main} has consequences on the study of polyominoes.
			It is also convenient to interpret binary words in $\Wordsdhwvarshort{d}{h}{w}$ from geometric and graph-theoretic points of view.\\
			Let $W \in \Wordsdhwvarshort{d}{h}{w}$.
	
			$W$ can also be seen as a $h \times w$ rectangle in which  $\o$ cells  are filled unit cells, and  $\e$ cells are empty unit cells (Figure~\ref{fig:2d-word}(c)).
			A \emph{horizontal} (resp. \emph{vertical}) \emph{$n$-pillar in $W$} is a factor of the form $\o^{1\times n}$ (resp. $\o^{n \times 1}$) in $W$. An $n$-pillar is \emph{maximal} when it is not a proper factor of a larger pillar in its row (or column) and  it is a solitary pillar when it is maximal and its  adjacent cells are empty.
			For $n \geq 3$, a \emph{north $n$-bench} is a factor of the form $W = \substack{\o\\\o} \left(\substack{\e\\\o}\right)^{1 \times (n-2)} \substack{\o\\\o}$.
			The bottom filled row is called the \emph{seat} of the bench and the two extremal columns are  called the \emph{legs} of the bench. 
			\emph{South, east and west $n$-benches} are defined similarly, with the two legs indicating the cardinal direction. 
			The sequence $(|W[1]|_\o,|W[2]|_\o,\ldots , |W[h]|_\o)$ giving the number of filled cells in each row $W[i]$ of a word $W$ is called the \emph{row distribution} of $W$.
			The row distribution of the word in Figure \ref{fig:2d-word} is $(2,2,3,3,2,2,2)$.
			We denote by $W^t$ the transpose of the matrix $W$ so that $W^t[j]$ is the $j$-th column of $W$.
			
			$W \in \Wordsdhwvarshort{d}{h}{w}$ is also related to graph-theoretical concepts.
			Recall that the \emph{grid graph $G_{h,w}$ of dimensions $h \times w$} is the simple graph with set of vertices $\DI{1,h} \times \DI{1,w}$ and set of edges $\{(i,j),(i',j')\}$, where $i,i' \in \DI{1,h}$, $j,j' \in \DI{1,w}$ and $(i - i')^2 + (j - j')^2 = 1$ (Figure~\ref{fig:2d-word}(d)).
			The \emph{subgraph induced by $W$ in $G_{h,w}$}, denoted by $G[W]$, 
			
			is the subgraph of $G_{h,w}$ induced by the set of vertices $\{(i,j) \in \DI{1,h} \times \DI{1,w} : W[i,j] = \o\}$ (Figure~\ref{fig:2d-word}(e)).
			There exists an extensive literature studying the induced subgraphs of bounded degree, in particular induced subgraphs of planar graphs~\cite{amini2008degree}.
			Theorem~\ref{thm:main} provides an exact solution for the case of grid graphs where the induced subgraphs are linear forests of bounded degree $d$.

				Finally, the study of words in $\Wordsdhwvarshort{d}{h}{w}$ is useful to the study of polyominoes.
			
				A \emph{snake polyomino} (or simply a \emph{snake}) is a polyomino inducing a chain subgraph, while a \emph{snake forest} (also called a \emph{linear forest}) is a set of polyominoes inducing a forest of chains as a subgraph.
			
				Therefore, snakes and snake forests are subsets of $\Wordstwhwshort$ and Theorem \ref{thm:main} provides an upper bound for their maximal size in a $h\times w$ rectangle, which is tight for $w\in \DI{1,4}$.
				We can also use the construction of $2$-full words in the proof of Theorem \ref{thm:main} to establish a lower bound for the maximal size of  snakes  contained in a $h\times w$ rectangle.
				Figure \ref{fig:full-words}(c) shows  a snake polyomino of maximal size.
				The remainder of this paper is devoted to the proof of Theorem~\ref{thm:main}.
				
				\section{The cases $d \in \{0, 1, 3, 4\}$}
				
				The case $d=4$ is immediate. In this section, we study the cases  $d \in \{0, 1, 3\}$, keeping the case $d = 2$ for Section~\ref{sec:d=2}.
				The reasons behind this division are 
				(1) the cases $d \in \{0, 1\}$ are  easy to prove,
				(2) the case $d = 3$ is proved by establishing a relation with  the dominating set problem on grid graphs, which is  solved in ~\cite{gonccalves2011domination}, 
				(3) the case $d = 2$ is  involved and requires a thorough combinatorial study.
				
				We start with the case $d = 0$.
				
				\begin{lemma}\label{lem:d=0}
					For any $(h,w) \in \DIhw$, $\maxdhw{0}{h}{w} = \lceil hw / 2 \rceil$.
				\end{lemma}
				
				\begin{figure}
					\centering
					\input{tikz/d0.tikz}
					\caption{The word $W = \left(\substack{\o\e\\\e\o}\right)^{h/2 \times w/2}$ superimposed with an arbitrary dominoes/monomino tiling, when the dimensions are (a) $6 \times 6$ (b) $5 \times 5$.}\label{fig:d=0}
				\end{figure}
				
				\begin{proof}
					Let $W = \left(\substack{\o\e\\\e\o}\right)^{h/2 \times w/2} \in \Wordshwshort$ (see Figure~\ref{fig:d=0}).
					Since each $\o$-cell of $W$ is surrounded by $\e$-cells,
					$\deg_{\o,W}(i,j) = 0$ for all $(i,j) \in \DIhw$, which implies that $W \in \Wordszhwshort$.
					Since $|W|_\o = \lceil hw/2 \rceil$, we have  $\maxdhw{0}{h}{w} \geq \lceil hw/2 \rceil$.
					We prove that $\lceil hw/2 \rceil$ is an upper bound with a pigeonhole principle argument.
					Let $W \in \Wordsdhwvarshort{0}{h}{w}$.
					Then $W$ can be partitioned into $\lfloor hw/2 \rfloor$ dominoes and $hw \bmod 2 \in \{0,1\}$ monomino (see Figure~\ref{fig:d=0}).
					Since at most one cell of each domino can be filled and since a monomino contains at most one filled cell, there are at most $\lfloor hw/2 \rfloor + hw \bmod 2 = \lceil hw/2 \rceil$ filled cells in $W$, i.e. $\maxdhw{0}{h}{w} \leq \lceil hw/2 \rceil$.
					Hence, $\maxdhw{0}{h}{w} = \lceil hw/2 \rceil$.
				\end{proof}
				
				Due to lack of space, we only provide a sketch of the proof for the case $d = 1$.
				
				\begin{lemma}\label{lem:d=14}
					Theorem~\ref{thm:main} holds for $d=1$.
				\end{lemma}
				
				\begin{proof}[sketched]
					This is proved using an argument similar to the one presented in the proof of Lemma~\ref{lem:d=0}, by partitioning the rectangle with $2 \times 2$ and $1 \times 3$ tiles.
					The detailed proof is available in the Appendix.
				\end{proof}	
				
				For the case $d = 3$, we need additional definitions.
				The \emph{boundary} $\bd(W)$ of $W$ is the set
				$\bd(W) = \{(i,j) \in \DI{1,h} \times \DI{1,w} \mid \mbox{$i \in \{1,h\}$ or $j \in \{1,w\}$}\}$.
				We say that $(i,j)\in \bd(W)$ is \emph{in a corner} of $W$ if $i \in \{1,h\}$ and $j \in \{1,w\}$.  Otherwise, $(i,j)$ is \emph{on the side} of $W$. 
				Since an expression for the domination number of any grid graph was recently proved in \cite{gonccalves2011domination}, we take advantage of that expression to reduce the problem.
				
				\begin{lemma}\label{lem:d=3}
					Theorem~\ref{thm:main} holds for $d = 3$.
				\end{lemma}
				
				\begin{proof}
					If $1 \leq h \leq 2$ or $1 \leq w \leq 2$, then it suffices to define $W[i,j] = \o$ for all $(i,j) \in \DIhw$, which makes $W$ obviously $3$-full. Now assume that $h, w \geq 3$. 
					
					First, we show that there exists a $3$-full word $W \in \Wordshwshort$ such that $W[i,j] = \o$ for all $(i,j) \in \bd(W)$, i.e. the cells on the boundary of $W$ are filled. Arguing by contradiction, assume that no such word exists and let $W$ be a $3$-full word with an empty cell on the left side on $W$, i.e. there exists $i \in \DI{h}$ such that $W[i,1] = \e$. Clearly, the $\e$-cell cannot be in a corner of $W$, i.e. $i \neq 1, h$: If it is the case, then the word $W'$ obtained from $W$ by replacing the entry at $(i,1)$ by $\o$ would satisfy $W' \in \Wordsthhwshort$ and $|W'|_{\o} = |W|_\o + 1$, contradicting the assumption that $W$ is $3$-full. Hence, $i \neq 1,h$. Moreover, the cell at the right of $(i,1)$ must be filled, i.e. $W[i,2] = \o$: If it is not the case, then the word $W''$ obtained from $W$ by replacing the entry at $(i,1)$ by $\o$ would satisfy $W' \in \Wordsthhwshort$ and $|W'|_{\o} = |W|_{\o} + 1$, also contradicting the assumption that $W$ is $3$-full. Finally, observe that the word $W'''$ obtained from $W$ by replacing the entry at $(i, 1)$ by $\o$ and the entry at $(i,2)$ by $\e$ satisfies $W' \in \Wordsthhwshort$, $|W'|_{\o} = |W|_{\o}$ and has  one empty cell less than $W$ on its boundary, contradicting the assumption that the boundary of $W$ contains a minimum number of  empty cells. Hence, there exists a $3$-full word $W \in \Wordshwshort$ with its boundary contains only filled cells.
					
					To conclude, let $G_{h-2,w-2}$ be the grid subgraph with set of vertices $\DI{2,h-1} \times \DI{2,w-1}$ and let $U = \{(i,j) \in \DIhw \mid W[i,j] = \e\}$.
					We observe that $W \in \Wordsthhw$ only if $U$ is a dominating set of $G_{h-2,w-2}$.
					Therefore, $W$ is $3$-full if and only if $U$ is a minimum dominating set of $G_{h-2,w-2}$. \qed
				\end{proof}
				
				\section{The case $d = 2$}\label{sec:d=2}
				
				This section is devoted to the proof of the case $d = 2$.
				As a first step, we introduce a simple but useful concept that facilitates the discussion:
				\begin{definition}[Excess of a word]\label{def:excess}
					Let $W\in\Wordshwshort$.
					The \emph{excess of $W$}, denoted by $e(W)$, is defined by $e(W) = |W|_{\o} - 2hw/3$.
					The maximal excess that can be realized by a word of dimensions $h \times w$ of degree $d$ less or equal to $2$ is  $e_{max}(h,w) = max\{ e(W):W \in \Wordsdhwvarshort{2}{h}{w}\}$.
				\end{definition}

				In other words, the excess of $W$ is the surplus, in number of filled cells, that $W$ has when $2/3$ of the area of its bounding rectangle is filled, and the maximal excess that a word can have if it is contained in a $h\times w$ rectangle $R$ is $e_{max}(h,w)$.
				If $R$ is the bounding rectangle  of $W$, we sometimes write $e(R)$ instead of $e(W)$. 	
				
				An immediate property of the excess function is that it is additive with respect to horizontal and vertical concatenation:
				\begin{proposition}\label{prop:additivity} 
					Let $W_1\in\mathcal{W}_{h \times w_1},W_2\in\mathcal{W}_{h \times w_2}$ be two words with respective sizes $h\times w_1$ and $h\times w_2$.   Similarly let 
					$W_3\in\mathcal{W}_{h_1 \times w},W_4\in \mathcal{W}_{h_2 \times w}$
					be two words of respective sizes $h_1\times w$ and $h_2\times w$. Then 
					$e(W_1\hcat W_2) = e(W_1)+e(W_2)$ and
					$e(W_3\vcat W_4) = e(W_3)+e(W_4)$.
				\end{proposition}

				Using Definition~\ref{def:excess}, the case $d = 2$ of Theorem~\ref{thm:main} can be reformulated:
				\begin{theorem}\label{thm:rephrased} Let $h,w\in \mathbb{Z}_{>0}$, with $h \geq w$. Then
					\label{thm:d=2}
					\begin{equation}\label{eq:emax}
						e_{max}(h,w) = \begin{cases}
							hw/3, & \mbox{if $w=1$ or $h=w=2$;}\\
							hw/12, & \mbox{if $h \equiv_2 0, h \geq 4$ and $w = 2$;}\\
							hw/12+1/2 & \mbox{if $h \equiv_2 1, h \geq 3$ and $w = 2$;}\\
							2 & \mbox{if $w = 3$ or $h\equiv_3 w\equiv_3 0$;}\\
							4/3 & \mbox{if $w \geq 4$ and $h\equiv_3 w \nequiv_3 0$;}\\
							1 & \mbox{if $w \geq 4$, $hw \equiv_3 0$ and $h \nequiv_3 w$;}\\
							2/3 & \mbox{otherwise.}
					\end{cases}\end{equation}
				
				\end{theorem}
				This equivalent form is particularly convenient in comparison with the definition of $\maxdhw{d}{h}{w}$, since the excess becomes bounded whenever $h, w \geq 3$.
			
				Finally, for $h, w \in \Zp$, let
				\begin{equation}
					\emax(h,w) = \mdhw{2}{h}{w} - 2/3(w\cdot h).
				\end{equation}
				Theorem \ref{thm:d=2} claims that $\emax(h,w)=e_{max}(h,w)$.
				The remainder of the paper is devoted to its proof.
				We begin with  the cases where the width $w$ is small.
				\begin{lemma}[Base cases]
					\label{lem:basecases}
					Theorem \ref{thm:d=2} holds for $(h,w) \in \left( \mathbb{Z}_{>0}\times \DI{1,6}\right)\cup (7,7)$. 
				\end{lemma} 
				
				\begin{proof}
					It is immediate that if $w=1$, then a $h$-pillar is both 2-full and element of $\Wordsdhwvarshort{2}{h}{1}$ which yields $e_{max} = h-2h/3 = h/3$.
					Due to space restriction, the proof of the other base cases are presented in the Appendix, except for the case $w = 4$ which is detailed in Subsection~\ref{ss:w=4}.
				\end{proof}

		\subsection{The subcase $w = 4$}\label{ss:w=4}
		
		Notice that
		\begin{equation}\label{eq:h=4}
			\emax(h, 4) = \begin{cases}
				4/3, & \mbox{if $h \equiv_3 1$;} \\
				1,   & \mbox{if $h \neq 3$ and $h \equiv_3 0$;} \\
				2/3, & \mbox{if $h \equiv_3 2$.}
			\end{cases}
		\end{equation}
		
		\begin{proposition}\label{prop4xhsc} 
			Let $h \geq 4$ be an integer. Then $e_{max}(h,4) = \emax(h,4)$.
			
		\end{proposition}
		
		The proof of Proposition~\ref{prop4xhsc} is combinatorial and requires the examination of several cases.
		We first introduce some lemmas.
		
		\begin{lemma}\label{lem:4xh}
			For any integer $h \geq 4$, $e_{max}(h, 4) \geq \emax(h, 4)$.
			Moreover, if $h >4$,  there exists a snake $S \in \Wordsdhwvarshort{2}{h}{4}$ such that $e(S) = \emax(h,4)$.
		\end{lemma}

		\begin{proof}
			Let
			$$A = \begin{wcells}{cccc}
				$\o$ & $\o$ & $\o$ & $\o$ \\
				$\o$ & $\e$ & $\e$ & $\o$ \\
				$\e$ & $\o$ & $\o$ & $\o$ \\
				$\o$ & $\o$ & $\e$ & $\e$
			\end{wcells}, \enskip
			B = \begin{wcells}{cccc}
				$\o$ & $\o$ & $\o$ & $\o$ \\
				$\o$ & $\e$ & $\e$ & $\o$ \\
				$\o$ & $\e$ & $\o$ & $\o$ \\
				$\e$ & $\o$ & $\o$ & $\e$
			\end{wcells}, \enskip
			C = \begin{wcells}{cccc}
				$\o$ & $\e$ & $\e$ & $\o$ \\
				$\o$ & $\o$ & $\o$ & $\o$
			\end{wcells}, \enskip
			D = \begin{wcells}{cccc}
				$\o$ & $\o$ & $\e$ & $\o$ \\
				$\o$ & $\e$ & $\e$ & $\o$ \\
				$\o$ & $\o$ & $\o$ & $\o$
			\end{wcells}, \enskip
			\tilde{D} = \begin{wcells}{cccc}
				$\o$ & $\e$ & $\o$ & $\o$ \\
				$\o$ & $\e$ & $\e$ & $\o$ \\
				$\o$ & $\o$ & $\o$ & $\o$
			\end{wcells}, \enskip
			E = \begin{wcells}{cccc}
				$\o$ & $\o$ & $\o$ & $\o$ \\
				$\o$ & $\e$ & $\e$ & $\o$
			\end{wcells}, $$
			$$U = \begin{wcells}{cccc}
				$\o$ & $\e$ & $\o$ & $\o$ \\
				$\o$ & $\e$ & $\o$ & $\e$ \\
				$\o$ & $\e$ & $\o$ & $\o$
			\end{wcells}, \enskip
			V = \begin{wcells}{cccc}
				$\o$ & $\o$ & $\e$ & $\o$ \\
				$\e$ & $\o$ & $\e$ & $\o$ \\
				$\o$ & $\o$ & $\e$ & $\o$
			\end{wcells}, \enskip
			X = \begin{wcells}{cccc}
				$\o$ & $\o$ & $\e$ & $\o$ \\
				$\o$ & $\e$ & $\o$ & $\o$ \\
				$\o$ & $\e$ & $\o$ & $\e$
			\end{wcells}, \enskip
			Y = \begin{wcells}{cccc}
				$\o$ & $\e$ & $\o$ & $\o$ \\
				$\o$ & $\o$ & $\e$ & $\o$ \\
				$\e$ & $\o$ & $\e$ & $\o$
			\end{wcells}, \enskip
			W_4 = \begin{wcells}{cccc}
				$\o$ & $\o$ & $\o$ & $\o$ \\
				$\o$ & $\e$ & $\e$ & $\o$ \\
				$\o$ & $\e$ & $\e$ & $\o$ \\
				$\o$ & $\o$ & $\o$ & $\o$
			\end{wcells},$$
			and for each integer $h \geq 5$, let $W_h$ be defined by
			$$W_h = \begin{cases}
				A \vcat (U \vcat V)^{(h - 6) / 6 \times 1} \vcat C,
				& \mbox{if $h \equiv_6 0$;} \\
				B \vcat (X \vcat Y)^{(h - 7) / 6 \times 1} \vcat D,
				& \mbox{if $h \equiv_6 1$;} \\
				E \vcat (U \vcat V)^{(h - 2) / 6 \times 1},
				& \mbox{if $h \equiv_6 2$;} \\
				A \vcat (U \vcat V)^{(h - 9) / 6 \times 1} \vcat U \vcat C,
				& \mbox{if $h \equiv_6 3$;} \\
				B \vcat (X \vcat Y)^{(h - 10) / 6 \times 1} \vcat X \vcat \tilde{D},
				& \mbox{if $h \equiv_6 4$;} \\
				E \vcat (U \vcat V)^{(h - 5) / 6 \times 1} \vcat U,
				& \mbox{if $h \equiv_6 5$.}
			\end{cases}$$
			
			It suffices to observe that, for each integer $h \geq 4$, $W_h \in \Wordsdhwvarshort{2}{h}{d}$ and $e(W_h) = \emax(h,4)$.
			Moreover, $W_h$ is a snake for $h \neq 4$.
			\qed
		\end{proof}
		
		To prove that $e_{max}(h, 4) \leq \emax(h, 4)$, we proceed by contradiction: We assume that there exists a word $W \in \Wordsdhwvarshort{2}{h}{4}$, where $h$ is as small as possible, such that $e(W) > \emax(h, 4)$, and show that $W$ cannot exist.
		Such a word $W$ is called a \emph{minimal counter-example} (MCE).

		\begin{lemma}
			Let $h \geq 1$ be an integer.
			Then the following statements hold.
			\begin{enumerate}[(i)]
				\item $\emax(h, 4) \leq 2$;
				\item If $h \neq 1$ and $h \equiv_3 1$, then $\emax(h, 4) - \emax(h - 2, 4) = 2/3$;
				\item If $h \neq 3$ and $h \equiv_3 0$, then $\emax(h, 4) - \emax(h - 4, 4) = 1/3$;
				\item If $h \neq 3$, then $\emax(h, 4) \leq 4/3$;
				\item If $h \geq 4$ and $k$ is an integer such that $1 \leq k \leq h - 1$ and $h - k \neq 3$, then $\emax(h, 4) - \emax(h - k, 4) \geq -2/3$.
			\end{enumerate}
		\end{lemma}
		
		\begin{proof}
			Follows from the definition of $\emax(h, 4)$.
			\qed
		\end{proof}
		
		\begin{lemma}\label{lem:proper}
			Let $W$ be a MCE of height $h \geq 4$, with $h \equiv_3 2$, and $U \in \Wordsdhwvarshort{2}{3}{4}$  an inner factor of $W$.
			Then
			\begin{enumerate}[(i)]
				\item $U \neq \begin{wcells}{cccc}
					$\o$ & $\o$ & $\o$ & $\o$ \\
					$\o$ & $\e$ & $\e$ & $\o$ \\
					$\o$ & $\o$ & $\o$ & $\o$
				\end{wcells}$.
				\item $|U|_\o \leq 8$.
			\end{enumerate}
		\end{lemma}
		
		\begin{proof}
			(i) Arguing by contradiction, assume the opposite.
			Then there exists a factor $U' \in \Wordsdhwvarshort{2}{5}{4}$ of $W$, having $U$ as an inner factor.
			Since all cells in the top and bottom rows of $U$ have degree $2$, then $|U'[1]|_\o = |U'[5]|_\o = 0$, so that $e(U') = -10/3$.
			Write $W = P \vcat U' \vcat S$.
			Then $e(W) = e(P) + e(U') + e(S) \leq 2 - 10/3 + 2 = 2/3 = \emax(h, 4)$, contradicting $e(W) > \emax(h, 4)$.
			
			(ii) We know from (i) that $|U|_\o \leq 9$.
			Again by contradiction, assume that $|U|_\o = 9$.
			Then $U$ belongs to the following set, up to symmetry:
			\[
			\left\{
			\begin{wcells}{cccc}
				$\e$ & $\o$ & $\o$ & $\o$ \\
				$\o$ & $\e$ & $\e$ & $\o$ \\
				$\o$ & $\o$ & $\o$ & $\o$
			\end{wcells}, \enskip
			\begin{wcells}{cccc}
				$\o$ & $\e$ & $\o$ & $\o$ \\
				$\o$ & $\e$ & $\e$ & $\o$ \\
				$\o$ & $\o$ & $\o$ & $\o$
			\end{wcells}, \enskip
			\begin{wcells}{cccc}
				$\o$ & $\o$ & $\o$ & $\o$ \\
				$\e$ & $\e$ & $\e$ & $\o$ \\
				$\o$ & $\o$ & $\o$ & $\o$
			\end{wcells}, \enskip
			\begin{wcells}{cccc}
				$\o$ & $\o$ & $\o$ & $\e$ \\
				$\o$ & $\e$ & $\o$ & $\o$ \\
				$\o$ & $\o$ & $\e$ & $\o$
			\end{wcells}
			\right\}
			\]
			Let $U'$ be a factor of height $4$ containing $U$, such that either $|U'[1]|_\o \leq 1$ or $|U'[4]|_\o \leq 1$.
			Such a factor exists since $U$ has at least one of its top or bottom row with at least $3$ cells of degree $2$.
			Therefore, $e(U') \leq -5/3$.
			Write $W = P \vcat U' \vcat S$, where $P$ has height $h'$.
			There are three subcases to consider according to the value of $h' \bmod 3$.
			If $h' \equiv_3 0$, then $e(W) = e(P) + e(U') + e(S) \leq 1 - 5/3 + 4/3 = 2/3 = \emax(h, 4)$, contradicting $e(W) > \emax(h, 4)$.
			If $h' \equiv_3 1$, then $e(W) = e(P) + e(U') + e(S) \leq 4/3 - 5/3 + 1 = 2/3 = \emax(h, 4)$, contradicting $e(W) > \emax(h, 4)$.
			Finally, if $h' \equiv_3 2$, then $e(W) = e(P) + e(U') + e(S) \leq 2/3 - 5/3 + 2/3 = -1/3$, also contradicting $e(W) > \emax(h, 4) = 2/3$.
			\qed
		\end{proof}

		\begin{lemma}\label{lem:nmce}
			There does not exist any MCE of height $h \geq 4$.
		\end{lemma}
		
		\begin{proof}
			By contradiction, assume that there exists a word $W$ of height $h \geq 4$ that is a MCE.
			There are three cases to consider, according to the value of $h \bmod 3$.
			
			Case $h \equiv_3 1$.
			Let $W = P \vcat S$, where $P$ has height $h - 2$ and $S$ has height $2$.
			Then $e(S) \leq 2/3$.
			Therefore $e(P) = e(W) - e(S) > \emax(h, 4) - 2/3 = \emax(h - 2, 4) + 2/3 - 2/3 = \emax(h - 2, 4)$, contradicting the minimality of $W$.
			
			Case $h \equiv_3 0$.
			Let $W = P \vcat S$, where $P$ has height $h - 4$ and $S$ has height $4$.
			First, assume that $e(S) \leq 1/3$.
			Then $e(P) = e(W) - e(S) > \emax(h, 4) - 1/3 > \emax(h - 4, 4) + 1/3 - 1/3 = \emax(h - 4, 4)$, contradicting the minimality of $W$.
			Hence, $e(S) \geq 4/3$.
			By exhaustive enumeration, this implies that $S$ belongs to the following set, up to symmetry:
			\[
			S \in \left\{
			\begin{wcells}{cccc}
				$\o$ & $\o$ & $\o$ & $\o$ \\
				$\o$ & $\e$ & $\e$ & $\o$ \\
				$\o$ & $\e$ & $\e$ & $\o$ \\
				$\o$ & $\o$ & $\o$ & $\o$
			\end{wcells}, \enskip
			\begin{wcells}{cccc}
				$\e$ & $\o$ & $\o$ & $\o$ \\
				$\o$ & $\o$ & $\e$ & $\o$ \\
				$\o$ & $\e$ & $\e$ & $\o$ \\
				$\o$ & $\o$ & $\o$ & $\o$
			\end{wcells}, \enskip
			\begin{wcells}{cccc}
				$\o$ & $\o$ & $\o$ & $\e$ \\
				$\o$ & $\e$ & $\o$ & $\o$ \\
				$\o$ & $\o$ & $\e$ & $\o$ \\
				$\e$ & $\o$ & $\o$ & $\o$
			\end{wcells}
			\right\}.
			\]
			Write $W = P' \vcat S'$, where $P'$ has height $h - 5$ and $S'$ has height $5$.
			Then $|S'[1]|_\o \leq 1$, since $S$ has at least $3$ cells of degree $2$ on each of its side.
			Therefore, $e(S') \leq -1/3$.
			But $e(P') = e(W) - e(S') > \emax(h, 4) + 1/3 = \emax(h - 5, 4) - 1/3 + 1/3 = \emax(h - 5, 4)$, contradicting the minimality of $W$.
			
			Case $h \equiv_3 2$.
			Write $W = P \vcat U_1 \vcat \cdots \vcat U_k \vcat S$, where both $P$ and $S$ have height $4$ and $U_i$ has height $3$ for $i = 1,2,\ldots,k$.
			Using an argument similar to the previous subcase, we have $e(P), e(S) \leq 1/3$.
			Moreover, by Lemma~\ref{lem:proper}, $e(U_i) = 0$ for $i = 1,2,\ldots,k$.
			Hence, $e(W) = 1/3 + 0 + \ldots + 0 + 1/3 = 2/3$, contradicting $e(W) > \emax(h, 4)$.
			\qed
		\end{proof}
		
		Proposition~\ref{prop4xhsc} follows from Lemmas~\ref{lem:4xh} and \ref{lem:nmce}.

	\subsection{The general case $w > 6$}
	
	\begin{lemma}
		\label{lem:w>6lowbound}
		Let $W \in \MWordsdhwvarshort{2}{h}{w}$ such that $h,w > 6$. Then
		\[
		e(W) \geq \emax(h,w).
		\]
	\end{lemma}
	
	\begin{proof} We  prove that for $i,j\leq 2$, $h=3k_1+j,w=3k_2+i$, there exist words $W \in \Wordsdhwvarshort{2}{h}{w}$ such that $e(W)=\emax(h,w)$. In a $h\times w$ rectangle  $R$, start by inserting a $Q$-shape of excess $1$ in the top left corner (purple cells in Figure \ref{7x7maxb}). Then concatenate   this $Q$-shape with the $Q$-shape from which the top left corner cell has been removed $k_1$ times horizontally and $k_2$ times vertically. Repeat this insertion in order to fill $R$ except for the bottom $i$ rows and the right $j$ columns. Then fill the remaining bottom $i$ rows with rows $R[3+i'],1\leq i'\leq i$ and the right $j$ columns with columns $R^t[3+j'],1\leq j'\leq j$. Then fill the bottom right $i\times j$ rectangle with $0,1$ or $3$ cells (purple cells) in bottom right of $R$ (Figures \ref{7x7maxb} to \ref{9x9maxb}). This produces a $h\times w$ word $W \in \Wordsdhwvarshort{2}{h}{w}$ with $e(W)=\emax(W)$.
		\begin{figure*}[h!]
			~ 
			\begin{subfigure}[t]{0.17\textwidth}
				\centering
				\includegraphics[height=0.42in]{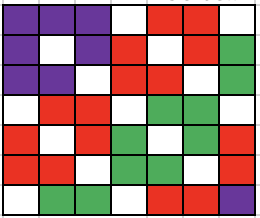}
				\caption{ $7\times 7$}
				\label{7x7maxb}
			\end{subfigure}%
			~ 
			\begin{subfigure}[t]{0.2\textwidth}
				\centering
				\includegraphics[height=0.48in]{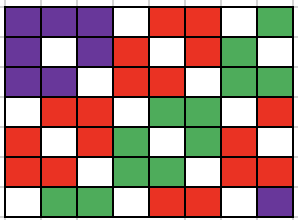}
				\caption{ $7\times 8$}
				\label{7x8maxb}
			\end{subfigure}%
			~ 
			\begin{subfigure}[t]{0.19\textwidth}
				\centering
				\includegraphics[height=.48in]{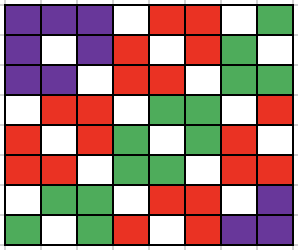}
				\caption{ $8\times 8$}
				\label{8x8maxb}
			\end{subfigure}%
			~ 
			\begin{subfigure}[t]{0.20\textwidth}
				\centering
				\includegraphics[height=0.54in]{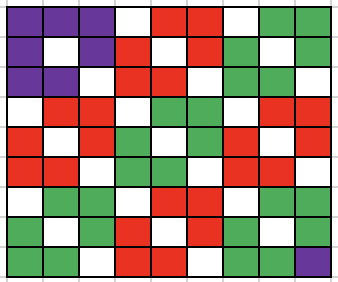}
				\caption{ $9\times 9$}
				\label{9x9maxb}
			\end{subfigure}%
			\begin{subfigure}[t]{0.21\textwidth}
				\centering
				\includegraphics[height=0.60in]{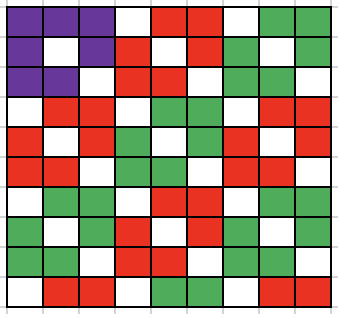}
				\caption{ $10\times 9$}
				\label{10x9max}
			\end{subfigure}%
			\caption{Some 2-full words}
			\label{hxwmax}
		\end{figure*}
	\end{proof}
	
	The following lemma completes the proof of Theorem \ref{thm:d=2}.
	It is separated in 6 cases.
	Due to space restriction, we only provide the complete proofs of the first 5 cases, and leave the sixth, and more technical, cases in the Appendix.
	Before that, we need additional results.
	
	\begin{corollary}\label{cor4xh_n2(c1)}
		Let $k\geq 2$ and $W\in \Wordsdhwvarshort{2}{(3k+1)}{5}$ such that $W^t\DI{1,4}\in\MWordsdhwvarshort{2}{3k+1}{4}$. Then $e(W^t[5])\leq -5/3$.
	\end{corollary}
	
	\begin{corollary}\label{cor:5xhsature}
		Let $k\geq1$ and $W\in\Wordsdhwvarshort{2}{(3k+2)}{6}$ such that $W^t\DI{1,5}\in\MWordsdhwvarshort{2}{(3k+2)}{5}$. Then $e(W^t[6]) \leq -4/3$.
	\end{corollary}
	
	\begin{corollary}\label{cor4xhn(c4)}
		Let $W\in \Wordsdhwvarshort{2}{h}{4}$ such that $W^t\DI{1,3}\in \MWordsdhwvarshort{2}{h}{3}$. Then
		\begin{enumerate}[(i)]
			\item $|W^t[4]|_\e \geq\lfloor (h-2)/3\rfloor + 3$;
			\item If $h \equiv_3 1$, then $|W^t[4]|_\e \geq \lfloor (h-2)/3 \rfloor + 4$;
			\item If $h \equiv_3 2$, then $|W^t[4]|_\e \geq\lfloor (h-2)/3 \rfloor + 5$.
		\end{enumerate}
	\end{corollary}
	
	\begin{lemma}
		\label{lem:w>6upbound}
		Let $W \in \MWordsdhwvarshort{2}{h}{w}$ such that $h,w > 6$. Then
		\[
		e(W) \leq \emax(h,w).
		\]
	\end{lemma}
	
	The proof of these results are included in the Appendix.
	Moreover, by combining Lemma \ref{lem:basecases}, Lemma \ref{lem:w>6lowbound} and Lemma \ref{lem:w>6upbound}, we have a proof of Theorem \ref{thm:d=2} and, by extension, Theorem \ref{thm:main}.
	
	\begin{proof}
		We prove that there exists no $W \in \Wordsdhwvarshort{2}{h}{w}$ such that $e(W) > \emax(h,w)$. 
		There are 6 cases up to symmetry and we prove  each of them next.
		
		\subsubsection{Case 1.}
	
		Let $W \in \Wordsdhwvarshort{2}{3k_1}{3k_2}$ with $k_1,k_2\geq3$. By minimal counterexample, assume that $W$ is minimal such that $e(W)>\emax(h,w)$. We have that $\emax(3k_1,3k_2)=2$ so we assume $e(W) = 3$.

		Write $W=T_1 \vcat T_2$ where each factor $T_1,T_2$ has respective height $3k_{1,1}+1, 3k_{1,2}+2$ for some $k_{1,1}\geq 1,k_{1,2}\geq 1$.
		By minimality hypothesis, we have $e(T_i) \leq \emax(3k_{1,1}+1,3k_2) = \emax(3k_{1,2}+2,3k_2) = 1$, so that $e(W) = e(T_1) + e(T_2) \leq 1 + 1 = 2$, in contradiction with the hypothesis $e(W)=3$.
		
		\subsubsection{Case 2.}
		Let $W \in \Wordsdhwvarshort{2}{(3k_1+2)}{(3k_2+2)}$ with $k_1,k_2\geq2$. By minimal counterexample, assume that $W$ is minimal such that $e(W)>\emax(h,w)$. We have that $\emax(3k_1+2,3k_2+2)=4/3$ so we assume $e(W) = 7/3$.

		Write $W=T_1\vcat T_2$, where each factor $T_1,T_2$ has respective height $3k_{1,1}+1, 3k_{1,2}+1$ for some $k_{1,1},k_{1,2}\geq1$.
		We have by minimality hypothesis $e(T_i) \leq \emax(3k_{1,1}+1,3k_2+2) = \emax(3k_{1,2}+1,3k_2+2) = 2/3$, which implies $e(W) = e(T_1)+e(T_2) \leq 2/3 + 2/3 = 4/3$, in contradiction with the hypothesis $e(W)=7/3$.
		
		\subsubsection{Case 3.}
		Let $W \in \Wordsdhwvarshort{2}{(3k_1+1)}{(3k_2+1)}$ with $k_1,k_2\geq2$. By minimal counterexample, assume that $W$ is minimal such that $e(W)>\emax(h,w)$. We have that $\emax(3k_1+1,3k_2+1)=4/3$.
		Therefore, we may assume $e(W) = 7/3$.
		
		The case  $k_1=k_2=2$ is proved as one of  the base cases in Lemma \ref{lem:basecases} ($\Wordsdhwvarshort{2}{7}{7}$).
		
		Assume $k_1\geq3$ and write $W=T_1\vcat T_2$, where the factors $T_1,T_2$ have respective heights $3k_{1,1}+2, 3k_{1,2}+2$ for some $k_{1,1},k_{1,2}\geq1$.
		By minimality hypothesis $e(T_i) \leq \emax(3k_{1,1}+2,3k_2+1) = \emax(3k_{1,2}+2,3k_2+1) = 2/3$, so that $e(W) =e(T_1)+e(T_2) \leq 2/3+2/3 = 4/3$, in contradiction with the hypothesis $e(W)=7/3$.
		
		\subsubsection{Case 4.}
		Let $W \in \Wordsdhwvarshort{2}{(3k_1+1)}{3k_2}$ with $k_1\geq2,k_2\geq3$.
		By minimal counterexample, assume that $W$ is minimal such that $e(W)>\emax(h,w)$. We have that $\emax(3k_1+1,3k_2)=1$ so we assume $e(W) = 2$.

		Now write $W = W^t\DI{1,4}\vcat W^t\DI{5,3k_2}$.
		By minimality hypothesis, we have $e(W^t\DI{1,4})\leq 4/3$ and $e(W^t\DI{5,3k_2})\leq 2/3$.
		Therefore, $e(W^t\DI{1,4})=4/3$ and $e(W^t\DI{5,3k_2})= 2/3$.
		But, by Corollary \ref{cor4xh_n2(c1)}, $e(W^t\DI{1,4})=4/3$ implies $e(W^t[5])\leq -5/3$.
		Thus $e(W^t\DI{1,5})\leq-1/3$, so that $e(W^t\DI{6,3k_2})\geq7/3$, in contradiction with the minimality hypothesis.

		\subsubsection{Case 5.}
		Let $W \in \Wordsdhwvarshort{2}{3k_1}{(3k_2+2)}$ with $k_1\geq3,k_2\geq2$.
		By minimal counterexample, assume that $W$ is minimal such that $e(W)>\emax(h,w)$.
		We have that $\emax(3k_1,3k_2+2)=1$ so we assume $e(W) = 2$.
		Now let $W=W\DI{1,5}\vcat W\DI{6,3k_1}$.
		By minimality hypothesis  we have 
		$e(W\DI{1,5}) \leq  4/3$ and $e(W\DI{6,3k_1}) \leq \emax(3(k_1-2)+1,3k_2+2) = 2/3$.
		This implies $e(W) =e(W\DI{1,5})+e(W\DI{6,3k_1}) \leq 4/3+2/3 = 2$.
		Hence, both $W\DI{1,5}$ and $W\DI{6,3k_1}$ are 2-full.
		Therefore, since $W\DI{1,5}$ is 2-full, by Corollary \ref{cor:5xhsature}, we deduce that $e(W[6]) \leq -4/3$.
		Moreover, since $e(W\DI{6,3k_1}) = 2/3$, we must have $e(W\DI{7,3k_1}) \geq 2$.
		If $k_1>3$, then $\emax(3(k_1-2),3k_2+2) = 1$, which means $e(W\DI{7,3k_1}) \geq 2 > 1 = \emax(3(k_1-2),3k_2+2)$, contradicting the minimality hypothesis.
		If $k_1 = 3$, then $W\DI{7,3k_1} = W\DI{7,9}$ and must be 2-full with $e(W\DI{7,9}) = 2$ which then forces $e(W[6]) = -4/3$.
		However, Corollary \ref{cor4xhn(c4)} $iii)$ stipulates that, if $W\DI{7,9}$ is 2-full, then $W[6]$ contains at least $\lfloor \frac{(3k_2+2)-2}{3}\rfloor + 5 = k_2 + 5$ empty cells.
		This means $e(W[6])\leq -13/3$, which is a contradiction. 
		
	\subsubsection{Case 6.}
	Let $W \in \Wordsdhwvarshort{2}{3k_1+2}{3k_2+1}$ with $k_1,k_2\geq2$. By minimal counterexample, assume that $W$ is minimal such that $e(W)>\emax(h,w)$. We have that $\emax(3k_1+2,3k_2+1)=2/3$ so we assume $e(W) = 5/3$. Partition $W$ as $W\DI{1,4}\vcat W\DI{5,3k_1+2}$.
	We have by minimality of $W$ that
	$$e(W\DI{1,4}),e(W\DI{5,3k_1+2}) \leq 4/3.$$
	There are two cases to consider : either $e(W\DI{1,4}) = 4/3$ and $e(W\DI{5,3k_1+2}) = 1/3$ or $e(W\DI{1,4}) = 1/3$ and $e(W\DI{5,3k_1+2}) = 4/3$.
	The proof that these two cases lead to contradictions is given in the Appendix.

\end{proof}

\paragraph{\bf 4. Conclusion}\label{conclusion}

The results presented in this paper are opening  to several questions. One of these is  in the following. 
\begin{conjecture}
For $h>w\geq 5$ and $h\equiv_3 w$, 2-full words  $W\in \MWordsdhwvarshort{2}{h}{w}$ are unique up to symmetries of the plane plus a local symmetry (a flip) on hook snakes.
\end{conjecture}
Thanks to Theorem \ref{thm:main}, the  length of maximal snakes in a $h\times w$ rectangle can be given upper and lower bounds but the question of finding their exact length in terms of $h$ and $w$ is open.
Theorem  \ref{thm:main} raises analogous  questions   for polycubes : is there a maximal ratio analogous to $2/3$ for sets of cubic cells of bounded degree  in a given rectangular parallelepiped. This question  is open to our knowledge. 
Theorem \ref{thm:main}  also lead to the question of enumerating $d$-full words with respect to width $w$.  
More interesting questions rise when languages with more than two letters  are used as  labels of the unit cells in words $W$.
\bibliographystyle{plain} 
\bibliography{MaximalWordsOfBoundedDegree}

\newpage
\section{Appendix}\label{appendix}
In this appendix, we provide the proofs omitted in the main paper.
Most of them are quite long and technical.
Some even introduce notions only relevant in the scope of themselves.
The appendix is separated into 4 sections.
In the first section, we give complete proofs of Theorem \ref{thm:main} for $d \in \{1,4\}$.
In the second section, we prove Lemma \ref{lem:basecases} for the remaining base cases.
In the third section, we extract useful properties of 2-full words in $\Wordsdhwvarshort{2}{h}{w}$ for $w=3,4,5$ to prove Corollaries \ref{cor4xh_n2(c1)}, \ref{cor:5xhsature}, \ref{cor4xhn(c4)}. In the final section, we provide the rest of the complete proof of \textbf{case 6} of Lemma \ref{lem:w>6upbound}.

\subsection{Proof of Theorem \ref{thm:main} for $d\in \{1,4\}$.}

The case $d = 4$ is straightforward.

\begin{lemma}\label{lem:d=4}
For any $(h,w) \in \mathbb{Z}_{>0}^2$, $max_4(h, w) = hw$.
\end{lemma}

\begin{proof}
On one hand, let $W \in \Wordshwshort$ be such that $W[i,j] = 1$ for all $(i,j) \in \DIhw$.
Clearly, $|W|_\o = hw$ and $W \in \Wordsfhwshort$, so that $\maxdhw{4}{h}{w} \geq hw$.
On the other hand, $\maxdhw{4}{h}{w} \leq hw$ since there are $hw$ letters in any $2$-dimensional word of dimensions $h \times w$.
Hence, $\maxdhw{4}{h}{w} = hw$.
\end{proof}

The case $d = 1$ is slightly more complicated, but can be proved with elementary combinatorial arguments.

\begin{lemma}\label{lem:d=1}
For any $(h,w) \in \DIhw$, where $h \geq w$,
\[\maxdhw{1}{h}{w} = \begin{cases}
	hw/2, & \mbox{if $(h, w) \equiv_2 (0,0)$;} \\
	(h - 1)w/2 + \lceil 2w / 3 \rceil,
	& \mbox{if $(h, w) \equiv_2 (1,0)$;} \\
	h(w - 1)/2 + \lceil 2h / 3 \rceil,
	& \mbox{otherwise.}
\end{cases}\]
\end{lemma}

\begin{figure}
\centering
\input{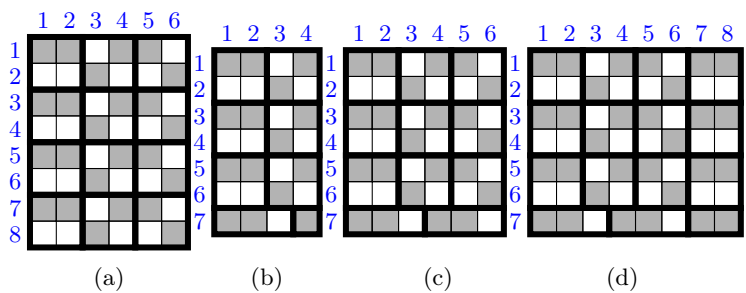}
\caption{The word $W$ defined in the proof of Lemma~\ref{lem:d=1}, superimposed with some tilings using as many as possible $2 \times 2$ tiles, some $1 \times 3$ tiles and at most one $1 \times (h \bmod 3)$ tile. (a) The case $8 \times 6$ (b) The case $7 \times 4$. (c) The case $7 \times 6$. (d) The case $7 \times 8$.}\label{fig:d=1}
\end{figure}

\begin{proof}
There are four cases to consider: (i) $(h, w) \equiv_2 (0,0)$, (ii) $(h, w) \equiv_2 (1,0)$, (iii) $(h, w) \equiv_2 (0,1)$ and (iv) $(h, w) \equiv_2 (1,1)$.

(i) Let $W \in \Wordshwshort$ be defined by
\begin{equation}\label{eq:w-def}
	W = \left(\substack{\o\o\e\o\o\e \\ \e\e\o\e\e\o}\right)^{h/2 \times w/6}.
\end{equation}
The word $W$ is illustrated in Figure~\ref{fig:d=1}(a) when $(h,w) = (8,6)$.
On one hand $|W|_{\o} = hw/2$ and $\deg(W[i,j]) \leq 1$ for each $(i,j) \in \DIhw$.
Therefore,
$\maxdhw{1}{h}{w} \geq hw/2$.
On the other hand, let $W \in \Wordsohwshort$.
Then $W$ can be partitioned into $hw/4$ tiles of dimensions $2 \times 2$ as in Figure~\ref{fig:d=1}(a).
Since each tile may contain at most two  $\o$-cells, we have $|W|_{\o} \leq 2 \cdot hw / 4 = hw/2$, which implies $\maxdhw{1}{h}{w} \leq hw/2$.
Hence, $\maxdhw{1}{h}{w} = hw/2$.

(ii) Let $W \in \Wordshwshort$ be defined by Equation~\eqref{eq:w-def}.
The word $W$ is illustrated in Figure~\ref{fig:d=1}(b-d) when $(h,w) \in \{(7,4), (7,6), (7,8)\}$.
Then $|W|_{\o} = (h - 1)w/2 + \lceil 2h/3 \rceil$.
Moreover, $\deg(W[i,j]) \leq 1$ for each $(i,j) \in \DIhw$, which implies $W \in \Wordsohwshort$ and shows that $\maxdhw{1}{h}{w} \geq (h - 1)w/2 + \lceil 2h/3 \rceil$.
Now, assume that $W \in \Wordsohwshort$.
Then $W$ can be partitioned into $h(w - 1)/4$ tiles of dimensions $2 \times 2$, $\lfloor h / 3 \rfloor$ tiles of dimensions $1 \times 3$ and $\indic(h \bmod 3 \neq 0) \in \{0,1\}$ tile of dimension $1 \times (h \bmod 3)$, as illustrated in Figure~\ref{fig:d=1}(b-d), which are respectively occupied with at most $2$, $2$ and $h \bmod 3$  $\o$-cells .
This implies
\[|W|_{\o} \leq 2 \cdot h(w - 1)/4 + 2 \cdot \lfloor h / 3 \rfloor + h \bmod 3 = (h - 1)w/2 + \lceil 2h/3 \rceil,\]
so that $\maxdhw{1}{h}{w} \leq (h - 1)w/2 + \lceil 2h/3 \rceil$.
Hence, $\maxdhw{1}{h}{w} = (h - 1)w/2 + \lceil 2h/3 \rceil$.

(iii) Follows from the fact that $\maxdhw{1}{h}{w} = \maxdhw{1}{w}{h}$.

(iv) Let $W$ be defined by Equation~\eqref{eq:w-def}.
Then $|W|_\o = h(w - 1)/2 + \lfloor 2h/3 \rfloor$, which implies $\maxdhw{1}{h}{w} \geq h(w - 1)/2 + \lfloor 2h/3 \rfloor$.
It remains to prove that $\maxdhw{1}{h}{w} \leq h(w - 1)/2 + \lfloor 2h/3 \rfloor$.
Arguing by contradiction, assume that there exists a word $W$ of minimal dimensions such that $|W|_\o > \maxdhw{1}{h}{w}$.
Write $W = L \hcat R$, where $R$ has width $2$.
Then, $|R|_\o \leq 2(h - 1)/2 + 2 = h + 1$: It suffices to partition $R$ into $(h - 1) / 2$ tiles of dimensions $2 \times 2$ and $1$ tile of dimension $1 \times 2$, noting that at most $2$ cells can fill each tile.
But if $|R|_\o \leq h$, then $|L|_\o = |W|_\o - |R|_\o > \maxdhw{1}{h}{w} - h = \maxdhw{1}{h}{w - 2}$, contradicting the minimality of $W$.
Hence, $|R|_\o = h + 1$.
Now, write $W = U \vcat D$, where $D$ has width $2$.
A similar argument leads to $|D|_\o = w + 1$.
This implies that the lower right corner of $W$ of dimensions $2 \times 2$ has $3$ filled cells, i.e. one of the cell has degree at least $3$, which is impossible.
\end{proof}

\subsection{The remaining base cases of Lemma \ref{lem:basecases}}

\subsubsection{\textbf{$w=2.$}}

\begin{proposition}
\label{prop:w=2}
\begin{align*}
	e_{max}(h,2) = \begin{cases}
		h/6 + 1/2 &\mbox{ if $h$ is odd},\\
		h/6 &\mbox{ if $h$ is even},
	\end{cases}
\end{align*}
or equivalently
\begin{align*}
	\maxdhw{2}{h}{2}=
	\begin{cases}
		\frac {3\cdot 2h+2} 4 &\mbox{ if $h$ is odd},\\
		\frac {3\cdot2h} 4 &\mbox{ if $h$ is even}.
	\end{cases}
\end{align*}
\end{proposition}
\begin{proof}
Observe that it is easy to construct words $W\in \Wordsdhwvarshort{2}{h}{2}$  such that $|W|_{\o}=\lceil \frac{3\cdot2h}{4}\rceil$ as shown in Figure  \ref{2xhred} so that $\maxdhw{2}{h}{2}\geq \lceil \frac{3\cdot2h}{4}\rceil$.
Assume $h$ is minimal such that there exists $W\in\mathcal{W}^{2}_{h \times 2}$ with $|W|_{\o}=\lceil \frac{3\cdot2h}{4}\rceil+1$. Then by the pigeon hole principle there is a filled $2\times 2$ factor $W\DI{i,i+1}$ of $W$. If  $W\DI{i,i+1}$ is in the interior of $W$ then $W[i-1],W[i+2]$ are empty. Factor $W$ as the vertical concatenation of the three factors $W=W_T\vcat W\DI{i-1,i+2} \vcat W_B$ where $W_T$ is the top and $W_B$ the bottom factor. Then
\begin{align}\label{eqTB}
	|W_T|_\o + |W_B|_\o= \lceil \frac {3\cdot2h} 4\rceil +1-4=\lceil \frac {3\cdot2(h-4)} 4\rceil+3
\end{align}
But by the minimality hypothesis 
\begin{align*}
	|W_T|_\o\leq \lceil \frac {3\cdot2h(W_T)} 4\rceil  \mbox{ and }|W_B|_\o\leq \lceil \frac {3\cdot2h(W_B)} 4\rceil 
	\Rightarrow |W_T|_\o +|W_B|_\o < \lceil \frac {3\cdot2(h-4)} 4\rceil+3
\end{align*}
in contradiction with Equation \eqref{eqTB}.
\begin{figure*}[h!]
	\centering
	\begin{subfigure}[t]{0.31\textwidth}
		\centering
		\includegraphics[height=0.6in]{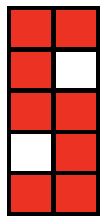}
		\caption{$ 5\times 2$ word, \\ $8$ cells}\label{2x5red}
	\end{subfigure}%

	\begin{subfigure}[t]{0.31\textwidth}
		\centering
		\includegraphics[height=0.8in]{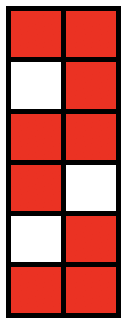}
		\caption{$ 6\times 2$ word, \\ $9$ cells}\label{2x6fs}
	\end{subfigure}%
	~ 
	\caption{ $h\times 2$ words}
	\label{2xhred}
\end{figure*}
So $W\DI{i,i+1}$ is not in the interior of $W$. If $i=1$ and  $W\DI{i,i+1}$ is at the top of $W$ then $W[i+2]$ is empty and using minimality hypothesis again we have 
\begin{align*}
	|W\DI{i+3,h}|_\o\leq \lceil \frac {3\cdot2(h-3)} 4\rceil < \lceil \frac {3\cdot2(h)} 4\rceil+1-4
\end{align*}
in contradiction with the hypothesis $|W|_{\o}=\lceil \frac{3\cdot2h}{4}\rceil+1$.
So  $W\DI{i,i+1}$ cannot be at the top of $W$ and $W$ cannot exist.
\end{proof}
\subsubsection{\textbf{$w=3.$}}

\begin{proposition}
\label{prop3xhfssize}
For  integers $h\geq 3$, the area $\maxdhw{2}{h}{3}$ of 2-full words  $W\in\Wordsdhwvarshort{2}{h}{3}$  is equal to the length $\maxshw{h}{3}$ of maximal snakes in $h\times 3$ rectangles and is equal to : 
\begin{align}
\label{prop3xhellfs}
\maxdhw{2}{h}{3}&=\maxshw{h}{3}=  2h +2
\end{align}
Equivalently, the excess of 2-full words $W\in\MWordsdhwvarshort{2}{h}{3}$  is $e_{max}(h,3)=2$.
\end{proposition}

\begin{proof}
Let $W\in \Wordsdhwvarshort{2}{h}{3}$.
First observe that words and snakes of length ${2h} +2$ do exist as shown in Figure \ref{hx3words}.  In Figures  \ref{3x3} and \ref{3x6s}, a 2-full word and a maximal snake  are shown and can be extended to a word and a snake of any height $h\geq 3$ and area $2h+2$ by inserting a row of two cells between  rows $W[1]$ and $W[2]$. Figure \ref{3x13fs} shows a maximal forest of snakes.
\begin{figure*}[h!]
\centering
\begin{subfigure}[t]{.3\textwidth}
	\centering
	\includegraphics[height=.25in]{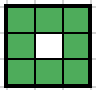}
	\caption{cycle}\label{3x3}
\end{subfigure}%
~
\begin{subfigure}[t]{0.3\textwidth}
	\centering
	\includegraphics[height=0.5in]{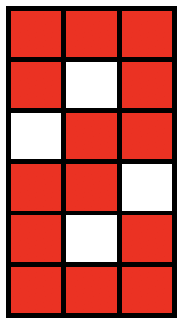}
	\caption{$6\times 3$ snake}\label{3x6s}
\end{subfigure}%
~
\begin{subfigure}[t]{0.3\textwidth}
	\centering
	\includegraphics[height=1.1in]{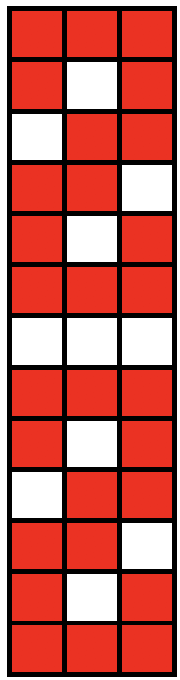}
	\caption{Forest of snakes}\label{3x13fs}
\end{subfigure}%
\caption{2-full $h\times 3$ words}
\label{hx3words}
\end{figure*}
We prove equation \eqref{prop3xhellfs} by minimal counter-example.
Assume that $h$ is minimal such that there exists a $h\times 3$ word $W$ of excess $e({\displaystyle  W})\geq 3$.

Consider the factorization $W=W\DI{1,i-1} \vcat W[i] \vcat W\DI{i+1,h}$ for all $ 1< i < h$.
First observe that $e(W[1])=1$ or else $e(W\DI{2,h})\geq3$ which contradicts the minimality of $h$.
We claim that $e(W[i])=0$ for $2\leq i\leq h-1$.

Assume that this is not the case and  let $W[i]$ be the first row from the top of $W$ such that $e(W[i])\neq0$, we must then have $e(W[i])\in\{-2,-1,1\}$.
If $e(W[i])\in\{-1,-2\}$, then the minimality of $i$ implies $e(W) = e(W\DI{1,i-1})+e(W[i])+ e(W\DI{i+1,h})$.

But $W[1]=(\o,\o,\o)$ and $e(W[j])=0$ for all $j$, $1 < j < i$, so that $3 = 1+e(W[i])+ e(W\DI{i+1,h})$.
Threfore $3 \leq  e(W[i+1,h])$, by assumptions on $W$ and $W[i]$, which contradicts the minimality of $h$.

If $e(W[i])=1$, then $W[i]=(\o,\o,\o)$ and by minimality of $i$, $e(W[j])=0$ for all $j$ such that $1<j<i$.
We must then have $W[i-1]=(\o,\e,\o)$ which implies $W[i+1]=(\e,\e,\e)$ because every cell in $W[i]$ has degree $2$. This implies  that $e(W\DI{i+2,h})=e(W)-e(W\DI{1,i+1})=e(W)\geq3$ which contradicts the minimality of $h$. This contradiction shows that for all $1<i<h$ we have $e(W[i])=0$ so that the only rows of excess $1$ are $W[1]$ and $W[h]$ concluding the proof.
\end{proof}

\subsubsection{\textbf{$w=5.$}}

Lemma \ref{lem:basecases} states that $e_{max}(h,5) = \emax(h,5)$ for all values of $h\geq 5$. We first prove the following lemma which says that, for all values of $h\geq5$, there exists a word in $\Wordsdhwvarshort{2}{h}{5}$ of excess $\emax(h,5)$.

\begin{lemma}
\label{lem:w=5lowbound}
$e_{max}(h,5) \geq \emax(h,5)\quad \forall h\geq5$.
\end{lemma}

\begin{proof}
We look at each congruence of $h \mod 3$ separately.

Let $h = 3k+2,k\geq1$. Then
\[
\emax(h,5) = \emax(3k+2,5)=4/3.
\]
The cases $k=1,2$ can be verified manually. For $k\geq3$, a word $W' \in \Wordsdhwvarshort{2}{(3k+2)}{5}$ such that $e(W') = 4/3$ does exist as shown in Figure \ref{subfig:5x8} where we see that the red $3 \times 5$ factor can be concatenated to itself to produce words $W'$ of excess $e(W') = 4/3$ of arbitrary height $3k + 2$.

\begin{figure}
\centering
\begin{subfigure}[t]{0.23\textwidth}\centering
	\input{tikz/5x5.tikz}
	\caption{The unique  $W \in \MWordsdhwvarshort{2}{5}{5}$ with 18 cells.}
	\label{subfig:5x5}
\end{subfigure}
\begin{subfigure}[t]{0.23\textwidth}\centering
	\input{tikz/5x8.tikz}
	\caption{The unique  $W \in \MWordsdhwvarshort{2}{8}{5}$ with 28 cells.}
	\label{subfig:5x8}
\end{subfigure}
\begin{subfigure}[t]{0.23\textwidth}\centering
	\input{tikz/5x6.tikz}
	\caption{A word $W \in \MWordsdhwvarshort{2}{6}{5}$. It contains 21 cells}
	\label{subfig:5x6}
\end{subfigure}
\begin{subfigure}[t]{0.23\textwidth}\centering
	\input{tikz/5x7.tikz}
	\caption{A word $W \in \MWordsdhwvarshort{2}{7}{5}$. It contains 24 cells}
	\label{subfig:5x7}
\end{subfigure}
\caption{}
\label{fig:5x3k+2}
\end{figure}

Now let $h = 3k,k\geq2$. Then
\[
\emax(h,5) = \emax(3k,5)=1.
\]
The case $k=2$ can be verified manually. For $k\geq3$, a word $W \in \Wordsdhwvarshort{2}{(3k)}{5}$ such that $e(W) = 1$   is shown in Figure \ref{subfig:5x6} where  the red $3 \times 5$ factor of excess $0$ can be concatenated to itself to produce words $W$ of excess $e(W) = 1$ of arbitrary height $3k$.

Finally, let $h = 3k+1,k\geq2$. Then
\[
\emax(h,5) = \emax(3k+1,5) = 2/3.
\]
The case $k=2$ can be verified manually. For $k\geq3$, a word $W \in \Wordsdhwvarshort{2}{(3k+1)}{5}$ such that $e(W) = 2/3$ does exist as shown in Figure \ref{subfig:5x7} where we see that the red $3 \times 5$ factor can be concatenated to itself to produce words $W$ of excess $e(W) = 2/3$ of arbitrary height $3k+1$.

For each congruence of $h \mod 3$, we can construct a word in $\Wordsdhwvarshort{2}{h}{5}$ with excess $\emax(h,5)$. This means that for all $h\geq5$, $e_{max}(h,5) \geq \emax(h,5)$.
\end{proof}

We now need to prove that there exists no words in $\Wordsdhwvarshort{2}{h}{5}$ with excess greater than $\emax(h,5)$.

We will need the following results.

\begin{proposition}
\label{prop:5xhmin2factor}
Let $W \in \Wordsdhwvarshort{2}{(3k+2)}{5}$ $k\geq1$ with $e(W)\geq4/3$ be such that no $5\times h'$ factor $h'<3k+2$ has excess greater than $\emax(h',5)$. $W$ cannot contain a $5 \times 3$ factor of excess $-2$.
\end{proposition}

\begin{proof}
Let $T$ be a $5\times 3$ factor of $W$ of excess $-2$. We can partition $W$ as $W = W_{top} \vcat T \vcat W_{bottom}$ which yields
\[
e(W_{top}\vcat T \vcat W_{bottom}) \geq 4/3.
\]
This implies $e(W_{top}) + e(W_{bottom}) \geq 4/3+2 = 10/3$. If $W_{top}$ is the empty word, then $e(W_{bottom}) = 10/3$ which contradicts the minimality hypothesis. Similarly, $W_{bottom}$ cannot be empty. Now since $h = 3k+2$ and $T$ has height 3, either both $W_{top}$ and $W_{bottom}$ have height congruent to $1 \mod 3$ or one has height congruent to $2 \mod 3$ and the other has height congruent to $0 \mod 3$. If both $W_{top}$ and $W_{bottom}$ have height congruent to $1 \mod 3$, then by minimality of $h$, we have $e(W_{top}),e(W_{bottom})\leq 5/3$ which implies $e(W_{top})+e(W_{bottom}) \leq 10/3$. Note however that $e(W_i) = 5/3$ is only possible when $W_i$ is a single row. So, in this case, $e(W_{top})+e(W_{bottom}) = 10/3$ implies both $W_{top}$ and $W_{bottom}$ are a single row which yields $h=5$. But Figure \ref{subfig:5x5} shows the only possible word in $\Wordsdhwvar{2}{5}{5}$ with excess greater or equal to $4/3$ which does not contain a $5 \times 3$ factor of excess $-2$. This situation is ruled out. Now assume, without loss of generality, that $W_{top}$ has height congruent to $2 \mod 3$ and that $W_{bottom}$ has height congruent to $0 \mod 3$. From the minimality of $h$, we have that $e(W_{top}) \leq 4/3$ and $e(W_{bottom}) \leq 2$. This yields $e(W_{top}) + e(W_{bottom}) \leq 10/3$. Note, however, that $e(W_{bottom}) = 2$ is only possible when $W_{bottom}$ is in $\MWordsdhwvarshort{2}{3}{5}$ which consists of a single word illustrated in Figure \ref{fig:5x3}. Observe that the top row of $W_{bottom}$ must be a full row of degree 2 cells. This forces the bottom row of $T$ to be empty which implies the top two rows of $T$ to be a 2-full $5 \times 2$ factor. This force, in turn, the top row of $T$ to contain at least three degree 2 cells which implies that the bottom row of $W_{top}$ is of excess at most $-4/3$. So $W_{top}$ is a factor of excess $4/3$ with a bottom row of excess at most $-4/3$. This means the factor $W_{top}$ minus it's bottom row is a factor of excess at least $8/3$. This contradicts the minimality of $h$. This last situation is thus ruled out meaning the assumption that $W$ can have a $5 \times 3$ factor of excess $-2$ is rejected.
\end{proof}

\begin{figure}
\centering
\input{tikz/5x3.tikz}
\caption{The unique word in $\MWordsdhwvarshort{2}{3}{5}$}
\label{fig:5x3}
\end{figure}

\begin{lemma}
\label{lem:5x3k+2upbound}
Let $W \in \Wordsdhwvarshort{2}{(3k+2)}{5}$ $k\geq1$ be such that no $h'\times 5$ factor,
$h'<3k+2$, has excess greater than $\emax(h',5)$. Then $e(W) \leq \emax(3k+2,5)$.
\end{lemma}
\begin{proof}
Assume by contradiction that there exists $W \in \Wordsdhwvarshort{2}{3k+2}{5}$ with no factor of height $h'<h$ of excess greater than $\emax(h',5)$ but with $e(W) = \emax(3k+2,5)+1 = 7/3$.

We know $e(W\DI{1,5}) \leq 4/3$. Assume $e(W\DI{1,5}) = 4/3$. Then, as can be observed in Figure \ref{subfig:5x5}, the bottom row of $W\DI{1,5}$ contains 3 cells of degree 2. This implies $|W[6]|_\o\leq 2$  and $e(W[6]) \leq -4/3$. 

The excess of the top 6 rows of $W$ is then $e(W\DI{1,6}) \leq 0$. This implies that $e(W\DI{7,3k+2}) \geq 7/3$ in contradiction with the minimality of $h$. We thus have $e(W\DI{1,5}) \leq 1/3$.

We know from Proposition \ref{prop3xhfssize} that $e(W\DI{3k,3k+2}) \leq 2$. Assume $e(W\DI{3k,3k+2}) = 2$. Then, as can be observed in Figure \ref{fig:5x3}, the top row of $W\DI{3k,3k+2}$ contains 5 cells of degree $2$. 
This implies that $e(W\DI{3k-3,3k-1}) \leq -2$ which from Proposition \ref{prop:5xhmin2factor} is impossible. We thus have $e(W\DI{3k,3k+2})\leq1$.

Since $e(W\DI{1,5}) \leq 1/3$ and $e(W\DI{3k,3k+2}) \leq 1$, we must have that at least one of $W\DI{6,8},\dots,W\DI{3k-3,3k-1}$ has excess greater or equal to 1.
Figure \ref{fig:5x3probleme} shows all possible words in $\Wordsdhwvarshort{2}{3}{5}$ with excess greater or equal to 1 up to symmetry.

\begin{figure}
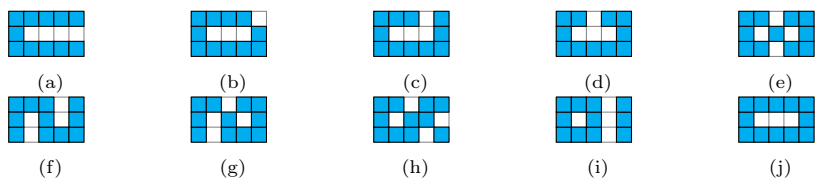

\centering
\begin{subfigure}[t]{0.19\textwidth}\centering
\input{tikz/5x3A.tikz}
\caption{}
\label{subfig:5x3A}
\end{subfigure}
\begin{subfigure}[t]{0.19\textwidth}\centering
\input{tikz/5x3B.tikz}
\caption{}
\label{subfig:5x3B}
\end{subfigure}
\begin{subfigure}[t]{0.19\textwidth}\centering
\input{tikz/5x3C.tikz}
\caption{}
\label{subfig:5x3C}
\end{subfigure}
\begin{subfigure}[t]{0.19\textwidth}\centering
\input{tikz/5x3D.tikz}
\caption{}
\label{subfig:5x3D}
\end{subfigure}
\begin{subfigure}[t]{0.19\textwidth}\centering
\input{tikz/5x3E.tikz}
\caption{}
\label{subfig:5x3E}
\end{subfigure}
\begin{subfigure}[t]{0.19\textwidth}\centering
\input{tikz/5x3F.tikz}
\caption{}
\label{subfig:5x3F}
\end{subfigure}
\begin{subfigure}[t]{0.19\textwidth}\centering
\input{tikz/5x3G.tikz}
\caption{}
\label{subfig:5x3G}
\end{subfigure}
\begin{subfigure}[t]{0.19\textwidth}\centering
\input{tikz/5x3H.tikz}
\caption{}
\label{subfig:5x3H}
\end{subfigure}
\begin{subfigure}[t]{0.19\textwidth}\centering
\input{tikz/5x3I.tikz}
\caption{}
\label{subfig:5x3I}
\end{subfigure}
\begin{subfigure}[t]{0.19\textwidth}\centering
\input{tikz/5x3J.tikz}
\caption{}
\label{subfig:5x3J}
\end{subfigure}
\caption{All the possible $3\times 5$ factors of excess greater than 1.}
\label{fig:5x3probleme}
\end{figure}

Let $W\DI{3j,3j+2}$ be the first factor among $W\DI{3i,3i+2}$, $i \in \{2,3,\dots,k-1\}$ such that $e(W\DI{3j,3j+2}) \geq 1$. $W\DI{3j,3j+2}$ must be of one of the forms in Figure \ref{fig:5x3probleme}. It is useful, in most cases, to study the $5\times5$ factor $W\DI{3j-1,3j+3}$. We have that $3j-2 \equiv_3 1$ and $3k+2 - (3j+3) = 3(k-j) -1 \equiv_3 2$. By hypothesis, we then have that $e(W\DI{1,3j-2}) \leq \emax(3j-2,5) = 2/3$ and $e(W\DI{3j+4,3k+2}) \leq \emax(3(k-j)-1) = 4/3$. Since $e(W\DI{1,3j-2})+e(W\DI{3j+4,3k+2}) \leq 6/3 = 2$ and $e(W) = 7/3$, we must have that $e(W\DI{3j-1,3j+3}) \geq 1/3$. This yields the condition
\[
|W[3j-1]|_\o+ |W[3j+3]|_\o \geq 7 - e(W\DI{3j,3j+2}).
\]
If $W\DI{3j,3j+2}$ is of the form \ref{subfig:5x3J} then we must have $|W[3j-1]|_\o+ |W[3j+3]|_\o \geq 5$. However, both $W[3j]$ and $W[3j+2]$ have 5  cells of degree 2. This means that  both $W[3j-1]$ and $W[3j+3]$ can have 0 $\o$-cells.

If $W\DI{3j,3j+2}$ is of any of the other forms, then we must have $|W[3j-1]|_\o+ |W[3j+3]|_\o \geq 6$. It is easy to verify that none of the forms allow 6 $\o$-cells  or more on $W[3j-1]$ and $W[3j+3]$, except for form \ref{subfig:5x3E}.

In this case, observe that both $W[3j]$ and $W[3j+2]$ have a degree 2 cell at each end. This means that $W[3j-1]$ and $W[3j+3]$ each contain at most $3$ cells. The only way for $W\DI{3j-1,3j+3}$ to have excess $1/3$ is for both $W[3j-1]$ and $W[3j+3]$ to contain 3 cells. So $W\DI{3j-1,3j+3}$ will be of the form shown in Figure \ref{subfig:5x5probleme}. Now, both $W[3j-1]$ and $W[3j+3]$ have 3 cells of degree 2. This means $W[3j-2]$ and $W[3j+4]$ will both contain at most 2 cells each. There is now two cases to consider. The first case is if $3j+4 = 3k+1$. In this case, the bottom row can be full and the bottom most 7 rows of $W$ will be of the form as seen in Figure \ref{subfig:5x7probleme}. Since the 8\textsuperscript{th} row from the bottom (i.e. $W[3j-2]$) contains at most 2 cells, the factor of the bottom 8 rows of $W$ will have excess at most $-2/3$. This forces $W\DI{1,3j-3}$ to have excess at least $7/3+2/3 = 9/3 = 3$. This is obviously in contradiction with the minimality of $h$.
Then we simply observe that the $5 \times 7$ factor $W\DI{3j-2,3j+4}$ is of excess $-7/3$. This forces $e(W\DI{1,3j-3}) + e(W\DI{3j+5,3k+2}) \geq 14/3$. But $e(W\DI{1,3j-3}) \leq \emax(3j-3,5) = 1$ and $e(W\DI{3j+5,3k+2}) \leq \emax(3k+2-(3j+4),5) = \emax(3(k-j)-2,5) = 2/3$. So $e(W\DI{1,3j-3}) + e(W\DI{3j+5,3k+2}) \leq 5/3$ which is a contradiction.

We have proved that it is impossible for any factor $W\DI{3i,3i+2}$, $i\in \{2,3,\dots,i-1\}$ to be of excess greater or equal than 1. Since $e(W\DI{1,5}) \leq 1/3$ and $e(W\DI{3k,3k+2}) \leq 1$, this implies that  it is impossible for $W$ to be of excess $7/3$.
\end{proof}

\begin{figure}
\begin{subfigure}[t]{0.49\textwidth}\centering
\centering
\input{tikz/5x5probleme.tikz}
\caption{$5\times 5$ factor forced in case \ref{subfig:5x3E}}
\label{subfig:5x5probleme}
\end{subfigure}
\begin{subfigure}[t]{0.49\textwidth}\centering
\centering
\input{tikz/5x7probleme.tikz}
\caption{$5\times7$ factor.}
\label{subfig:5x7probleme}
\end{subfigure}
\caption{}
\label{subfig:5x5casE}
\end{figure}

The following results give valuable information on words in $\MWordsdhwvarshort{2}{(3k+2)}{5}$ which will help us for the other lemmas.

\begin{corollary}
\label{cor:5xhTile}
For $W \in \MWordsdhwvarshort{2}{(3k+2)}{5}$ $k\geq2$, we call $W\DI{6,8},W\DI{9,10},\dots,W\DI{3k-3,3k-1}$ the interior tiles of $W$. Assume $W$ is such that it has no $5\times h'$ factor $h'<3k+2$ of excess greater than $\emax(h',5)$. Then
\begin{enumerate}[(i)]
\item no interior tile of $W$ has excess greater than or equal to $1$.
\item $W\DI{1,5}$ has excess $1/3$ and $e(W\DI{3k,3k+2})=1$.
\item no interior tile of $W$ has excess $-1$.
\end{enumerate}
\end{corollary}

\begin{proof}

We prove $i)$ by minimal counterexample. Let $h = 3k+2$ be minimal such that there exists $W \in \MWordsdhwvarshort{2}{(3k+2)}{5}$  with at least one interior tile of excess $\geq 1$. Let $W\DI{3i,3i+2}$ be the last interior tile that is of excess $\geq 1$.

If $W\DI{3i,3i+2}$ is of the form \ref{subfig:5x3B}, \ref{subfig:5x3C}, \ref{subfig:5x3D} or \ref{subfig:5x3J}, then either $W[3i]$ or $W[3i+2]$ is a full row of degree 2 cells. A $3\times 5$ factor adjacent to such a row has excess at most $-2$. Furthermore, if $W\DI{3i,3i+2}$ is of the form \ref{subfig:5x3J} then both adjacent $3\times5$ factors have excess at most $-2$. This is impossible by Proposition \ref{prop:5xhmin2factor} which means $W\DI{3i,3i+2}$ cannot be of any of those forms. 

If $W\DI{3i,3i+2}$ is of the form \ref{subfig:5x3A}, then  $W[3i]$ and $W[3i+2]$ contain 4 degree 2 cells. Any $3\times 5$ factor adjacent to $W\DI{3i,3i+2}$ is of excess at most $-1$.

If $W\DI{3i,3i+2}$ is of the form \ref{subfig:5x3F}, \ref{subfig:5x3G}, \ref{subfig:5x3H} or \ref{subfig:5x3I}, then either $W[3i]$ or $W[3i+2]$ is of the form $(\o,\o,\o,\e,\o)$ where the first three $\o$ cells are degree 2. A $3\times5$ factor adjacent to such row has excess at most $-1$.

Finally, if $W\DI{3i,3i+2}$ is of the form \ref{subfig:5x3E}, the only $3\times 5$ factor of excess 0 that can be adjacent to it is shown in red in Figure \ref{fig:5x6probleme}. This new factor has at one end a full row of degree 2 cell. This means that a $3 \times 5$ factor adjacent to this row has excess at most $-2$, a $2 \times 5$ factor adjacent to this row is of excess at most $-5/3$ and a $1 \times 5$ factor adjacent to this row is of excess at most $-10/3$. This means $W\DI{3i,3i+2}$ is adjacent to either a $3\times 5$ factor of excess $-1$, a $6\times 5$ factor of excess $-2$, a $5\times5$ factor of excess $-5/3$ or a $4\times 5$ factor of excess $-10/3$. By symmetry, this is true for both sides of $W\DI{3i,3i+2}$.

For all forms of $W\DI{3i,3i+2}$ we have a $3\times 5$ factor of excess at most $-1$ that is forced.

\begin{figure}
\centering
\input{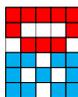}
\caption{The only $3\times5$ factor of excess 0 (in red) that can be adjacent to a factor of the form \ref{subfig:5x3E} (in cyan).}
\label{fig:5x6probleme}
\end{figure}

Observe that we must have $e(W\DI{3k,3k+2})\geq 0$ for if $e(W\DI{3k,3k+2})\leq -1$, then $e(W\DI{1,3k-1}) \geq 7/3$ which is impossible by Lemma \ref{lem:5x3k+2upbound}. Now let $e(W\DI{3k,3k+2})=0$. If $W\DI{3i,3i+2}$ is of any form other than \ref{subfig:5x3E}, then $e(W\DI{3(i+1),3(i+1)+2}) \leq -1 \implies e(W\DI{1,3i+2})\geq 7/3$ which is impossible and $e(W\DI{3(i-1),3(i-1)+2}) \leq -1 \implies e(W\DI{1,3(i-1)-1}) = 4/3$ and $e(W\DI{1,3i+2}) = 4/3$ which implies $W\DI{1,3i+2}$ is a 2-full word in $\Wordsdhwvarshort{2}{(3i+2)}{5}$ with an interior tile of excess 1 in contradiction with the minimality of $k$. 

If $W\DI{3i,3i+2}$ is of the form \ref{subfig:5x3E}, then it is either adjacent to a $3\times5$ factor of excess $-1$ which is impossible from the argument above, or both of its adjacent $3\times 5$ factors are of excess 0. If $i+1 < k$, then $e(W\DI{3(i+2),3(i+2)+2}) \leq -2$ which implies $e(W\DI{3i,3(i+2)+2}) \leq -1$ and $e(W\DI{1,3i-1}) = 7/3$ which is impossible. $i+1 = k$ is thus forced. Now, $e(W\DI{3(i-2),3(i-2)+2}) = -2$ which implies $e(W\DI{3(i-2),3k+2})=-1$ which finally implies $e(W\DI{1,3(i-2)-1}) = 7/3$. This is impossible. In all cases, $W\DI{3k,3k+2}$ must have excess 1.

At this point, we have proved that if $W \in \MWordsdhwvarshort{2}{(3k+2)}{5}$ has at least one interior tile of excess 1, (in this case $W\DI{3i,3i+2}$), then $W$ has at least one interior tile of excess $-1$ and $e(W\DI{3k,3k+2}) = 1$.

First assume there exists $j>0$ such that $i+j <k$ and $e(W\DI{3(i+j),3(i+j)+2}) \leq -1$ (i.e. there is an interior tile of excess $-1$ below $W\DI{3i,3i+2}$). 

Observe that there can be only one $3\times 5$ factor below $W\DI{3i,3i+2}$
of negative excess because this would imply $e(W\DI{1,3i+2}) \geq7/3$ which is impossible. Moreover if there is one $j>0$ such that $e(W\DI{3(i+j),3(i+j)+2}) \leq -1$ then $e(W\DI{1,3i+2}) = 4/3$ and there exists no $3\times 5$ interior tile of excess $\leq-1$ above by minimality of $k$. So we have $e(W\DI{3(i+j),3k+2})=0$. If $W\DI{3i,3i+2}$ is of the form \ref{subfig:5x3E}, then $e(W\DI{3(i+1),3(i+1)+2}) = 0$ implies $e(W\DI{3(i+2),3(i+2)+2}) \leq -2$ which is impossible by Proposition \ref{prop:5xhmin2factor}. This implies $W\DI{3i,3i+2}$, regardless of its form, must be adjacent to a $3\times 5$ factor of excess $-1$. This means that $e(W\DI{3(i+1),3(i+1)+2}) = -1$ and thus $j=1$.

We now have that $e(W\DI{1,3i+2}) = 4/3$ which from Lemma \ref{lem:5x3k+2upbound} means $W\DI{1,3i+2} \in \MWordsdhwvarshort{d}{3i+2}{5}$. This means, by minimality of $h$, that $W\DI{1,3i+2}$ has no interior tile of excess 1 or greater.

So we have the following picture:
$$e(W\DI{1,5}) = 1/3, e(W\DI{3k,3k+2}) =1,$$
$$e(W\DI{3i,3i+2}) = 1, e(W\DI{3(i+1),3(i+1)+2}) = -1$$
and all other interior tiles of $W$ have excess 0. We claim that no forms of $W\DI{3i,3i+2}$ satisfy the conditions above. 

Forms \ref{subfig:5x3B}, \ref{subfig:5x3C}, \ref{subfig:5x3D} and \ref{subfig:5x3J} were ruled out at the beginning of the proof. Forms \ref{subfig:5x3A}, \ref{subfig:5x3F} and \ref{subfig:5x3I} force both $W\DI{3(i+1),3(i+1)+2}$ and $W\DI{3(i-1),3(i-1)+2}$ to be of excess at most $-1$ which is not permitted. If $W\DI{3i,3i+2}$ is of the forms \ref{subfig:5x3E} or \ref{subfig:5x3G}, then $e(W\DI{3(i-1),3(i-1)+2}) = 0$ implies $e(\DI{3(i-2),3(i-2)+2}) = -2$ (Figure \ref{fig:5x6probleme}). Finally, if $W\DI{3i,3i+2}$ is of the form \ref{subfig:5x3H}, then $e(W\DI{3(i+1),3(i+1)+2}) = -1$ implies $e(W\DI{3(i+2),3(i+2)+2}) \leq -1$ which is not permitted (Figure \ref{fig:5xhEdge}). No possible forms of $W\DI{3i,3i+2}$ satisfy the conditions. This means there cannot be an interior tile of excess $-1$ below $W\DI{3i,3i+2}$.

Finally, observe that if there are no interior tile of excess $-1$ below $W\DI{3i,3i+2}$, then the factor $W\DI{3i,3k+2}$ must have excess exactly $2>1=\emax(3(k-i+1),5)$ which contradicts the hypothesis. This completes the proof of $i)$

\begin{figure}

\centering
\input{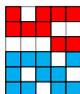}
\caption{In red, one  possible $3\times 5$ factor of excess $-1$  adjacent to  a factor of form \ref{subfig:5x3H}. This factor and any other possible factor has an adjacent $3\times 5$ factor above of excess at most $-1$.}

\label{fig:5xhEdge}
\end{figure}

To prove $ii)$, observe that if $e(W) = 4/3$ and $W$ has no interior tile of excess of excess 1 (from $i)$) then $e(W\DI{1,5}) + e(W\DI{3k,3k+2}) = 4/3$. If $e(W\DI{3k,3k+2}) = 2$, then it is of the form seen in Figure \ref{fig:5x3}. This means $W[3(k-1)+2]$ is empty thus $e(W\DI{3(k-1),3(k-1)+2}) \leq 2$ which is impossible by Proposition \ref{prop:5xhmin2factor}. If $e(W\DI{1,5}) = 4/3$, then it must be of the form seen in Figure \ref{subfig:5x5} and $e(W\DI{3k,3k+2}) = 0$. Observe that $W[5] = (\o,\e,\o,\o,\o)$ where the three last cells are of degree 2. This force $e(W\DI{6,8})\leq-1$. But this yields $e(W\DI{1,8}) = 1/3$ which, along with $e(W\DI{3k,3k+2}) = 0$ implies one interior tile of $W$ has excess at least 1 which is impossible. Only $e(W\DI{1,5}) = 1/3, e(W\DI{3k,3k+2})$ is possible which is what was needed.

$iii)$ is immediate from $i)$ and $ii)$.

\end{proof}

\begin{corollary}
\label{cor:5xhTileH}
For $k \geq 2$, up to symmetry, the unique $3 \times 5$ factor of excess $1$ at the top and bottom of a 2-full $W \in \MWordsdhwvarshort{2}{(3k+2)}{5}$ is the factor in Figure \ref{subfig:5x3H}.
\end{corollary}

\begin{proof}
Figure \ref{fig:5x3probleme} shows all $3\times 5$ factors of excess 1 up to symmetry. All factors other than \ref{subfig:5x3H} are discarded as a top tile of $W$ because they either admit adjacent $3\times 5$ factors of excess at most $-1$ or a factor of excess 0 that itself has an adjacent factor of excess at most $-1$ (Figures \ref{fig:5xhEdgeB} and \ref{fig:5x6probleme} for examples). The presence of these factors of excess at most $-1$ are in contradiction with Corollary \ref{cor:5xhTileH}. 
By symmetry, the same argument is true for the bottom $3\times 5$ factor of $W$.
\end{proof}

\begin{figure}

\centering
\input{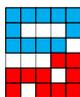}
\caption{In red, the unique $3\times 5$ factor of excess 0 that can be adjacent to a factor of form \ref{subfig:5x3B}. We see that any $3\times 5$ factor below it has excess at most $-1$.}

\label{fig:5xhEdgeB}
\end{figure}

We will use the last results to prove the following lemmas.

\begin{lemma}
\label{lem:5x3kupbound}
Let $W \in \Wordsdhwvarshort{2}{3k}{5}$ $k\geq 2$ be such that no $5\times h'$ factor $h'<3k$ has excess greater than $\emax(h',5)$. Then $e(W) \leq \emax(3k,5)$.
\end{lemma}
\begin{proof}
Let $h = 3k$, $k \geq 2$. Then
$$\emax(h,5) = \emax(3k,5) = 1.$$

We now prove that no $W \in \Wordsdhwvarshort{2}{3k}{5}$ with $k\geq2$ has excess $e(W) \geq 2$. The cases $k=2,3,4,5$ can be verified with a computer program.

Let $k \geq 6$ and $W \in \Wordsdhwvarshort{2}{3k}{5}$ be minimal such that $e(W) = 2$. We factorize $W$ as $W=W\DI{1,8}\vcat W\DI{9,10}\vcat W\DI{11,3k}$. Observe that $8 = 3(2)+2$ and $3k-10 = 3(k-4)+2$. So Corollaries \ref{cor:5xhTile} and \ref{cor:5xhTileH} apply to both $W\DI{1,8}$ and $W\DI{11,3k}$. If the factors $W\DI{1,8}$ and $W\DI{11,3k}$ are 2-full, they have excess $4/3$ so that in order to satisfy $e(W)=2$ we must have $|W\DI{9,10}|_{\o} = 6$. However, from Corollary \ref{cor:5xhTileH}, $W\DI{1,8}$ has 3 degree 2 cells on its bottom row and $W\DI{11,3k}$ has 3 degree 2 cell on its top row 
which leaves room for at most 4 $\o$ cells.

If $e(W\DI{1,8})=1/3$ and $e(W\DI{11,3k}) = 4/3$, then we must have $|W\DI{9,10}|_{\o} = 7$ which is again impossible because there is room for at most 6 $\o$ cells in  $W\DI{9,10}$.

Likewise, if $e(W\DI{11,3k}) = e(W\DI{1,8})=1/3$, then we must have $|W\DI{9,10}|_{\o} = 8$ which is impossible because it implies rows $W[8]$ and $W[11]$ both have at most 2 $\o$ cells which contradicts $e(W\DI{11,3k}) = e(W\DI{1,8})=1/3$.

Finally, if $e(W\DI{1,8})=-2/3$ and $e(W\DI{11,3k}) = 4/3$ then we must have $|W\DI{9,10}|_{\o} = 8$ which is impossible because $W\DI{11,3k}$ being 2-full implies there is at most room for 6 $\o$ cells in $W\DI{9,10}$.

All cases have been covered up to symmetry.

\end{proof}

\begin{lemma}
\label{lem:5x3+1kupbound}
Let $W \in \Wordsdhwvarshort{2}{3k+1}{5}$ $k\geq 2$ be such that no $5\times h'$ factor $h'<3k+1$ has excess greater than $\emax(h',5)$. Then $e(W) \leq \emax(3k,5)$.
\end{lemma}

\begin{proof}
Let $h = 3k +1$, $k \geq 2$. Then
$$\emax(h,5) = \emax(3k+1,5) = 2/3$$

We now prove that no $W \in \Wordsdhwvarshort{2}{3k+1}{5}$ with $k\geq2$ has excess $e(W) \geq 5/3$. The cases $k=2,3,4,5$ can be verified with a computer program.

Let $k\geq 6$ and $W \in \Wordsdhwvarshort{2}{3k+1}{5}$ be minimal such that $e(W)=5/3$. We partition $W$ into $W\DI{1,8}\vcat W\DI{9,11}\vcat W\DI{12,3k+1}$. Observe that $8 = 3(2)+2$ and $3k+1-11 = 3(k-4)+2$. So Corollaries \ref{cor:5xhTile} and \ref{cor:5xhTileH} apply to both $W\DI{1,8}$ and $W\DI{12,3k+1}$. If both $W\DI{1,8}$ and $W\DI{12,3k+1}$ are 2-full, they both have excess $4/3$ so in order to have $e(W) = 5/3$, we must have $|W\DI{9,11}|_{\o} = 9$. However, from Corollary \ref{cor:5xhTileH}, $W\DI{1,8}$ has 3 degree 2 cell on its bottom row and $W\DI{12,3k+1}$ has 3 degree 2 cell on its top row which leaves room for at most 8 $\o$ cells.

If $e(W\DI{1,8})=4/3$ and $e(W\DI{12,3k+1})=1/3$, then we must have $|W\DI{9,12}|_{\o} = 10$. $W\DI{10,11}$ can have at most 8 $\o$ cells so $W[9]$ must contain at least 2 $\o$ cells. But since $W\DI{1,8}$ is 2-full, $W[9]$ can have at most 2 full cells one of which is degree 2. Both these constraints together creates at least one $\o$ cell of degree 3 which is not permitted (Figure \ref{fig:5xhdeg3}).

\begin{figure}
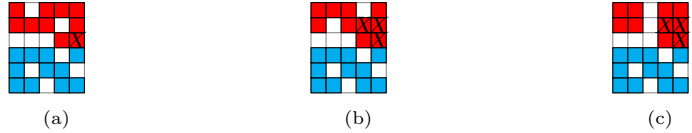

\begin{subfigure}[t]{0.32\textwidth}\centering
\centering
\input{tikz/5xhdeg3A.tikz}
\caption{}
\label{subfig:5xhdeg3A}
\end{subfigure}
\begin{subfigure}[t]{0.32\textwidth}\centering
\centering
\input{tikz/5xhdeg3B.tikz}
\caption{}
\label{subfig:5xhdeg3B}
\end{subfigure}
\begin{subfigure}[t]{0.32\textwidth}\centering
\centering
\input{tikz/5xhdeg3C.tikz}
\caption{}
\label{subfig:5xhdeg3C}
\end{subfigure}
\caption{Marked in red are rows 9,10,11. Marked with an $X$ are the $\o$ cells of degree 3 that are forced.}
\label{fig:5xhdeg3}
\end{figure}

If $e(W\DI{1,8}) = e(W\DI{12,3k+1}) = 1/3$, then we must have $|W\DI{9,12}|_{\o} = 10$. This means $W\DI{9,12}$ must be of one of the forms in Figure \ref{fig:5x3probleme} except \ref{subfig:5x3J}. Observations from the proof of Corollary \ref{cor:5xhTile} tells us at least one of $W\DI{6,8}$ and $W\DI{12,14}$ has excess at most $-1$. We also observe that both $W[9]$ and $W[11]$ have at least 2 degree 2 cells. If $e(W\DI{6,8}) = -1$ then $W\DI{1,5}$ is 2-full and $W[5]$ has 3 degree 2 cells. If $e(W\DI{12,14})=-1$ then $W\DI{15,3k+1}$ is 2-full and $W[15]$ has three degree 2 cells. This means one of $W\DI{6,8}$ and $W\DI{15,3k+1}$ has at least 9 $\o$ cells but is in between a row of 3 degree 2 cells and a row of 2 degree 2 cells. This is impossible.

All other cases up to symmetry are easily excluded and we have proved that $e(W) = 5/3$ is impossible. 

\end{proof}

We are now ready to prove the following corollary.

\begin{corollary}
\label{cor:w=5}
$e_{max}(h,5) = \emax(h,5) \quad \forall h\geq 5$.
\end{corollary}
\begin{proof}
Lemma \ref{lem:w=5lowbound} yields $e_{max}(h,5) \geq \emax(h,5)$. Now, by contradiction, let $h$ be minimal such that there exists $W \in \Wordsdhwvarshort{2}{h}{5}$ with $e(W) > \emax(h,5)$. By Lemma \ref{lem:5x3kupbound}, $h \neq 3k$ for any $k$. By Lemma \ref{lem:5x3+1kupbound}, $h \neq 3k+1$ for any $k$. Finally, by Lemma \ref{lem:5x3k+2upbound}, $h \neq 3k+2$ for any $k$. This is impossible which means $e_{max}(h,5) = \emax(h,5)$.
\end{proof}

Note that $w=5$ is the smallest width such that maximal snakes are of smaller area than 2-full words.

\subsubsection{\textbf{$w=6.$}}
Recall that for $h\geq 6$ we have 
\begin{align*}
\emax(h,6)=
\begin{cases}
2 &\mbox{ if } h\equiv_3 0,\\
1 &\mbox{ if } h\equiv_3 1,2
\end{cases}
\end{align*}

We will need the following propositions.
\begin{proposition}\label{prope(hx6)=2W[1]}
Let $k\geq 2,\; h=3k, h'<h$ and $W\in \Wordsdhwvarshort{2}{h}{6}$  with $e(W)=2$ such that every $h'\times 6$ factor $W'$ of $W$ has excess $e(W')\leq \emax(W')$. Then 
$$|W[1]|_\o=5,\; |W[2]|_\o=4,\; |W[3]|_\o=4.$$
\end{proposition}
\begin{proof}
We first prove that $|W\DI{1,2}|_\o=9$. Indeed $|W\DI{1,2}|_\o>9$ is  impossible and 
\begin{align*}
|W\DI{1,2}|_\o\leq 8 &\Rightarrow  e(W\DI{1,2})\leq 0 \Rightarrow e(W\DI{3,h})\geq 2 \mbox{ and } h(W\DI{3,h})\equiv_3 1
\end{align*}
so that $e(W\DI{3,h})>\emax(W\DI{3,h})$ in contradiction with our hypothesis. We must then have $|W\DI{1,2}|_\o=9$. 
Next we prove that $|W[1]|_\o=5$. If $|W[1]|_\o=6$ then $|W[2]|_\o\leq 2$, $|W\DI{1,2}|_\o\leq 8$ which was just proved to be impossible. Also
\begin{align*}
|W[1]|_\o\leq 4 &\Rightarrow  e(W\DI{2,h})\geq2
\end{align*}
in contradiction with our hypothesis.  So we must have 
$|W[1]|_\o=5$ and $|W[2]|_\o=4$. Next we prove that $|W[3]|_\o=4$. We have
\begin{align*}
|W[3]|_\o= 3 &\Rightarrow  e(W\DI{1,3})=0 \Rightarrow  e(W\DI{4,h})= 2, \\
&\Rightarrow |W[4]|_\o=5 \mbox{ and }|W[5]|_\o=4.
\end{align*}
But $|W[3]|_\o= 3 ,|W[4]|_\o=5, |W[5]|_\o=4$ is impossible by inspection (Figure \ref{hx6_45254}). Since the hypothesis $|W[3]|_\o \leq 2$ is also rejected, we must have $|W[3]|_\o\geq 4$. But 
\begin{align*}
|W[3]|_\o= 5 &\Rightarrow W\DI{1,3} \mbox{ is unique and implies } |W[4]|_\o\leq 2, \\
&\Rightarrow e(W\DI{1,4})\leq 0,\\
&\Rightarrow e(W\DI{5,h})\geq 2 \mbox{ and } h(W\DI{5,h})\equiv_3 2
\end{align*}
in contradiction with our hypothesis. Since  $|W[3]|_\o=6$ is impossible when $|W[2]|_\o=4$ we must have  $|W[3]|_\o=4$ and the proof is complete. 
\end{proof}
Theorem \ref{thm:main} for $h\times 6$ is proved through a sequence of lemmas and propositions that cover all.cases. 

\begin{proposition} \label{prop_emax>=ê}
For $h\geq 6$ let $e_{max}(h,6)$ be the maximal excess of a word $W\in \Wordsdhwvarshort{2}{h}{6}$ and let $e_{max,fs}(h,6)$ be the maximal excess of a $h\times 6$ rectangle of forest of snakes.. Then
\begin{align*}
e_{max}(h,6)\geq
\begin{cases}
\emax(h,6) &\mbox{ if } h\geq 6\\
e_{max,fs}(h,6)&\mbox{ if } h\geq 7.
\end{cases}
\end{align*}
\end{proposition}
\begin{proof}
Figure \ref{hx6max} shows the case $h=6$ and illustrates a recipe for the construction of $(3k+i)\times 6$ forests of snakes W for $k\geq 2,i\in{0,1,2}$ with excess $e(W)=\emax(h,6)$. Blue and red copies of the same $3\times 6$ tile $T$ of excess $0$ are stacked at the bottom of each other in alternating order to create a $3k\times6$ rectangle. For $i\in\{1,2 \}$,
the bottom  $i\times 6$ tile is made of the top $i$ rows of  $T$. For $i=0$, the bottom row is $W[3]$. A
black cell is added to the top and bottom row when $i=0$ and only to the top when $i\in\{1,2 \}$. 
We thus have $e_{max}(h,6)\geq \emax(h,6)$.
\begin{figure*}[h!]
\begin{subfigure}[t]{.16\textwidth}
\centering
\includegraphics[height=.48in, valign=t]{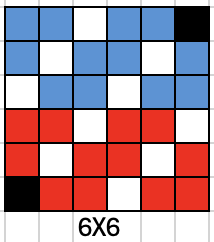}
\end{subfigure}%
~ 
\begin{subfigure}[t]{.16\textwidth}
\centering
\includegraphics[height=.56in,valign=t]{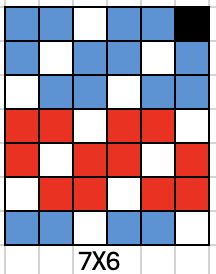}
\end{subfigure}%
~ 
\begin{subfigure}[t]{.16\textwidth}
\centering
\includegraphics[height=.62in,valign=t]{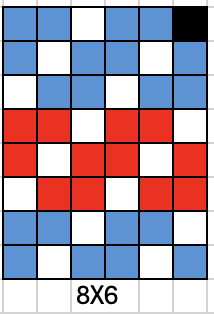}
\end{subfigure}%
~ 
\begin{subfigure}[t]{.16\textwidth}
\centering
\includegraphics[height=.68in,valign=t]{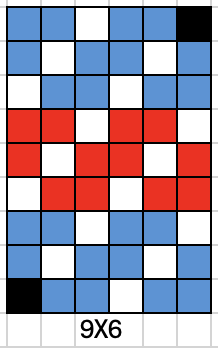}
\end{subfigure}%
~ 
\begin{subfigure}[t]{.16\textwidth}
\centering
\includegraphics[height=.8in,valign=t]{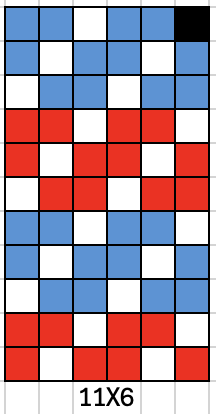}
\end{subfigure}%
~ 
\caption{  $h\times 6$ 2-full words}
\label{hx6max}
\end{figure*}
\end{proof}

\paragraph{\bf Case 1. $\mathbf{h\equiv_3 0}$}
\begin{lemma}\label{lem6xhe_maxh=0} For $h\geq 7$ let $h$  be minimal. such that 
$W\in \Wordsdhwvarshort{2}{h}{6}$ satisfies $e(W)=\emax(h,6)+1$. Then
$$ h\equiv_3 0 \Rightarrow |W[1]|_\o=|W[h]|_\o=6.$$
\end{lemma}
\begin{proof}
\begin{align*}
h\equiv_3 0  \mbox{ and } |W[1]|_\o \leq 5 &\Rightarrow e(W[1])\leq 1,\\
&\Rightarrow e(W\DI{2,h})\geq 2 \mbox{ and } h(W\DI{2,h}) \equiv_3 2,\\
&\Rightarrow e(W\DI{2,h})>\emax(h-1,6)
\end{align*}
in contradiction with the minimality hypothesis on $h$. By symmetry we also have $|W[h]|_\o=6$.
\end{proof}

\begin{lemma}\label{lem6xhe_maxh=0b} For $h\geq 7$ let $h$  be minimal such that 
$W\in \Wordsdhwvarshort{2}{h}{6}$ satisfies $e(W)=\emax(h,6)+1$. Then
$$ h\not\equiv_3 0 $$
\end{lemma}
\begin{proof} Assume $ h\equiv_3 0 $. Since $|W[1]|_\o=6$ from Lemma \ref{lem6xhe_maxh=0}, we have
\begin{align*}
|W[1]|_\o=6 &\Rightarrow |W[2]|_\o\leq 2,\\
&\Rightarrow e(W\DI{1,2})\leq 0,\\
&\Rightarrow e(W\DI{3,h})\geq \emax(h,6)+1= \emax(h-2,6)+2
\end{align*}
in contradiction with the minimality hypothesis of $h$. So $h\equiv_3 0$ is rejected. 
\end{proof}
\paragraph{\bf Case 2. $\mathbf{h\equiv_3 1}$}
\begin{lemma}\label{lem6xhe_maxh=1} For $h\geq 6$ let $h$  be minimal such that 
$W\in \Wordsdhwvarshort{2}{h}{6}$ satisfies $e(W)=\emax(h,6)+1$. Then
$$ h\equiv_3 1 \Rightarrow |W[1]|_\o=5,\;|W[2]|_\o= 4 \mbox{ and } |W[3]|_\o=4$$
\end{lemma}
\begin{proof} 
Assume that $|W\DI{1,2}|_\o\leq 8$
Then $e(W\DI{1,2})\leq 0$, so that $e(W\DI{3,h})\geq \emax(h,6)+1=2$ and $h(W\DI{3,h})\equiv_3 2$, in contradiction with the minimality hypothesis.
So $|W\DI{1,2}|_\o \geq 9$. Since $|W\DI{1,2}|_\o=10$ is impossible, 
we must have $|W\DI{1,2}|_\o=9$. Moreover  $|W[1]|_\o=6 \Rightarrow |W\DI{1,2}|_\o\leq 8$ which is impossible. So $|W[1]|_\o\leq 5$. Now
$|W[1]|_\o=4$ and $|W[2]|_\o=5$ imply $|W[3]|_\o\leq 2 \Rightarrow e(W\DI{1,3})\leq -1$.
Therefore, $e(W\DI{4,h})\geq \emax(6,h)+2$, in contradiction with the minimality hypothesis.
So we have  $|W[1]|_\o=5$ and  $|W[2]|_\o=4$. 
Next, on one hand $|W[3]|_\o\leq 3$ implies $e(W\DI{1,3})\leq 0 \Rightarrow e(W\DI{4,h})\geq \emax(h,6)+1=2$ and $h(W\DI{4,h})\equiv_3 1$, in contradiction with the minimality hypothesis.
On the other hand $|W[3]|_\o= 5$ implies that $W\DI{1,3}$ is unique and $|W[4]|_\o\leq 2$, so that $e(W\DI{5,h})\geq 2$ and $h(W\DI{5,h})\equiv_3 0$.
Therefore, $|W[5]|_\o=5, |W[6]|_\o=4$, from Proposition \ref{prope(hx6)=2W[1]}.
But this is in contradiction with $|W[1]|_\o=5$ and $|W[2]|_\o=4$ (Figure \ref{hx6_45254}).
\end{proof}

\begin{lemma}\label{lem6xhe_maxh=1b} For $h\geq 6$ let $h$  be minimal such that 
$W\in \Wordsdhwvarshort{2}{h}{6}$ satisfies $e(W)=\emax(h,6)+1$. Then
$$ h\equiv_3 1 \Rightarrow  |W[i]|_\o=4 \; \forall \; i: 1<i<h$$
\end{lemma}
\begin{proof} 
Let $1<i<h$ be the smallest integer such that $ |W[i]|_\o \neq 4$.
Then $|W[i]|_\o= 3$ so that $e(W\DI{1,i})=0$, which implies $e(W\DI{i+1,h})=2$, $|W[i+1]|_\o=5$ and $|W[i+2]|_\o=4$ from Proposition \ref{prope(hx6)=2W[1]}, in contradiction with $|W[i]|_\o= 3$.
Now, if $|W[i]|_\o= 5$, then $e(W\DI{1,i})=2$, so that $|W[i+1]|_\o\leq 2$.
Therefore, $e(W\DI{1,i+1})\leq 0$ and $e(W\DI{i+2,h})=2$, which implies $|W[i+2]|_\o=5, |W[i+3]|_\o=4$, from Proposition \ref{prope(hx6)=2W[1]}.
This implies that the row distribution of  $W\DI{i-1,i+3}$ is $(4,4,5,2,5,4)$ or $(5,4,5,2,5,4)$ 
But these two  row distributions are not feasible as shown in Figure \ref{hx6_45254}
and the proof is complete. 
\begin{figure*}[h!]
\begin{subfigure}[t]{.2\textwidth}
\centering
\includegraphics[height=.48in, valign=t]{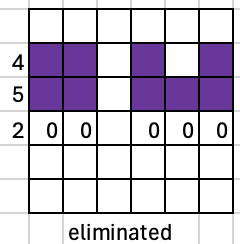}
\caption{}\label{45254a}
\end{subfigure}%
~ 
\begin{subfigure}[t]{.2\textwidth}
\centering
\includegraphics[height=.48in,valign=t]{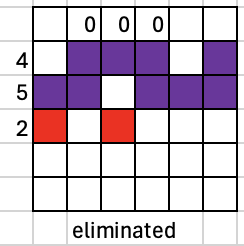}
\caption{}\label{45254b}
\end{subfigure}%
~ 
\begin{subfigure}[t]{.2\textwidth}
\centering
\includegraphics[height=.48in,valign=t]{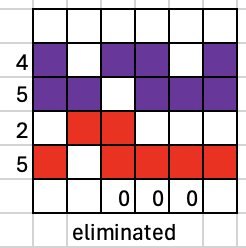}
\caption{}\label{45254c}
\end{subfigure}%
~ 
\begin{subfigure}[t]{.2\textwidth}
\centering
\includegraphics[height=.48in,valign=t]{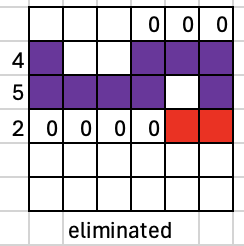}
\caption{}\label{45254d}
\end{subfigure}%
\caption{  row distributions $(4,4,5,2,5,4)$ and $(5,4,5,2,5,4)$}
\label{hx6_45254}
\end{figure*}
\end{proof}

\begin{proposition}\label{prop6xh_h_equiv1}
For $h\geq 6$ let  $W\in \Wordsdhwvarshort{2}{h}{6}$. Then
$$h\equiv_3 1 \Rightarrow e(W)\leq \emax(h,6)$$ 
\end{proposition}
\begin{proof}By minimal counterexample. Assume $h\equiv_3 1$ and $h>6$ is minimal such that there is a $h\times 6$ word  $W$ with $e(W)=\emax(h,6)+1$. Since  $h\equiv_3 1$ we know from Lemma \ref{lem6xhe_maxh=1} and Lemma \ref{lem6xhe_maxh=1b} that $|W[1]|_\o =5$ and $|W[i]|_\o =4$ for all $1<i<h$. Up to symmetry, there are 
four $2\times 6$ words $W'$ such that $e(W'[1])=5, e(W'[2])=4$ and they appear in Figure \ref{6x2n=9} where only the right-most rectangle admits a third row $W[3]$ with $|W[3]|_\o=4$. The $3\times 6$ rectangle in Figure \ref{6x3outa} is excluded because it does not admit a row $W[4]$ with $|W[4]|_\o =4$. The $h\times 6$ rectangles in Figures \ref{6x3101a}, \ref{6x3101b}, \ref{6x3101c} with $W[3]=(\purple,\e,\purple,\ldots) $ show the  $3$ possibilities for $W[4]$ and they are all excluded because they do not admit one, two  or three supplementary rows with $4$ cells. 
\begin{figure*}[h!]
\centering
\begin{subfigure}[t]{0.82\textwidth}
\centering
\includegraphics[height=.3in]{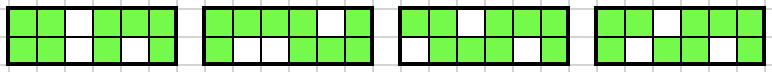}
\caption{}\label{6x2n=9}
\end{subfigure}%
\\
\begin{subfigure}[t]{0.16\textwidth}
\centering
\includegraphics[height=.38in,valign=t]{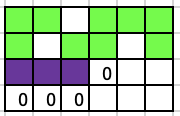}
\caption{impossible}\label{6x3outa}
\end{subfigure}%
\begin{subfigure}[t]{0.16\textwidth}
\centering
\includegraphics[height=.48in,valign=t]{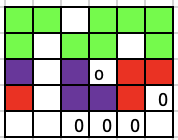}
\caption{impossible}\label{6x3101a}
\end{subfigure}%
~
\begin{subfigure}[t]{0.16\textwidth}
\centering
\includegraphics[height=.57in,valign=t]{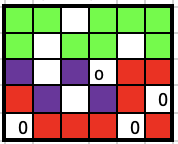}
\caption{impossible}\label{6x3101b}
\end{subfigure}%
~
\begin{subfigure}[t]{0.16\textwidth}
\centering
\includegraphics[height=.57in,valign=t]{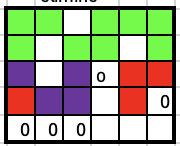}
\caption{impossible}\label{6x3101c}
\end{subfigure}%
~
\begin{subfigure}[t]{0.16\textwidth}
\centering
\includegraphics[height=.48in,valign=t]{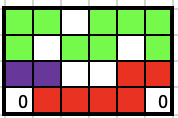}
\caption{impossible}\label{6xhr3=110}
\end{subfigure}%
\begin{subfigure}[t]{0.16\textwidth}
\centering
\includegraphics[height=.92in,valign=t]{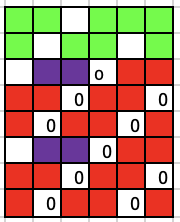}
\caption{}\label{6xhr3=011}
\end{subfigure}%
\caption{$h\times 6$ rectangles with $|W[1,2]|_\o =9$ }\label{6xhmax}
\end{figure*}
The $h\times 6$ rectangle in Figure \ref{6xhr3=110} with  $W[3]=(\purple,\purple,\ldots)$   (Figure \ref{6xhr3=110}) is eliminated for the same reason. 
The only remaining possibility is $W[3]=(\e,\o,\o,\e,\o,\o)$ (Figure \ref{6xhr3=011}) and it allows only $W[4]=(\o,\o,\e,\o,\o,\e)$ and $W[5]=(\o,\e,\o,\o,\e,\o)$. But then $W[5]=W[2]$ and either $|W[6]|_\o =5$ or $W[6]=W[3]$. But  $W[6]=W[3]$ does not admit  $|W[7]|_\o=5$. We thus obtain a periodic sequence of rows of period $3$ such that $W[3k+2]=W[2]=(\o,\e,\o,\o,\e,\o),W[3k]=W[3]=(\e,\o,\o,\e,\o,\o),W[3k+1]=W[4]=(\o,\o,\e,\o,\o,\e)$ with $|W[3k]|_\o =5$ as the only possibility for a row of $W$ to have $5$ filled cells. So we must have $|W[3k+1]|_\o =4$ and the proof is complete. 
\end{proof}
\paragraph{\bf Case 3. $\mathbf{h\equiv_3 2}$}

\begin{lemma}\label{lem6xhe_maxh=0c} For $h\geq 6$ and  $W\in \Wordsdhwvarshort{2}{h}{6}$, let $h$  be minimal. such that  $e(W)=\emax(h,6)+1$. Then
$$ h\equiv_3 2 \Rightarrow |W[1]|_\o= 5 \mbox{ and } |W[i]|_\o= 4 \mbox{ for all } 1<i<h$$
\end{lemma}
\begin{proof} 
\begin{align*}
|W[1]|_\o\leq 4 &\Rightarrow e(W\DI{2,h})\geq \emax(h-1,6) +1
\end{align*}
in contradiction with the minimality hypothesis. 
\begin{align*}
|W[1]|_\o= 6  &\Rightarrow |W[2]|_\o\leq 2 \Rightarrow e(W\DI{3,h})=2
\end{align*}
which forces $|W[2]|_\o= 2$.  But from Proposition \ref{prope(hx6)=2W[1]} we have $|W[3]|_\o=5$ and 
$|W[4]|_\o=4$ which, together with $|W[1]|_\o= 6$, implies $|W[2]|_\o\leq 1$ which is impossible. So we must have
$|W[1]|_\o= 5$. Now assume that $1<i$ is minimal such that $|W[i]|_\o\neq4$. 
\begin{align*}
|W[i]|_\o= 3  &\Rightarrow e(W\DI{1,i})\leq 0 \Rightarrow e(W\DI{i+1,h})\geq 2 \mbox{ and } 
h(W\DI{i+1,h})\equiv_3 0,\\
&\Rightarrow |W[i+1]|_\o= 5 \mbox{ and }  |W[i+2]|_\o= 4,\\
&\Rightarrow |W[i]|_\o\leq 2
\end{align*}
in contradiction with the hypothesis. 
\begin{align*}
|W[i]|_\o= 5  &\Rightarrow e(W\DI{1,i})=2  \mbox{ and } i\equiv_3 0,\\
&\Rightarrow |W[i+1]|_\o\leq 2 \mbox{ by inspection, } \\
&\Rightarrow e(W\DI{1,i+1})=0 \\
&\Rightarrow e(W\DI{i+2,h})= \emax(h,6)+1 \mbox{ and } h-(i+1)\equiv_3 1,
\end{align*}
in contradiction with the minimality hypothesis. The other possibility $|W[i]|_\o=6$ is easily discarded with similar arguments and we have proved that $|W[i]|_\o=4$. 
\end{proof}
\begin{proposition}\label{prop6xh_h_equiv2}
For $h\geq 6$ let  $W\in \Wordsdhwvarshort{2}{h}{6}$. Then
$$h\equiv_3 2 \Rightarrow e(W)\leq \emax(h,6)$$ 
\end{proposition}
\begin{proof}
The proof is similar to the proof of Proposition \ref{prop6xh_h_equiv1}: the sequence of rows $W[i]$ is unique and is periodic of period $3$. The only value $h$ for which $|W[h]|_\o=5$ is possible is when  $h\equiv_3 0$.
\end{proof}

\begin{corollary}
\label{cor:w=6}
For all $h\geq 6$ and  $W\in \Wordsdhwvarshort{2}{h}{6}$ we have 
$$ e_{max}(h,6)= \emax(h,6)
$$
\end{corollary}
\begin{proof}
From Proposition \ref{prop_emax>=ê} we know that $e_{max}(h,6)\geq \emax(h,6)$ and 
from Lemma \ref{lem6xhe_maxh=0b} and Propositions \ref{prop6xh_h_equiv1}, \ref{prop6xh_h_equiv2} we have $e_{max}(h,6)\leq \emax(h,6)$. 
\end{proof}

\subsubsection{\textbf{$(h,w) = (7,7).$}}

\begin{proposition}
\label{prop:7x7}
$e_{max}(7,7) = \emax(7,7) = 4/3.$
\end{proposition}

Figure \ref{fig:7x7} shows a $7\times7$ word of excess $4/3$.

Assume $W \in \Wordsdhwvarshort{2}{7}{7}$ and $e(W) = 7/3$. Partition $W$ as the vertical concatenation  $W=W\DI{1,2}\vcat W\DI{3,7}$. We know from Proposition \ref{prop4xhsc} that $e(W\DI{3,7}) \leq \emax(5,7) = 2/3$ and from Proposition \ref{prop:w=2} $e(W\DI{1,2}) \leq \emax(2,7) = 5/3$. So in order for $W$ to have excess $7/3$, $W\DI{3,7}$ and $W\DI{1,2}$ must both be 2-full. $W\DI{1,2}$ 2-full implies that it is unique up to symmetry and of the form shown in Figure \ref{fig:2x7Max}. We observe that $W[2]$ has at least 3 degree 2 cells in $W\DI{1,2}$ which means $|W[3]|_\o \leq 4$ implying $e(W[3]) \leq -2/3$. In total, we have that $e(W\DI{1,3}) \leq 1$ is forced. This means, in order for $W$ to have excess $7/3$, that $e(W\DI{1,3}) = 1$ and that $e(W\DI{4,7}) = 4/3$ and is thus 2-full. But by Corollary \ref{cor4xh_n2(c1)} a row adjacent to a 2-full $4 \times (3k+1)$ word has excess at most $-5/3$ contradicting $e(W[3]) = 1$. 

This is a contradiction.
\begin{figure*}[h!]
\centering
\includegraphics[height=.2in]{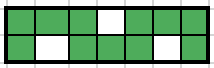}
\caption{2-full $2\times 7$ word}
\label{fig:2x7Max}
\end{figure*}

\begin{figure}
\centering
\input{tikz/7x7.tikz}
\caption{A word in $\MWordsdhwvarshort{2}{7}{7}$}
\label{fig:7x7}
\end{figure}

\subsubsection{Proof of Lemma \ref{lem:basecases}}

\begin{proof}
Propositions \ref{prop4xhsc}, \ref{prop:w=2}, \ref{prop3xhfssize}, \ref{prop:7x7} along with Corollaries \ref{cor:w=5}, \ref{cor:w=6} together prove that $e_{max}(h,w) = \emax(h,w)$ for all $(h,w) \in \Zp \times \{1,2,\dots,6\}\cup (7,7)$.
\end{proof}


\subsection{Properties of $d$-full words.}
\subsubsection{ Properties of words $W\in \MWordsdhwvarshort{2}{h}{3}$}

We introduce the concept of \emph{atomic words} in the context of words in $\Wordsdhwvarshort{d}{h}{w}$.

\begin{definition}[Atomic word]
We say that a word $W \in \Wordsdhwvarshort{2}{h}{w}$ is atomic if there is no $i$, $1\geq i \geq w$ such that $W[i] = \e^w$.
\end{definition}

Figures \ref{3x3} and \ref{3x6s} show examples of atomic words. Figure \ref{3x13fs} is not an atomic word and it is obtained by the concatenating of two copies of the atomic word in Figure \ref{3x6s} with $\e^3$ in between.

\begin{lemma} \label{lem3xhmax_r1r_2} 
Let $W\in \MWordsdhwvarshort{2}{h}{3}$. Then \\
$i)$ $W[1]=(\o,\o,\o),\; W[2]=(\o,\e,\o)$,\\
$ii)$ If $W$ is atomic then for all $1<i<h$, $|W[i]|_\o=2$,\\
$iii)$  If $W$ is atomic then  there is no interior $2\times 2$ filled square in $W$,\\
$iv)$ If $W$ is atomic then  there is no pair of consecutive empty cells in $W^t[1]$ and $W^t[3]$. 

\end{lemma}
\begin{proof} $i)$ By minimal counterexample. Observe that we only need to prove this statement for atomic words because if it is true for atomic words then it is also true for any $h\times 3$ word which is a "shuffle" product of 2-full atomic words with empty rows $(\e,\e,\e)$. 
We know by inspection that the statement is true for $h\leq 6$ 
So assume that $W$ is atomic and $h>6$ is minimal such that $W[1]\neq (\o,\o,\o)$. 
If $|W[1]|_\o\leq 1$ then $e(W\DI{2,h})\geq 3$ which is impossible. If $|W[1]|_\o=2$ then $e(W\DI{2,h})=2$ and $W\DI{2,h}$ is 2-full atomic and by minimality hypothesis we have 
$W[2]= (\o,\o,\o)$ and since $|W[1]|_\o=2$ we must have $W[3]= (\e,\e,\e)$ which is impossible. So we must have 
$W[1]= (\o,\o,\o)$ \\
Now if $|W[2]|_\o=1$ then $e(W\DI{3,h})=2$ and $W\DI{3,h}$ is 2-full atomic so by minimality hypothesis $W[3]= (\o,\o,\o)$ and $W[4]= (\o,\e,\o)$ which implies that $|W[2]|_\o=0$ contradicting the hypothesis. So we must have $W[2]= (\o,\e,\o)$. \\
$ii)$ By contradiction. Assume that $W$ is atomic and there exists $1<i<h$ such that  $|W[i]|_\o\neq2$. Let $i$ be minimal such $|W[i]|_\o\neq2$. If $W[i]=(\o,\o,\o)$ then $W[i+1]=(\e,\e,\e)$ which is forbidden.  If $|W[i]|_\o=1$ then $e(W\DI{1,i})=0$ and  $e(W\DI{i+1,h})=2$ which implies from $i)$ that $W[i+1]=(\o,\o,\o)$ and $W[i+2]=(\o,\e,\o)$ in contradiction with $|W[i]|_\o=1$. So we must have  $|W[i]|_\o=2$ for all $1<i<h$. 
\\
$iii)$ Suppose that $W$ is atomic and there exists $1<i<h-1$ such that $W\DI{i,i+1}=(\o,\o,\e)\vcat(\o,\o,\e)$. Then $|W[i-1]|_\o\leq 1$ in contradiction with $ii)$.\\
$iv)$ If there are two consecutive empty cells in $W^t[1]$ or $W^t[3]$ then, by $ii)$, on the rows of these two cells  there is  a filled $2\times 2$ square in contradiction with $iii)$. 

\end{proof}

\begin{proposition} \label{prop3xhmax_bench1} 
Let $h\geq 3$ and let $W\in \MWordsdhwvarshort{2}{h}{3}$ be an atomic word.
Then for some integers $t_1,t_2\geq 3$, there is a $t_1$-bench with its seat on  $W^t[1]$ and a $t_2$-bench with its seat on $W^t[3]$.
\end{proposition} 
\begin{proof}
This is verified by observation for $h\in\{3,4,5\}$. Assume that there exists an atomic $W\in \MWordsdhwvarshort{2}{h}{3}$ such that  there is no West bench with its seat on  column $W^t[3]$.
The reading of the rows of  $W$ from  top to bottom must satisfy the following rules : 

{\small 
\begin{align} \label{eqefs3b}
LP_{3}&=
\left(
\begin{pmatrix}
\o \\ \o \\ \o
\end{pmatrix}
\begin{pmatrix}
\o \\ \e \\ \o
\end{pmatrix}^{+}
\left\{
\begin{pmatrix}
\e&\o \\ \o&\o \\ \o&\e
\end{pmatrix}
\begin{pmatrix}
\o \\ \e \\ \o
\end{pmatrix}
^{+}
\right\}
^{*}
\begin{pmatrix}
\o \\ \o \\ \o
\end{pmatrix}
\right)^t
\end{align}
}
But then, the row $W[h]=(\o,\o,\o)$ forces $W$ to contain a West bench with its seat in $W^t[3]$ at the bottom of $W$ contradicting the hypothesis.
\end{proof}

\begin{corollary}\label{corol:hequiv1,2}
For integers $h\geq 4$ such that $h\equiv 0,1 \pmod 3$, atomic 2-full words 
$W\in \MWordsdhwvarshort{2}{(h+1)}{3}$ of height $(h+1)$ are obtained by the insertion of a row $W[i]=(\o,\e,\o)$ adjacent to an already existing row  $W[i-1]=W'[i-1]=(\o,\e,\o)$ in an atomic word $W'\in \MWordsdhwvarshort{2}{h}{3}$ of height $h$.
\end{corollary}

\begin{proof}
This is a consequence of Proposition \ref{prop3xhmax_bench1} and the rule given by equation \eqref{eqefs3b}. 
\end{proof}

\begin{proposition} \label{prop3xhe(W[4])} 
Let $W\in \Wordsdhwvarshort{2}{h}{4}$ such that $W^t\DI{1,3}\in \MWordsdhwvarshort{2}{h}{3}$ is 
a 2-full atomic subword of $W$. Then there is no $3$-pillar in the interior of $W^t[4]$. 
\end{proposition}
\begin{proof} 
Assume $W\DI{i,i+2}$  is a $3$-pillar in the interior of  $W^t[4]$ (purple cells in Figure \ref{4xhc4_3pilier}). Then, up to symmetry and from Lemma \ref{lem3xhmax_r1r_2} the red cells are mandatory and $W[i+4]$ contains at most one cell in contradiction with Lemma \ref{lem3xhmax_r1r_2} $ii)$. 
\begin{figure*}[h!]
\centering
\begin{subfigure}[t]{0.2\textwidth}
\centering
\includegraphics[height=0.6in]{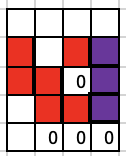}
\caption{}
\end{subfigure}%
\caption{ $3$-pillar in  $W^t[4]$}
\label{4xhc4_3pilier}
\end{figure*}
\end{proof}

\paragraph{\textbf{Proof of Corollary \ref{cor4xhn(c4)} $i)$}}

\begin{proof} Observe that if the statement is true for atomic words  $W^t\DI{1,3}\in \MWordsdhwvarshort{2}{h}{3}$ then it is also true for non atomic words $W^t\DI{1,3}\in \MWordsdhwvarshort{2}{h}{3}$. So we only need to prove the corollary for 2-full atomic words. For $i\in\{ 0,1,2\} $,Let $h=3k+i$.
Factorize $W$ as product of $t=2+\lfloor\frac{h-2}{3} \rfloor$ factors $W=W[1] W\DI{2,4}\cdots 
W\DI{h-4-(i+1\mod 3),h-1}W[h]$
such that  the interior factors have $3$ rows except possibly the penultimate factor which has $3+(i+1\mod 3)$ rows. 
From Proposition \ref{prop3xhe(W[4])} and Corollary \ref{lem3xhmax_r1r_2} $i)$, the intersection of each of the $t$ factors  with $W^t[4]$ contains an empty cell. Moreover, since there is a West bench on the right side of $W^t\DI{1,3}$, at least one intersection of an  interior factor with $W^t[4]$ must contain at least one more empty cell. The proof is complete.
\end{proof}

\begin{lemma}\label{lem:hx3_n+1bench}
Let $h\geq 3$ and $W\in \MWordsdhwvarshort{2}{h}{3}$ an atomic 2-full word.
If there are $n$ solitary vertical pillars in $W^t[3]$ 
then there are precisely $n+1$ West benches with their seat in $W^t[3]$.
\end{lemma}
\begin{proof}
Let $p_{1}$ be the first solitary vertical pillar  in  $W^t[3]$  from the top of $W$.
(purple cell in Figure \ref{3xhfssolpillara}). Then the three rows of $W$ below and above $p_{1}$  are mandatory  (red cells in Figure \ref{3xhfssolpillara}).  If there is a second solitary pillar $p_{2}$  in  $W^t[3]$ below $p_{1}$ then there must exist a West bench with its seat in $W^t[3]$ between $p_{1}$ and $p_{2}$. This is proved by inspecting each possible row distribution below the red cells:

a row $(\e,\o,\o)$ (purple cells in Figure \ref{3xhfssolpillara}) gives immediately a bench. 
A row $(\o,\o,\e)$ (purple cells in Figure \ref{3xhfssolpillarb}) 
is immediately followed by $(\e,\o,\o)$ and $(\o,\e,\o)$ (green cells in Figure \ref{3xhfssolpillarb})  which leads to a West bench with its seat on $W^t[3]$. A sequence of rows $(\o,\e,\o)$ (purple cells in Figure \ref{3xhfssolpillarc}) similarly lead to a West bench. 
With the same argument, there also exists a bench above $p_{1}$. The repetition of  this argument for every maximal pillar in $W^t[3]$ shows that there are at least  $n+1$ West benches with their seat in $W^t[3]$. \\
Now suppose that there are two consecutive benches with their seat in $W^t[3]$ (purple cells in Figure \ref{3xhconseqbencha}). The two rows below the first bench and the two rows above the second bench 
(red cells in Figure \ref{3xhconseqbencha}) are mandatory.
We claim that there is a solitary pillar between these two benches. We prove this claim by considering the three possible row distributions below the two top red rows (Figures \ref{3xhconseqbencha},\ref{3xhconseqbenchb},\ref{3xhconseqbenchc}). 
If all rows are $(\o,\e,\o)$ (purple cells in Figure \ref{3xhconseqbencha}), they form a solitary pillar. If the first row below the two mandatory red cells is $(\o,\o,\e)$  (purple cells in Figure \ref{3xhconseqbenchb}) then the  mandatory red row above it contains a solitary pillar. 
If the first row below the two mandatory red cells is $(\e,\o,\o)$  (purple cells in Figure \ref{3xhconseqbenchc}) then its two adjacent rows are $(\o,\o,\e), (\o,\e,\o)$ (green cells in Figure \ref{3xhconseqbenchc}) which starts a cyclic disposition that must end with a solitary pillar because of the second bench. 
This proves that benches with their seats in $W^t[3]$ and solitary pillars in $W^t[3]$ alternate  and  the number of benches in $W^t[3]$ is $n+1$.

\begin{figure*}[h!]
\begin{subfigure}[t]{0.15\textwidth}
\centering
\includegraphics[height=1. in]{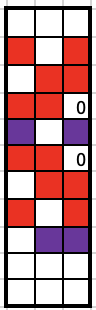}
\caption{}
\label{3xhfssolpillara}
\end{subfigure}%
~ 
\begin{subfigure}[t]{0.15\textwidth}
\centering
\includegraphics[height=1. in]{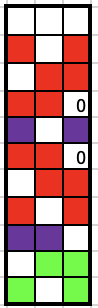}
\caption{}
\label{3xhfssolpillarb}
\end{subfigure}%
~ 
\begin{subfigure}[t]{0.15\textwidth}
\centering
\includegraphics[height=1. in]{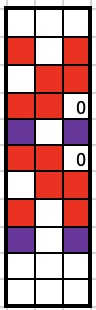}
\caption{}
\label{3xhfssolpillarc}
\end{subfigure}
\begin{subfigure}[t]{0.15\textwidth}
\centering
\includegraphics[height=1. in]{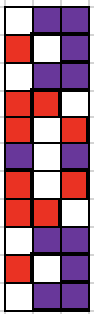}
\caption{}
\label{3xhconseqbencha}
\end{subfigure}%
~ 
\begin{subfigure}[t]{0.15\textwidth}
\centering
\includegraphics[height=1.1 in]{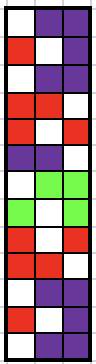}
\caption{}
\label{3xhconseqbenchb}
\end{subfigure}%
~ 
\begin{subfigure}[t]{0.15\textwidth}
\centering
\includegraphics[height=1.1 in]{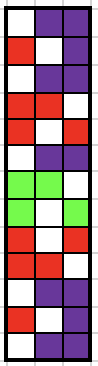}
\caption{}
\label{3xhconseqbenchc}
\end{subfigure}%
\caption{Solitary pillar and consecutive benches in $W^t[3]$}
\label{3xhfs3-bancvspillier}
\end{figure*}

\end{proof}

\paragraph{\textbf{Proof of Corollary \ref{cor4xhn(c4)} $ii)$ and $iii)$}}
\begin{proof}
We already know from Corollary \ref{cor4xhn(c4)} that for all $h\geq 3$, we have $|W[4]|_\e \geq\lfloor  \frac{h-2}
{3}\rfloor +3$. From Corollary \ref{corol:hequiv1,2} we know that $(h+i)\times 3,i\in\{ 1,2\}$ 2-full atomic words $W$ are obtained from $(h+i-1)\times 3$ 2-full atomic words $W'$ by inserting a row $(\o,\e,\o)$ in 
$W'$. We need to observe that the insertion of $(\o,\e,\o)$ in $W'$  imposes a supplementary empty cell in 
$W^t[4]$. Assume without loss of generality that $W^t\DI{1,3}[i]=(\o,\e,\o)$ is inserted below  $W'^t\DI{1,3}[i-1]=(\o,\e,\o)$. If $W'^t\DI{1,3}[i-2]=(\e,\o,\o)$ or $W'^t\DI{1,3}[i]=(\e,\o,\o)$
(purple cells in Figures \ref{hx3[101]a} and \ref{hx3[101]b})
then it is immediate that the insertion of  $(\o,\e,\o)]$ increases by one the number of cells of degree $2$ in $W'^t[3]$ and therefore, the number of empty cells in $W^t[4]$. If $W'^t\DI{1,3}[i-2]=(\o,\o,\e)$ and $W'^t\DI{1,3}[i]=(\o,\o,\e)$ then $W'[i-1,3]$ contains a solitary $1$-pillar (Figure \ref{hx3_101}) and 
from Lemma  \ref{lem:hx3_n+1bench}, there is a bench with its seat in $W'^t[3]$ above and below $W'[i]$ so that   the minimal number of  empty cells in $W'^t[4]$ increases by one. This completes the proof. 
\begin{figure*}
\centering 
\begin{subfigure}[t]{0.2\textwidth}
\centering
\includegraphics[height=0.4 in]{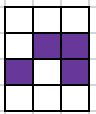}
\caption{}
\label{hx3[101]a}
\end{subfigure}%
~ 
\begin{subfigure}[t]{0.15\textwidth}
\centering
\includegraphics[height=.4 in]{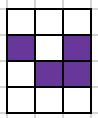}
\caption{}
\label{hx3[101]b}
\end{subfigure}%
~ 
\begin{subfigure}[t]{0.15\textwidth}
\centering
\includegraphics[height=.4 in]{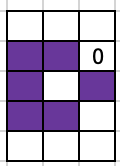}
\caption{}
\label{hx3_101}
\end{subfigure}%
~

\label{fig:hx3_101}
\caption{Insertion of $[1,0,1]$ in $W'^t\DI{1,3}$}
\end{figure*}

\end{proof}
\subsubsection{  $e(W^t[2])>0$.}
\label{secAut3.1}

Let $W\in \Wordsdhwvarshort{2}{h}{3}$  and consider a maximal $r$-pillar $W^t[2]\cap W\DI{j,j+r-1}$ as shown in Figure \ref{3xh_centr_pillar}. For $r\geq 2$, we have
\begin{align}\label{c1c2}
|W^t[1]\cap W\DI{j,j+r-1}|_\e + |W^t[3]\cap W\DI{j,j+r-1}|_{\e}=|W\DI{j,j+r-1}|_\e \geq 2(r-1)
\end{align}
(i.e. there are at least $2(r-1)$ empty cells in $W\DI{j,j+r-1}$ located in columns $W^t[1]$ and $W^t[3]$.)
\begin{figure*}[h!]
\begin{subfigure}[t]{0.5\textwidth}
\centering
\includegraphics[height=0.7 in]{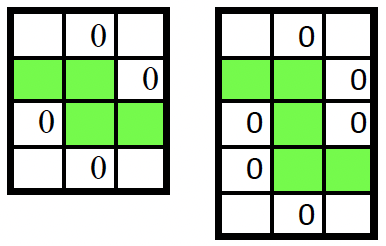}
\caption{} \label{3xh_centr_pillar}
\end{subfigure}%
\begin{subfigure}[t]{0.2\textwidth}
\centering
\includegraphics[height=.9 in]{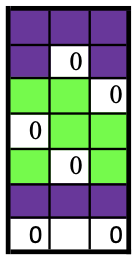}
\caption{}
\label{3x3km2}
\end{subfigure}%
\caption{$r$-pillar in the central column $W^t[2]$ }
\label{3xhcpill}
\end{figure*}

The sequence of lengths of maximal pillars in a  column $W^t[j]$ of a word $W$ read from top to bottom 
forms a composition of the integer $|W^t[j]|_\o$ and we are interested in the corresponding partition obtained from that composition denoted $\lambda(W^t[j])$. In the following, we will use the additive notation for partitions 
$\lambda=\lambda_1,\lambda_2,\ldots,\lambda_k$ with $\lambda_1\geq \lambda_2\geq, \ldots, \lambda_k$
and $\lambda_1+ \lambda_2,+\ldots + \lambda_k=|W^t[j]|_\o$ as well as the multiplicative notation 
$\lambda=1^{m_1}2^{m_2}\cdots n^{m_n}$ where $m_i$ is the number of occurrences of $i$ in $\lambda$. Also 

$$
\ell(\lambda)=\sum_i m_i
$$
is called the length of $\lambda$ and
$$
|W^t[j]|_\o=\sum_i im_i
$$
is called the weight of $|W^t[j]|_\o$.
\begin{lemma}
Let $W\in \Wordsdhwvarshort{2}{h}{3}$   and let $\lambda(W^t[2])=1^{m_1}2^{m_2}\cdots $
be the partition obtained from the distribution of maximal pillars in $W^t[2]$. Then 
\begin{align}
\label{no(c1)n0(c3)}
|W^t[1]|_\e +|W^t[3]|_\e
\geq \sum_{i\geq 2} 2(i-1)m_i+2max\{0,\lfloor \frac{m_1-1}2\rfloor\}
\end{align}
\end{lemma}
\begin{proof}
The first part of the right-hand side of  \eqref{no(c1)n0(c3)} is the result of inequality \eqref{c1c2}. The second part comes from the fact that when $m_1>2$, there are $m_1-2$ maximal $1$-pillars in the interior of $W^t[2]$. Every pair of adjacent maximal $1$-pillars   and every non adjacent maximal $1$-pillars in the interior of $W^t[2]$ is adjacent orthogonally or diagonally to at least one empty cell in each column $W^t[1],W^t[3]$ so that there are at least 
$2max\{0,\lfloor \frac{m_1-1}2\rfloor\}$ empty cells produced by the maximal $1$-pillars.
\end{proof}

\begin{lemma}\label{lem:ellc2}
For any integer  $h\geq 1$ and $W\in \Wordsdhwvarshort{2}{h}{1}$  

We have 
\begin{align}
\ell(\lambda(W^t))\leq 
min\{ |W^t|_\e +1,|W^t|_\o\} 
\end{align}
\end{lemma}
\begin{proof}
This inequality is the result of the observation that the maximum length $\ell(\lambda(c))$ is upper bounded by the number of empty cells in $W^t$ plus one unless there are not enough filled cells in which case $\ell(\lambda(c))$ is bounded by the number of filled cells. 

\end{proof}

\begin{proposition}\label{propc1c2c3e(c2)>0}
Let $h=3k+1$ and $W\in \Wordsdhwvarshort{2}{h}{3}$. Then
\begin{align*}
e(W^t[1])+e(W^t[3]) \leq 2(-k -1/3+ \ell(\lambda(W^t[2]))-e(W^t[2])).
\end{align*}

\end{proposition}
\begin{proof}
Let $\lambda(W^t[2]))=\lambda_1\geq \lambda_2, \ldots , =1^{m_1}2^{m_2}\cdots $. 
Recall that
$$\sum_{i\geq 1}im_i=|W^t[2]|_\o =2k+2/3+e(W^t[2])$$
by definition of excess  and 
$$\sum_{i\geq 1}m_i=\ell(\lambda(W^t[2])).$$
We have 
\begin{align*}
\sum_{i\geq 1} im_i -2\sum_{i\geq 1} m_i&=-m_1+\sum_{i\geq 2} (i-2)m_i,\\
\Rightarrow m_1&=\sum_{i\geq 2} (i-2)m_i -2k-2/3-e(W^t[2])+2\ell(\lambda(W^t[2]))
\end{align*}
From \eqref{c1c2} we know
\begin{align*}
|W^t[1]|_\e  + |W^t[3]|_\e &\geq \sum_{i\geq 1}2(i-1)m_i. \\
&\geq 2\left( \sum_{i\geq 2} (i-2) m_i +\sum_{i\geq 2} m_i \right),\\
\Rightarrow |W^t[1]|_\e +|W^t[3]|_\e &\geq 
2\left(  m_1+2k+2/3-2\ell(\lambda(W^t[2]))+e(W^t[2])+\sum_{i\geq 2} m_i\right),\\
& \geq  2\left( 2k+2/3 -\ell(\lambda(W^t[2]))+e(W^t[2])\right),\\
\end{align*}
Since 
\begin{align*}
|W^t[1]|_\o +|W^t[3]|_\o &= 2(3k+1)-(|W^t[1]|_\e +|W^t[3]|_\e),\\
& \leq 2k+2/3 +2\ell(\lambda(W^t[2]))-2e(W^t[2]),
\end{align*}
We obtain
\begin{align*}
|W^t[1]|_\o +|W^t[3]|_\o \leq &2k+2/3+2\ell(\lambda(W^t[2]))-2e(W^t[2]),\\
\Rightarrow e(W^t[1])+e(W^t[3])\leq &2\left( -k-1/3+\ell(\lambda(W^t[2]))-e(W^t[2])\right) 
\end{align*}
which is what we wanted.
\end{proof}

\begin{proposition}\label{prop3xhlambda(c2)}
Let  $W\in \Wordsdhwvarshort{2}{3}{3k+1}$  
with $e(W[1])=1/3,e(W[2])=1/3,e(W[3])=1/3$. Then 
$$\lambda(W[2])\in\{2^k1,\; 32^{k-2}1^2 \}.$$
\end{proposition}
\begin{proof}
From the hypothesis, Lemma \ref{lem:ellc2} and Proposition \ref{propc1c2c3e(c2)>0}, we obtain that $\ell(\lambda(W[2]))=k+1$ which implies that $m_1\geq 1$. \\
If $m_1>2$ then there is at least one  maximal $1$-pillar in the interior of $W[2]$ which imposes more empty cells in $W[1] \cup W[3] $

so that the hypothesis $e(W[1])=1/3,e(W[2])=1/3,e(W[3])=1/3$ becomes impossible. Therefore $1\leq m_1\leq 2$. If $m_1=1$ then $\lambda=2^k1$ and if $m_1=2$ then $\lambda=32^{k-2}1^2$.

\end{proof}

\begin{proposition} \label{prop3xhmax-1_n(ri)} 
Let $W\in \Wordsdhwvarshort{2}{3}{h}$ be atomic with $e(W)=1$.  
$$|W^t[1]|_\o = 2\Rightarrow \forall \; i:\; 2\leq i\leq h-1,\; |W^t[i]|_\o=2 \mbox{ and } |W^t[h]|_\o=3$$
\end{proposition}
\begin{proof}
By induction. Assume that for all $2\leq i\leq r<h-1$, we have $|W^t[i]|_\o=2$. Then 
\begin{align*}
|W^t[r+1]|_\o = 1 \Rightarrow e(W^t\DI{1,r+1}) \leq -1 &\Rightarrow e(W^t\DI{r+2,h})\geq 2
\end{align*}
which, from Proposition \ref{prop3xhfssize} is in contradiction with $|W^t[r+1]|_\o = 1$.  Moreover
\begin{align*}
|W^t[r+1]|_\o = 3 \Rightarrow |W^t[r+2]|_\o= 0
\end{align*}
in contradiction with the fact that $W$ is atomic.
So we must have $|W^t[r]|_\o=2$ for all $r\leq h-1$ and $|W^t[h]|_\o=3$.
\end{proof}

\begin{proposition}\label{prop3xh2^k1}
Let $W\in \Wordsdhwvarshort{2}{5}{(3k+1)}$ with $e(W[2])=1/3,e(W[3])=1/3,e(W[4])=1/3$ and
$\lambda(W[3])=2^k1$. Then $e(W[1])<-2/3 \mbox{ or }e(W[5])<-2/3$.
\end{proposition}
\begin{proof}
If $\lambda(W[3])=2^k1$ and $e(W[2])=e(W[3])=e(W[4])=1/3$ then there is an East $3$-bench at the left or a  West $3$-bench at the right of $W\DI{2,4}$. Assume without loss of generality that  there is  a West $3$-bench at the right of  $W\DI{2,4}$ (Figure \ref{3x(3k+1)2_k1a}) so that  from the hypothesis $\lambda(W[3])=2^k1$, 
the cells $W[3,1]$ and $W[3,2]$ are filled and 
we have
$|W^t[1]\cap W\DI{2,4}|_\o=2$  which implies, from Proposition \ref{prop3xhmax-1_n(ri)}, that every interior column $W^t[i]\cap W\DI{2,4}$ of $W$ has two $\o$ cells on $W\DI{2,4}$. This also implies that there is a South or North $3$-bench at the right of $W$. Assume without loss of generality that there is a North $3$-bench at the right of $W$ with its seat in $W[4]$ (Figure \ref{3x(3k+1)2_k1a}). 

We claim that $|W[5]|_\e \geq k+2$. Starting on the right, factorize $W\DI{2,4}$ as product of one $3\times 2$  
factor plus $k-1$ $3\times 3$ factors plus one $3\times 2$  rectangle at the right. $W[4]$ contains  at least one cell of degree $2$ in each  $3\times 3$ rectangle because it contains a $2$-pillar in $W[3]$. $W[4]$ contains two other cells of degree two in the right $3\times 2$ factor and either one cell of degree two in the left $3\times 2$  factor (Figure \ref{3x(3k+1)2_k1b}) or a maximal $1$-pillar in the interior of $W[4]$ that generates a cell of degree zero in $W[5]$ by Lemma \ref{lem:hx3_n+1bench} (Figure \ref{3x(3k+1)2_k1c}) so that $|W[5]|_\e\geq (k-1)+2+1=k+2$. This proves the proposition. 
\begin{figure*}[h!]
\centering
\begin{subfigure}[t]{0.3\textwidth}
\centering
\includegraphics[height=.6in]{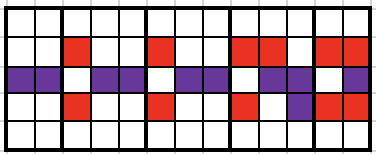}
\caption{}
\label{3x(3k+1)2_k1a}
\end{subfigure}%
~ 
\begin{subfigure}[t]{0.3\textwidth}
\centering
\includegraphics[height=.6in]{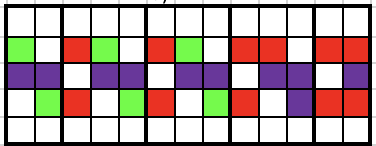}
\caption{}
\label{3x(3k+1)2_k1b}
\end{subfigure}%
~ 
\begin{subfigure}[t]{0.3\textwidth}
\centering
\includegraphics[height=.6in]{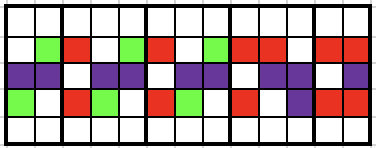}
\caption{}
\label{3x(3k+1)2_k1c}
\end{subfigure}%
~ 
\caption{  $\lambda(W[2])=2^k1$}
\label{3xhc1c2c32^k1}
\end{figure*}
\end{proof}

\begin{proposition}\label{prop3xh_32^(k-2)1^2}
Let $W\in \Wordsdhwvarshort{2}{5}{(3k+1)}$ with $e(W[2])=e(W[3])=e(W[4])=1/3$ and
$\lambda(W[3])=32^{k-2}1^2$. Then $e(W[1])<-2/3 \mbox{ or }e(W[5])<-2/3$.
\end{proposition}
\begin{proof}
If $\lambda(W[3])=32^{k-2}1^2$ and $e(W[2])=e(W[3])=e(W[4])=1/3$ then the two maximal $1$-pillars are at the left and right of $W\DI{2,4}$ because otherwise if a $3$-pillar is, say, at the left side of $W[3]$ then the first and second columns of $W\DI{2,4}$ each contain at most one cell in contradiction with Proposition \ref{prop3xhmax-1_n(ri)}.
There is an East $3$-bench at the left and a West $3$-bench at the right of $W\DI{2,4}$ (Figure \ref{3xhc1c2c3_32^(k-2)1^2}). 
There is an North or South $3$-bench  at the left and at the right of $W\DI{2,4}$. 
Assume without loss of generality that there is a North $3$-bench at the right of $W\DI{2,4}$ (Figure \ref{3xhc1c2c3_32^(k-2)1^2}). 
We claim that $|W[5]|_\e \geq k+2$. Starting at the right, factorize $W$ as the product of one $5\times 2$  
factor plus $k-1$ $5\times 3$ factors plus one $5\times 2$  factor at the left. $W[4]$ contains at least one cell of degree $2$ in each  $5\times 3$ factor because  $W[3]$  contains a $2$-pillar or a $3$-pillar in each factor. $W[4]$ contains  two cells of degree two in the right $5\times 2$ factor and one cell of degree two in the left $5\times 2$  factor (Figure \ref{3xhc1c2c3_32^(k-2)1^2}). This yields that the number of degree 2 cells in $W[4]$ is greater or equal to $k+2$  which implies  $|W[5]|_\e \geq k+2$. This proves the proposition. 
\begin{figure*}[h!]
\centering
\begin{subfigure}[t]{0.4\textwidth}
\centering
\includegraphics[height=.6in]{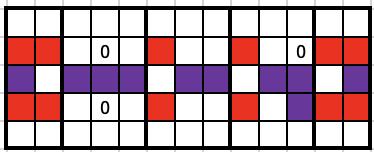}
\caption{}
\end{subfigure}%
~ 
\caption{  $\lambda(W[3])=32^{k-2}1^2$}
\label{3xhc1c2c3_32^(k-2)1^2}
\end{figure*}
\end{proof}

\subsubsection{Properties of words $W\in \MWordsdhwvarshort{2}{h}{4}$.}
\begin{proposition}\label{prop:Rmaxn(r1)}Let $W\in \MWordsdhwvarshort{2}{h}{4}$. 
\begin{align}
h\equiv_3 1 &\Rightarrow |W[1]|_\o \geq 3 \mbox{ and } |W\DI{1,2}|_\o \geq 6,\\
h\equiv_3 0 &\Rightarrow |W[1]|_\o\geq 3 \mbox{ and } |W\DI{1,2}|_\o\geq 5.
\end{align}
\end{proposition}
\begin{proof}
Assume that $h\equiv_3 1$.
If $|W[1]|_\o\leq 2$, then $e(W[1])\leq -2/3$, which implies $e(W\DI{2,h})\geq 4/3+2/3$, contradiction with Proposition \ref{prop4xhsc}.
If $|W\DI{1,2}|_\o\leq 5$, then $e(W\DI{1,2})\leq -1/3$, which implies $e(W\DI{3,h})\geq 4/3+1/3$, contradiction with Proposition \ref{prop4xhsc}.

Now, assume that $h\equiv_3 0$.
If $W[1]|_\o \leq 2$, then $e(W[1])\leq -2/3$, which implies $e(W\DI{2,h} )\geq 1+2/3$, contradiction with Proposition \ref{prop4xhsc}.
If $|W\DI{1,2}|_\o \leq 4$, then $e(W\DI{1,2})\leq -4/3$, which implies $e(W\DI{3,h})\geq 1+4/3$, contradiction with Proposition \ref{prop4xhsc}.
\end{proof}

\begin{proposition} \label{propleq3geq2} 
Let $W\in \MWordsdhwvarshort{2}{h}{4}$ and $1<i<h$.
\begin{align*}
a)\; h\equiv_3 0,1 &\Rightarrow |W[i]|_\o\leq 3, \\
b) \; h\equiv_3 1 &\Rightarrow  |W[i]|_\o\geq 2,
\end{align*}
\end{proposition} 
\begin{proof}
$a)$ Suppose  $1<i<h$ and $|W[i]|_\o= 4$. Then $|W\DI{i-1,i+1}|_\o\leq 6$ and 
$e(W\DI{i-1,i+1})\leq -2$ . Factor $W=W_tW\DI{i-1,i+2}W_b$ as product of three factors. Then 
\begin{align*}
h\equiv_3 1 &\Rightarrow  e(W_t)+ e(W_b) \geq 4/3+2 \mbox{ which is impossible},&\\
h\equiv_3 0 &\Rightarrow  e(W_t)+ e(W_b) \geq 1+2 \mbox{ which is impossible},&
\end{align*}
$b)$ Suppose $|W[i]|_\o \leq 1$ implying $e(W[i])\leq -5/3$. Factor $W=W_tW[i]W_b$ as product of three factors. We have
$$e(W_t)+ e(W_b) \geq 4/3+5/3$$
which is impossible and the proof is complete. 
\end{proof}
\begin{proposition}\label{prop:hx4cycle} Let $W\in \MWordsdhwvarshort{2}{h}{4}$.
The unique $3\times 4$ 2-full word of area $10$ is a cycle and\\
$a)$  $h\equiv_3 0,1 \Rightarrow |W\DI{1,3}|_\o \leq 9$.\\
$b)$ There is no 2-full $3\times 4$ factor of area $10$  in the interior of $W$
\end{proposition} 
\begin{proof}
$a)$ Let  $|W\DI{1,3}|_\o =10$. Then $|W[4]|_\o=0$ and $e(W\DI{1,4})=-2/3$ so that 
\begin{align*}
h\equiv_3 1 &\Rightarrow e(W\DI{5,h} )=4/3+2/3 \mbox{ which is impossible},\\
h\equiv_3 0 &\Rightarrow e(W\DI{5,h} )=1+2/3 \mbox{ which is impossible},
\end{align*}
$b)$ Let  $|W\DI{i,i+2}|_\o =10$ then $|W[i-1]|_\o =|W[i+3]|_\o =0$ and  $e(W\DI{i-1,i+3})=-10/3$. 
Factorize $W=W_t W\DI{i-1,i+3} W_b$, as product of three factors. Then 
\begin{align*}
h\equiv_3 1 &\Rightarrow e(W_t)+e(W_b) = 4/3+10/3 \mbox{ which is impossible},\\
h\equiv_3 0 &\Rightarrow e(W_t)+e(W_b) = 1+10/3 \mbox{ which is impossible},\\
h\equiv_3 2 &\Rightarrow e(W_t)+e(W_b) = 2/3+10/3 \mbox{ which is impossible}. 
\end{align*}
\end{proof}

\begin{proposition}\label{prop:Rmaxequiv1bench} Let $k\in \{3,4 \}$ and $W\in \MWordsdhwvarshort{2}{h}{4}$  with $h\equiv_3 1$. 
Then the $3\times 4$ factors at the bottom and top of $W$ have $9$ filled cells and there is a North $k$-bench at the bottom and a South $k$-bench at the top of $W$. 
\end{proposition}
\begin{proof}
By minimal counterexample. We know that the claim is true for $h\in\{ 4,7\}$. Let $h$ be minimal such that  $|W\DI{1,3}|_\o \leq 8$ or ($|W\DI{1,3}|_\o=9$ and  $W\DI{1,3}$ has no South $k$-bench at the top of $W$). If $|W\DI{1,3}|_\o =9$ and  $W\DI{1,3}$ has no south $k$-bench (Figure \ref{3x4n=9nobench})  then by inspection,  $|W[4]|_\o \leq 1$ and $W\DI{1,4}|_\o \leq 10, e(W\DI{1,4})\leq -2/3\Rightarrow  e(W\DI{5,h})\geq 4/3+2/3$ which is impossible for $h\geq 10$.\\
If $|W\DI{1,3}|_\o \leq 8$ then $|W\DI{1,3}|_\o= 8$ for otherwise $e(W\DI{4,h})>4/3$ which is impossible. But then $W\DI{4,h}$ is 2-full and by minimality hypothesis  there is a South $k$-bench at the top of  $W\DI{4,h}$ which implies that $|W[3]|_\o \leq 1$ in contradiction with $|W\DI{1,3}|_\o =8$. The proof is complete

\begin{figure*}[h!]
\begin{subfigure}[t]{.45\textwidth}
\centering
\includegraphics[height=0.4in]{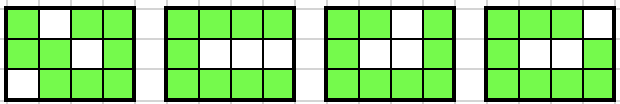}
\caption{$3\times 4$ words of area 9 with no south bench }
\label{3x4n=9nobench}
\end{subfigure}%
~
\begin{subfigure}[t]{.45\textwidth}
\centering
\includegraphics[height=0.5in]{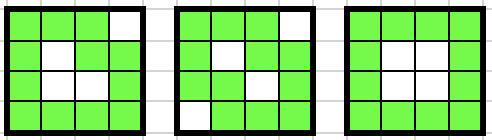}
\caption{$4\times 4$ words of area 12}
\label{4x4n=12}
\end{subfigure}%
\caption{}
\label{3x4_4x4}
\end{figure*}
\end{proof}

\begin{corollary}
Let $W\in \MWordsdhwvarshort{2}{h}{4}$ with $h\equiv_3 1$. 
Then $|W[4]|_\o =2$.
\end{corollary}
\begin{proof} Observe that $|W\DI{1,4}|_\o = 12$ is impossible because then $|W[5]|_\o \leq 1$ (Figure \ref{4x4n=12}), $e(W\DI{1,5}) \leq  -1/3$ and  $h(W\DI{6,h}) \equiv_3 2$ so that $e(W\DI{6,h}) \geq 4/3+1/3$ which is impossible.\\
Since $|W\DI{1,3}|_\o =9$ from Proposition \ref{prop:Rmaxequiv1bench} and $|W\DI{1,4}|_\o \leq 11$, 
if $|W[4]|_\o \leq 1$ then $e(W\DI{1,4} )\leq -2/3$ which implies $e(W\DI{5,h})\geq 4/3+2/3$ which is impossible for $h\geq 10$. So we must have $|W[4]|_\o =2$. 
\end{proof}

\begin{proposition}\label{prop:n=9int} Let $W\in \MWordsdhwvarshort{2}{h}{4}$. Then
\begin{align*}
a) &\;  h\equiv_3 1,0  \Rightarrow  \mbox{ there is no $3\times 4$ factor of area $9$ in the interior of }W,\\
b) &\; h\equiv_3 1 \Rightarrow \mbox{ there is no $3\times 4$ factor of area $7$ or less in the interior of }W
\end{align*}
\end{proposition}

\begin{proof}
a) By minimal counterexample. Assume that $|W\DI{i,i+2}|_\o =9$ for $1<i<h$. By inspection, we observe that $|W\DI{i-1,i+3}|_\o \leq 12$ so that  $e(W\DI{i-1,i+3})\leq -4/3$. Factorize $W=W_tW\DI{i-1,i+3}W_b$ as product of three. factors. Then
\begin{align}\label{eqh=1n=9}
h\equiv_3 1 \Rightarrow e(W_t)+e(W_b)\geq 4/3+4/3 
\end{align}
and since $h(W_t)+h(W_b)\equiv_3 2$, we have $(h(W_t), h(W_b))\equiv_3 (1,1) \mbox{ or }(0,2)$. 
We discard $(0,2)$ because then $e(W_t)>\emax(h(W_t),4)$ or $e(W_b)>\emax(h(W_b),4)$ in contradiction with the minimal hypothesis. 
If $(h(W_t),h(W_b))\equiv_3 (1,1)$ then $W_t$ and $W_b$ are both 2-full and from Proposition \ref{prop:Rmaxequiv1bench} they both have benches adjacent to $W[i-1]$ and $W[i+4]$ and 
$$|W\DI{i-1,i+3}|_\o =12 \Rightarrow |W[i-1]|_\o \geq 2 \mbox{ or } |W[i+4]|_\o \geq 2$$
in contradiction with the fact that a row adjacent to a $3$-bench in $W$ has at most one filled cell. Now
\begin{align}\label{eqhequiv0}
h\equiv_3 0 \Rightarrow  e(W_t)+e(W_b)\geq 1+4/3 
\end{align}
and since $h(W_t)+h(W_b)\equiv_3 1$, we have $(h(W_t),h(W_b))\equiv_3 (2,2) \mbox{ or }(0,1)$. 
But $(h(W_t),h(W_b))\equiv_3 (2,2) \Rightarrow e(W_t)+e(W_b)\leq 2/3+2/3$
in contradiction with \eqref{eqhequiv0}. Also  
$(h(W_t),h(W_b))\equiv_3 (0,1) \Rightarrow e(W_t)+e(W_b)\leq 1+4/3$. $W_t$ and $W_b$ are then both 2-full and from Proposition \ref{prop:Rmaxequiv1bench} they both have benches adjacent to $W[i-1]$ and $W[i+4]$ which is again impossible. \\
b) Assume that $h\equiv_3 1$ and $|W\DI{i,i+2}|_\o \leq 7$ so that $e(W\DI{i,i+2})\leq -1$. Let $W=W_tW\DI{i,i+2}W_b$. Then  
\begin{align}
\label{eq_1hequiv1} e(W_t)+e(W_b)\geq 4/3+1
\end{align}
and $(h(W_t),h(W_b))\equiv_3 (0,1) \mbox{ or } (2,2)$. If $(h(W_t),h(W_b))\equiv_3 (2,2)$ then 
$e(W_t)+e(W_b)\geq 2/3+2/3$ in contradiction with \eqref{eq_1hequiv1}. So we must have $(h(W_t),h(W_b))\equiv_3 (0,1)$ and $W_b$ is 2-full with $h(W_b)\equiv_3 1$ so that $W_b$ has a South bench at its top and $|W[i+2]|_\o\leq 1$ in contradiction with Proposition \ref{propleq3geq2}. The proof is complete.

\end{proof}

\begin{proposition}\label{prop:h=3k+1n(ri)} 
For integers $k\geq 2$,  $h=3k+1$ and $W\in \MWordsdhwvarshort{2}{h}{4}$ we have 
\begin{align*}\label{}
a)\; \forall t=1\dots k-1: \;  |W[3t+1]|_\o =2, \\
b)\;  \forall t=1\dots k-1: \;  |W[3t+2]|_\o =3,\\
c) \; \forall t=1\dots k-2: \;  |W[3t+3]|_\o =3.
\end{align*}
\end{proposition}

\begin{proof}
by induction on $t$. The statement is verified by inspection for $k=2$. Assume that the statement is true for 
$k\geq 3$ and $t<k-1$. \\
a)  We have that $|W[3t+1]|_\o=3$ is impossible because we would then have 
$|W\DI{3t-1,3t+1}|_\o =9$ in contradiction with Proposition \ref{prop:n=9int}. 
Since $|W[3t+1]|_\o\leq 1$ is also impossible from Proposition \ref{propleq3geq2}, we must have  $|W[3t+1]|_\o=2$. \\
b) If  $|W[3t+2]|_\o \leq 2$ then $|W\DI{3t+3,h}|_\o \geq 4/3+1/3$ which is impossible. Since $|W[3t+2]|_\o =4$ is also impossible from Proposition \ref{propleq3geq2}, we must have $|W[3t+2]|_\o=3$. \\
c) If  $|W[3t+3]|_\o \leq 2$ then $|W\DI{3t+4,h}|_\o \geq 4/3$ which is possible only when $3t+3=3k$. Otherwise we must also have  $|W[3t+3]|_\o =3$ and the proof is complete.
\end{proof}

\begin{proposition}\label{prop:2x2bottom} 
Let $h\equiv_3 1$ and $W\in \MWordsdhwvarshort{2}{h}{4}$. There is no $2\times 2$ factor of area $4$ at the bottom and top of $W$.
\end{proposition}
\begin{proof}
Suppose that there is a $2\times 2$ square at a top corner of $W$. Then $|W\DI{1,3}|_\o\leq 8$  in contradiction with Proposition \ref{prop:Rmaxequiv1bench}. 
If there is a $2\times 2$ square at the top center of $W$, then $|W\DI{1,3}|_\o \leq 6$  in contradiction with Proposition \ref{prop:Rmaxequiv1bench}. 
\end{proof}

\begin{proposition}\label{prop:2x2int} Let $k\geq 2,\; h=3k+1$ and $W\in \MWordsdhwvarshort{2}{h}{4}$. Then
\begin{align*}
a) \; k\geq 3 &\Rightarrow \nexists \; 2\times 2 \mbox{ inner factor of $W$ of area $4$ in columns } W^t\DI{2,3},\\
b)\; k\geq 4 &\Rightarrow \nexists \;2\times 2 \mbox{ inner factor  of $W$ of area $4$ in columns } W^t\DI{1,2}.
\end{align*}
\end{proposition}
\begin{proof}
a) Suppose that there is a $2\times 2$ factor at the intersection of columns $W^t\DI{2,3}$ and rows $W\DI{i,i+1}, 1<i<h-1$. Then $|W\DI{i-1,i+2}|_\o \leq 8$ in contradiction with Proposition \ref{prop:h=3k+1n(ri)}. \\
b)  Suppose that $k\geq 4$ and  there is a $2\times 2$ factor at the intersection of  columns  $W^t\DI{1,2}$ 
and rows $W\DI{i,i+1}, 1<i<h-1$. Then $|W\DI{i-1,i+2}|_\o \leq 10$ and in fact $|W\DI{i-1,i+2}|_\o =10$ so that  $e(W\DI{i-1,i+2})=-2/3$. But then $|W[i+2]|_\o=2$,  $|W[i+3]|_\o=3$ and  $|W[i+4]|_\o\leq 2$ (Figure \ref{hx4_2x2}) in contradiction with Proposition
\ref{prop:h=3k+1n(ri)}.
\begin{figure*}[h!]
\centering
\includegraphics[height=0.5in]{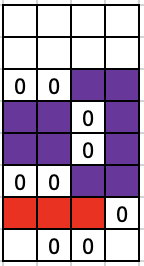}
\caption{$2\times 2$ inner square}
\label{hx4_2x2}
\end{figure*}
\end{proof}

\begin{proposition}\label{prop:4xhbench} Let $k\geq 2$,   $h=3k+1$ and $W\in \MWordsdhwvarshort{2}{h}{4}$. Then for some $t_1,t_2\geq 3$, $W$ contains at least one East $t_1$-bench with its seat on column $W^t[1]$  and one West $t_2$-bench with its seat on column $W^t[4]$. 
\end{proposition}
\begin{proof}
Assume that $W$ contains no East $t_1$-bench with its seat on column $W^t[1]$.  Then, since $W\DI{1,3}$ contains either a West $3$-bench or an East $3$-bench, there is a West $3$-bench with its seat on 
$W^t[4]$ (Figure \ref{4xhbenchtop}). 
Knowing that $|W[4]|_\o =2$ and that there is, up to symmetry, a unique disposition of two adjacent rows of $3$ cells each  which appear in Figure \ref{4x2(3,3)}, we have three choices for $W[4]$ : one is discarded  (Figure \ref{4x4r_4imp} because it cannot be adjacent to Figure \ref{4x2(3,3)}. 
The second is also discarded because it leads to an East bench with its seat in $W^t[1]$. 
Only the third choice remains (Figure \ref{4xhbenchtop}) and this choice must be repeated periodically because it is always adjacent on its bottom to a row of unique degree distribution. But this unique $3\times 4$ tile that is stacked on the bottom of itself to avoid an East bench with its seat on $W^t[1]$ is sentenced to terminate with an East bench with its seat on $W^t[1]$ at the bottom of $W$ as shown in Figure \ref{4xhbenchtop}.
\begin{figure*}[h!]
\begin{subfigure}[t]{0.2\textwidth}
\centering
\includegraphics[height=0.16in]{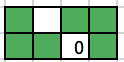}
\caption{ }
\label{4x2(3,3)}
\end{subfigure}%
~ 
\begin{subfigure}[t]{0.24\textwidth}
\centering
\includegraphics[height=0.32in]{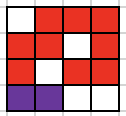}
\caption{ $W[4]$ impossible}
\label{4x4r_4imp}
\end{subfigure}%
~ 
\begin{subfigure}[t]{0.24\textwidth}
\centering
\includegraphics[height=0.48in]{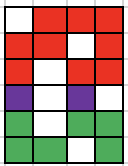}
\caption{  bench on the left}
\label{4x6bench}
\end{subfigure}%
~ 
\begin{subfigure}[t]{0.24\textwidth}
\centering
\includegraphics[height=.88in]{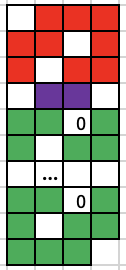}
\caption{  bench at the top}
\label{4xhbenchtop}
\end{subfigure}%
\caption{}
\label{4xhbench}
\end{figure*}
\end{proof}

\paragraph{\textbf{Proof of Corollary \ref{cor4xh_n2(c1)}}}

\begin{proof}
This is a consequence of he fact that when reading rows of $W$ starting from the top and 
$$ W=W_1W_2\cdots W_k$$
is factorized as product of $k-1$ $4\times 3$ rectangles plus one $4\times 4$ factor at the bottom of $W$,
then from Proposition \ref{prop:h=3k+1n(ri)} and Figure \ref{4x2(3,3)}, in each intersection  $W^t[1]\cap W_i$  there is  at least one cell of degree $2$. So there are at least $k$ cells of of degree $2$ in $W^t[1]$.  Moreover since, from Proposition \ref{prop:4xhbench}, there is at least one $t$-bench, $t\geq 3$ with its seat in $W^t[1]$,  one of the rectangles $W_i$ must contain $3$ cells of degree $2$. This proves the claim. 
\end{proof}

\subsubsection{Properties of words $W\in \MWordsdhwvarshort{2}{h}{5}$.}

\begin{proposition}\label{prop:5xhunicity}
For integers $k\geq2$,	2-full words $W \in \Wordsdhwvarshort{2}{(3k+2)}{5}$ are unique up to symmetry.
\end{proposition}
\begin{proof}
We know from Corollary \ref{cor:5xhTile} that
$$e(W\DI{4,6}) = e(W\DI{7,9}) = \dots = e(W\DI{3(k-2)+1,3(k-1)}) = 0.$$
and
$$|W\DI{4,6}|_{\o} = |W\DI{7,9}|_{\o} = \dots = |W\DI{3(k-2)+1,3(k-1)}|_{\o} = 10 $$
From Corollary \ref{cor:5xhTileH}, 
we also know the exact configuration of $W\DI{1,3}$. In particular, we know $W[3]$ has two  cells of degree $2$
either on its 1\textsuperscript{st} and 4\textsuperscript{th} cells or on its 2\textsuperscript{nd} and 5\textsuperscript{th} cells. Assume without loss of generality that it is on its 1\textsuperscript{st} and 4\textsuperscript{th} cells (Figure \ref{fig:5xhnbcell}). This means that $|W[4]|_{\o} \leq 3$. But $|W[4]|_{\o}<2$ implies $|W\DI{5,6}|>8$ which is impossible. If $|W[4]|_{\o} =2$ then $|W\DI{5,6}|_{\o} = 8$. Figure \ref{fig:5xh2cellsr4} shows that no such configuration is possible. So we must have $|W[4]|_{\o} = 3$ and $W[4]$ must have a degree 2 cell on it's 2\textsuperscript{nd} cell (Figure \ref{subfig:5xhfirst4}). This means $|W[5]|_{\o} \leq 4$. If $|W[5]|_{\o} =4$, then $W[5]$ has 3 degree 2 cells which implies $|W[6]|_{\o} \leq 2$ however in order to have $|W\DI{4,6}|_{\o} = 10$, we need $|W[6]|_{\o}=3$ so  $|W[5]|_{\o} =4$ is impossible. If $|W[5]|_{\o}<2$ then  $|W[6]|_{\o}>5$ which is obviously impossible. If $|W[5]|_{\o}=2$ then  $|W[6]|_{\o}=5$ which is only possible if the 2 $\o$ cells of $W[5]$ are at both ends (Figure \ref{subfig:5xh5bench}) so that $W[6]$ has 5 degree 2 cells which forces $W[7]$ to be empty. If $7 = 3(k-1)+1$, then from Corollary \ref{cor:5xhTile} we know that $e(W\DI{7,11}) = 1/3$ which is impossible if $W[7]$ is empty. If $7 < 3(k-1)+1$, then $e(W\DI{7,9}) = 0$ is  impossible if $W[7]$ is empty. Thus we cannot have $|W[5]|_{\o}=2$. This means $|W[5]|_{\o}=3$ which in part 
means that $|W[6]|_{\o}=4$. Figure \ref{fig:5xhtuiles} shows all possible configurations for $W\DI{5,6}$. The configuration on Figure \ref{subfig:5xhtuilesA} is rejected as it forces $W[7]$ to have at most 1 $\o$ cell which is impossible. Indeed, if $7 = 3(k-1)+1$, then from Corollary \ref{cor:5xhTile} we know $e(W\DI{7,11}) = 1/3$ which is impossible if $|W[7]|_{\o}\leq 1$  and if $7 < 3(k-1)+1$, then $e(W\DI{7,9}) = 0$ which is also impossible if $|W[7]|_{\o}\leq 1$. The configuration on Figure \ref{subfig:5xhtuilesB} forces $W[6]$ to contain 3 degree 2 cells. This implies that $|W[7]|_{\o}\leq 2$. If $7 = 3(k-1)+1$, then we know from Corollary \ref{cor:5xhTile} that $e(W\DI{7,11}) = 1/3$ and from Corollary \ref{cor:5xhTileH}, we know the configuration of $W\DI{9,11}$. This configuration forces 6 cells of $W\DI{8,11}$ to be empty. We also know 3 cells of $W[7]$ are forced to be empty as well. So 9 cells of $W\DI{7,11}$ are forced to be empty in contradiction with $e(W\DI{7,11}) = 1/3$ (Figure \ref{subfig:5xhtuilesB1}). If $7 < 3(k-1)+1$, then we have $e(W\DI{7,9}) = 0$. This  forces $|W[7]|_{\o}=2$ but with a degree 2 cell at the extremity (Figure \ref{subfig:5xhtuilesB2}) This leaves the configuration on Figure \ref{subfig:5xhtuilesC} as the only valid one.

\begin{figure}
\centering
\tikz\node[yscale=-1,inner sep=0,outer sep=0]{\input{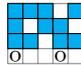}};
\caption{The unique configuration  of $W\DI{1,3}$ up to symmetry. "o"'s represent empty forced cells of $W[4]$.}
\label{fig:5xhnbcell}
\end{figure}

\begin{figure}
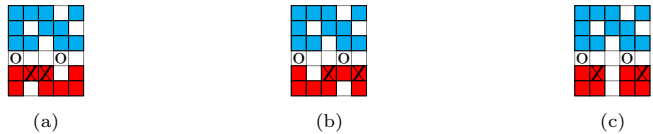

\centering
\begin{subfigure}[t]{0.3\textwidth}\centering
\tikz\node[yscale=-1,inner sep=0,outer sep=0]{\input{tikz/5xh2cellsr4A.tikz}};
\caption{}
\label{subfig:5xh2cellsr4A}
\end{subfigure}
\begin{subfigure}[t]{0.3\textwidth}\centering
\tikz\node[yscale=-1,inner sep=0,outer sep=0]{\input{tikz/5xh2cellsr4B.tikz}};
\caption{}
\label{subfig:5xh2cellsr4B}
\end{subfigure}
\begin{subfigure}[t]{0.3\textwidth}\centering
\tikz\node[yscale=-1,inner sep=0,outer sep=0]{\input{tikz/5xh2cellsr4C.tikz}};
\caption{}
\label{subfig:5xh2cellsr4C}
\end{subfigure}
\caption{ Joint configurations of $W[4]$ and $W[5]$:  Two mandatory empty cells in $W[4]$ plus two other empty cells in $W[4]$ forced by two cells of degree two in in $W[5]$ marked with "X"'}
\label{fig:5xh2cellsr4}
\end{figure}

\begin{figure}
\centering
\begin{subfigure}[t]{0.47\textwidth}\centering
\tikz\node[yscale=-1,inner sep=0,outer sep=0]{\input{tikz/5xhfirst4.tikz}};
\caption{The unique configuration  of $W\DI{1,4}$ up to symmetry.}
\label{subfig:5xhfirst4}
\end{subfigure}
\begin{subfigure}[t]{0.47\textwidth}\centering
\tikz\node[yscale=-1,inner sep=0,outer sep=0]{\input{tikz/5xh5bench.tikz}};
\caption{Unique configuration when $|W[5]|_{\o} = 2$. "o"'s mark mandatory empty cells.}
\label{subfig:5xh5bench}
\end{subfigure}
\caption{}
\label{fig:5xhfirstfew}
\end{figure}

\begin{figure}
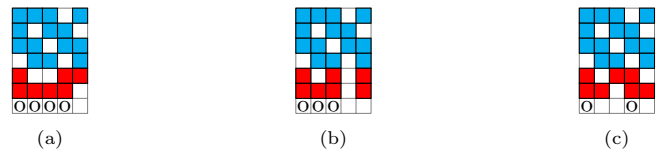

\centering
\begin{subfigure}[t]{0.3\textwidth}\centering
\tikz\node[yscale=-1,inner sep=0,outer sep=0]{\input{tikz/5xhtuilesA.tikz}};
\caption{}
\label{subfig:5xhtuilesA}
\end{subfigure}
\begin{subfigure}[t]{0.3\textwidth}\centering
\tikz\node[yscale=-1,inner sep=0,outer sep=0]{\input{tikz/5xhtuilesB.tikz}};
\caption{}
\label{subfig:5xhtuilesB}
\end{subfigure}
\begin{subfigure}[t]{0.3\textwidth}\centering
\tikz\node[yscale=-1,inner sep=0,outer sep=0]{\input{tikz/5xhtuilesC.tikz}};
\caption{}
\label{subfig:5xhtuilesC}
\end{subfigure}
\caption{All configurations of $W\DI{1,6}$. Cells of $W\DI{5,6}$ are in red. "o"'s mark mandatory empty cells.}
\label{fig:5xhtuiles}
\end{figure}

\begin{figure}
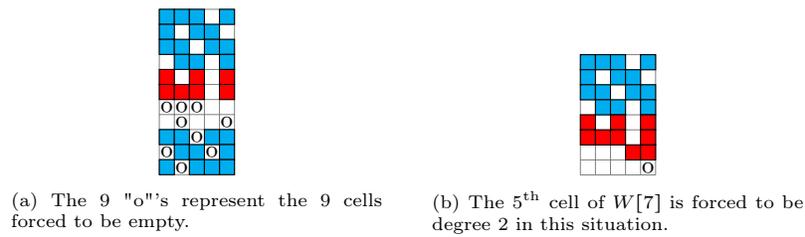

\centering
\begin{subfigure}[t]{0.45\textwidth}\centering
\tikz\node[yscale=-1,inner sep=0,outer sep=0]{\input{tikz/5xhtuilesB1.tikz}};
\caption{The 9 "o"'s represent the 9 cells forced to be empty.}
\label{subfig:5xhtuilesB1}
\end{subfigure}
\begin{subfigure}[t]{0.45\textwidth}\centering
\tikz\node[yscale=-1,inner sep=0,outer sep=0]{\input{tikz/5xhtuilesB2.tikz}};
\caption{The 5\textsuperscript{th} cell of $W[7]$ is forced to be degree 2 in this situation.}
\label{subfig:5xhtuilesB2}
\end{subfigure}
\caption{The two possible outcomes when $W\DI{5,6}$ are of the configuration in Figure \ref{subfig:5xhtuilesB}.}
\label{fig:5xhtuilesB}
\end{figure}
Now observe that $W[6]$ is forced to have the exact same configuration as $W[3]$ with both degree 2 cells at te exact same place. Thus the above arguments may be repeated  to show that if $W\DI{7,9}$ is an interior tile, then it has the same configuration as $W\DI{4,6}$ and so on for all interior factors
$$W\DI{4,6}, W\DI{7,9}, \dots,W\DI{3(k-2)+1,3(k-1)}.$$
Finally, we have that $W\DI{3(k-1)-1,3(k-1)}$ has the same configuration as $W\DI{2,3}$ and from Corollary \ref{cor:5xhTileH}, there are only two possible situations for $W\DI{3(k-1)+1,3k+2}$ (Figure \ref{fig:5xhfin}). The configuration in Figure \ref{subfig:5xhfinA} is rejected as it contains $\o$ cells of degree 3. This means the configuration in Figure \ref{subfig:5xhfinB} is forced and $W$ is  unique up to vertical and horizontal symmetry.

\begin{figure}
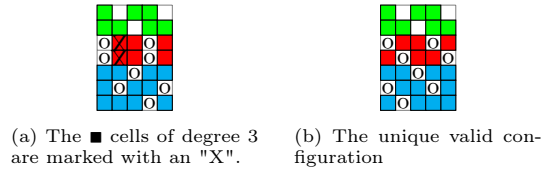

\centering
\begin{subfigure}[t]{0.3\textwidth}\centering
\tikz\node[yscale=-1,inner sep=0,outer sep=0]{\input{tikz/5xhfinA.tikz}};
\caption{The $\o$ cells of degree 3 are marked with an "X".}
\label{subfig:5xhfinA}
\end{subfigure}
\begin{subfigure}[t]{0.3\textwidth}\centering
\tikz\node[yscale=-1,inner sep=0,outer sep=0]{\input{tikz/5xhfinB.tikz}};
\caption{The unique valid configuration}
\label{subfig:5xhfinB}
\end{subfigure}
\caption{The two possible situations for $W\DI{3(k-1)+1,3k+2}$. In green is the factor $W\DI{3(k-1)-1,3(k-1)}$ that is forced to have the same configuration as $W\DI{2,3}$. In cyan is the factor $W\DI{3k,3k+2}$ which can only be in those two configurations from Corollary \ref{cor:5xhTileH}. "o"'s mark forced empty  cells. In red is the factor $W\DI{3(k-1)+1,3(k-1)+2}$ in which all cells that are not forced to be empty are "$\o$" in order to respect $e(W\DI{3(k-1)+1,3k+2})=1/3$.}
\label{fig:5xhfin}
\end{figure}
\end{proof}

\paragraph{\textbf{Proof of Corollary \ref{cor:5xhsature}}}

\begin{proof}
The unique 2-full $W \in \Wordsdhwvarshort{2}{8}{5}$ is represented on Figure \ref{fig:5x8unique}.  We can see that $W^t[5]$'s 3 bottom cells are degree 2. It also has one degree 2 cell in the factor in red. Here, $k=2$ so $W^t[5]$ has $4=k+2$ degree 2 cells. As we get $W_1 \in \Wordsdhwvarshort{2}{3(k+1)+2}{5}$ from $W_2 \in \Wordsdhwvarshort{2}{(3k+2)}{5}$ by stacking a copy of the red factor on top itself, $W_1^t[5]$ always has exactly 1 more degree 2 cell than $W_2^t[5]$, the number of degree 2 cells in $W' \in \Wordsdhwvarshort{2}{(3k+2)}{5}$ being equal to $k+2$ always stays true. The argument is entirely symmetric for $W^t[1]$.
\end{proof}

\begin{figure}
\centering
\input{tikz/5x8.tikz}
\caption{The unique 2-full $W \in \MWordsdhwvarshort{2}{8}{5}$.}
\label{fig:5x8unique}
\end{figure}

\subsection{Proof of Case 6. $(3k_1+1 \times 3k_2+2)$}

We need the following results.

\begin{corollary}\label{corole(ri)<4/3}
Let $W\in \Wordsdhwvarshort{2}{(3k_2+2)}{(3k_1+1)}$ with $e(W)=5/3$.
Then for all $i$ such that $4\leq i \leq h-3$ we have
$$e(W[i])\leq 1/3$$
\end{corollary}
\begin{proof}
Suppose that there exists $4\leq i \leq h-3$ such that $e(W[i])\geq 4/3$.  Then we have  $|W[i]|_\o\geq 2k_1+2,
|W[i]|_\e\leq k_1-1$ and from Lemma \ref{lem:ellc2} we obtain
$$
\ell(\lambda(W[i]))\leq k_1.
$$
Now  factorize $W=W_tW\DI{i-1,i+1}W_b$. From Proposition \ref{propc1c2c3e(c2)>0} we have 
\begin{align*}
e(W\DI{i-1,i+1})\leq 2(-k_1-1/3+k_1-4/3)+4/3=-2,\\
\Rightarrow e(W_t)+e(W_b)\geq 5/3+2
\end{align*}
which is impossible for any $h(W_t),h(W_b)\geq 3$ such that $h(W_t)+h(W_b)\equiv_3 2$.

\end{proof}

\begin{corollary}\label{corolhx3maxhx4}
Let $W\in \Wordsdhwvarshort{2}{h}{4}$ be a $h\times 4$ word such that  $e(W\DI{1,3})=2$. 
Then $|W[4]|_\o\leq\lfloor  \frac{2(h-1)}{3}\rfloor -1$. In particular 
\begin{align}
e(W[4])\leq - 8/3
\end{align}
\end{corollary}
\begin{proof} 
This is an immediate consequence of Corollary \ref{cor4xhn(c4)}.
\end{proof}

\begin{proposition}\label{prop:3x(3k+1)}
Let $W\in \Wordsdhwvarshort{2}{(3k_2+2)}{(3k_1+1)}$
Assume that $e(W) = 5/3$ and that all proper subwords $W'$ of $W$ satisfy $e(W') \leq \emax(W')$.
For $1\leq i\leq 3k_2$,  there is no factor $W\DI{i,i+2}$ of $W$ with $e(W\DI{i,i+2})=2$. 
\end{proposition}

\begin{proof}By contradiction. 
Suppose that $W\DI{i,i+2}$ exists with $e(W\DI{i,i+2})=2$. 
If $i=1$ and $W\DI{i,i+2}$ is at the top of $W$, we know from Corollary \ref{corolhx3maxhx4} that $e(W[4])\leq -8/3$ which implies that $e(W\DI{5,h})\geq 5/3-2+8/3=7/3$ which is impossible by minimality hypothesis. $i=3k_2$ is also impossible by symmetry.
If $W\DI{i,i+2}$ is an interior factor of $W$ (i.e. $1<i<3k_2$), then partition $W=W\DI{1,i-2}W\DI{i-1,i+3}W\DI{i+4,h}$ as a product of three factors.  We have
\begin{align*}
e(W[i-1])\leq -8/3 \mbox{ and } e(W[i+3])& \leq -8/3 \\
\Rightarrow e(W\DI{1,i-2})+e(W\DI{i+4,h})&\geq 5/3-2+8/3+8/3=15/3
\end{align*}
so that either $e(W\DI{1,i-2})\geq 8/3$ or $e(W\DI{i+4,h})\geq 8/3$
which, again, is impossible.
\end{proof}

\begin{proposition}Let $W\in \Wordsdhwvarshort{2}{(3k_2+2)}{(3k_1+1)}$.
Assume that $e(W) = 5/3$ and that all proper factors $W'$ of $W$ satisfies $e(W')\leq \emax(W')$. 
There is no factor $W\DI{i,i+2}$  of $W$ with $e(W\DI{i,i+2})<-1$. 
\end{proposition}

\begin{proof}
By contradiction. Suppose that $W\DI{i,i+2}$ exists with $e(W\DI{i,i+2})=-2$. It is immediate that  $W\DI{i,i+2}$ is not a top or bottom factor of $W$. Partition $W$ as product of three factors $W=W_tW\DI{i,i+2}W_b$. Then
\begin{align*}
e(W_b)+e(W_t)>5/3+2=11/3 \mbox{ and } h(W_b)+h(W_t)\equiv_3 2 
\end{align*}
so that $e(W_b)\geq 2$ or $e(W_t)\geq 2$. This is impossible by minimality hypothesis. 
\end{proof}

\begin{proposition}\label{prop{3k_2+2}{3k_1+1}-5/3}
Let $W\in \Wordsdhwvarshort{2}{(3k_2+2)}{(3k_1+1)}$.
Assume that $e(W) = 5/3$ and that all proper factor $W'$ of $W$ satisfies $e(W')\leq \emax(W')$. Then 
$e(W[i])> -5/3$ for all $1\leq i\leq h$.
\end{proposition}

\begin{proof}By contradiction. 
Suppose that for some $i$, $e(W[i])\leq -5/3$. Then write $W=W_tW[i]W_b$ as product of three factors 
so that $e(W_t)+e(W_b)\geq 5/3+5/3$ and $h(W_t)+h(W_b)\equiv_3  0$ so that
$(h(W_t),h(W_b))\equiv_3  (0,0) or (1,2)$ up to symmetry. But the two inequalities 
\begin{align*}
(h(W_t),h(W_b))&\equiv  (0,0) \Mod 3 \Rightarrow e(W_t)+e(W_b)\leq 1+1\\
(h(W_t),h(W_b))&\equiv  (1,2) \Mod 3 \Rightarrow e(W_t)+e(W_b)\leq 4/3+2/3
\end{align*}
are in contradiction with  $e(W_t)+e(W_b)\geq 10/3$. This completes the proof. 
\end{proof}

\begin{proposition}\label{prop6x(3k+1)}
Let $t<k_2$ $W\in \Wordsdhwvarshort{2}{(3k_2+2)}{(3k_1+1)}$ and let $W\DI{1,3t}$ be the $3t\times (3k_1+1)$ upper factor of $W$.
Assume that $e(W) = 5/3$ and that all proper factor $W'$ of $W$ satisfies $e(W')\leq \emax(W')$. Then 
i) $e(W\DI{1,3})=1$, ii) $e(W\DI{1,3t})=1$,  and  iii) $e(W\DI{3t+1,3(t+1)})=0$ for all $1\leq t\leq k_2-2$.
\end{proposition}

\begin{proof}
i) From Proposition \ref{prop:3x(3k+1)} we know that $e(W\DI{1,3})\leq 1$. But $e(W\DI{1,3})\leq 0$ implies $e(W\DI{4,h})\geq 5/3$ in contradiction with the minimality hypothesis. So we must have $e(W\DI{1,3})=1$. \\
ii) If $e(W\DI{4,6}) \geq 1$ then $e(W\DI{1,6}) \geq 2$ which is impossible. If $e(W\DI{4,6}) \leq -1$ then $e(W\DI{7,h}) \geq 5/3$ in contradiction with the hypothesis. So we must have $e(W\DI{4,6}) =0$ and $e(W\DI{1,6}) = 1$. We can repeat the same argument and prove that $e(W\DI{3t+1,3t+3}) =0$ for all $1\leq t \leq k_2-2$ so that $e(W\DI{1,3t}) =1$ for all $t\leq k_2-1$.\\
iii) This was proved in $ii)$.
\end{proof}

\begin{proposition}\label{prop5x(3k+1)}
Let $t<k_2-1$, $W\in \Wordsdhwvarshort{2}{(3k_2+2)}{(3k_1+1)}$ and let $W\DI{1,3t+2}$ be the $(3t+2)\times (3k_1+1)$ upper factor of $W$.
Assume that $e(W) = 5/3$ and that all proper factor $W'$ of $W$ satisfies $e(W')\leq \emax(W')$. Then 
i) $e(W\DI{1,3t+2})=2/3$,
ii) $e(W[3t])=1/3$ for all $2\leq t\leq k_2-1$. 
\end{proposition}

\begin{proof}
i)  From Proposition \ref{prop6x(3k+1)} we have $e(W\DI{3t+3,3k_2+2})=1$ which implies $e(W\DI{1,3t+2})=2/3$.
ii) Since  from Proposition \ref{prop6x(3k+1)} $e(W\DI{1,3t})=1$ and from i) $e(W\DI{1,3t-1})=2/3$, we must have $e(W[3t])=e(W\DI{1,3i}) -e(W\DI{1,3t-1})= 1/3$ and
\begin{align}
|W[3t]|_\o =2k_1+1 \; \mbox{ for all } 2\leq t\leq k_2-1.
\end{align}
\end{proof}

\begin{proposition}\label{prop5x(3k+1)2}
Let $W\in \Wordsdhwvarshort{2}{(3k_2+2)}{(3k_1+1)}$ and $W\DI{1,4}$ the $4\times (3k_1+1)$ upper factor of $W$.  
Assume that $e(W) = 5/3$ and that all proper factor $W'$ of $W$ satisfies $e(W')\leq \emax(W')$. Then 
\begin{align}
i) \;\{e(W\DI{1,4}),e(W\DI{5,h})\}&=\{1/3,4/3\},\\
ii) \;e(W\DI{1,4})=1/3 &\Rightarrow e(W[5])=1/3 \mbox{ and } e(W[4])=-2/3,\\
e(W\DI{1,4})=4/3 &\Rightarrow e(W[5])=-2/3 \mbox{ and } e(W[4])=1/3
\end{align}
\end{proposition}

\begin{proof}
i) This is due to the fact that  $e(W\DI{1,4})\leq 4/3$, $\{e(W\DI{5,h})\leq 4/3$ and $e(W\DI{1,4})+e(W\DI{5,h}=5/3$.
ii) If $e(W\DI{1,4})=1/3$ then from Proposition \ref{prop6x(3k+1)}i) $e(W[4])=-2/3$ and from $e(W\DI{1,5})=2/3$ we obtain $e(W[5])=1/3 $. 
Similarly if  $e(W\DI{1,4})=4/3$ then from Proposition \ref{prop6x(3k+1)}i) $e(W[4])=1/3$
and from $e(W\DI{1,5})=2/3$ we obtain $e(W[5])=-2/3$.
\end{proof}

\subsubsection{ Hypothesis  $e(W\DI{1,4})=4/3$}
We show that the hypothesis $e(W\DI{1,4})=4/3$ leads to a contradiction.  

\begin{proposition}  \label{prop_e(w[5])-5/3}
For $k_1\geq 2$ , let $W\in \Wordsdhwvarshort{2}{5}{(3k_1+1)}$. 
$$e(W\DI{1,4})=4/3 \mbox{ and } e(W[4])=1/3 \mbox{ and } e(W[3])=1/3 \Rightarrow e(W[5])\leq -5/3
$$ 
\end{proposition}

\begin{proof} This is a consequence of Corollary \ref{cor4xh_n2(c1)} that implies that $|W[5]|_\e\geq k+2$.
\end{proof}

\begin{proposition}  \label{prop4xh4/3a} For $k_1\geq 2$ , let $W\in \Wordsdhwvarshort{2}{(3k_2+2)}{(3k_1+1)}$ with $e(W)=5/3$. 
Then $e(W\DI{1,4})=4/3$ is impossible.
\end{proposition}

\begin{proof} Assume  $e(W\DI{1,4})=4/3$. Then on one hand $e(W[5])\leq -5/3$ from Proposition \ref{prop_e(w[5])-5/3}. On the other hand we know that $e(W\DI{1,5})=2/3$ which implies that 
$e(W[5])=e(W\DI{1,5})-e(W\DI{1,4})=2/3-4/3=-2/3$. We have a contradiction and $e(W\DI{1,4})\neq 4/3$\end{proof}

\subsubsection{ Hypothesis  $e(W\DI{1,4})=1/3$}
We prove  that the hypothesis $e(W\DI{1,4})=1/3$ leads to a contradiction and that words 
$W\in \Wordsdhwvarshort{2}{(3k_2+2)}{(3k_1+1)}$ with $e(W)=5/3$ are impossible.

\begin{corollary}\label{cor5xh_e(c1)e(c5)}
Let $W\in \Wordsdhwvarshort{2}{5}{(3k+1)}$ with $e(W[2])=e(W[3])=e(W[4])=1/3$.
Then 
\begin{align*}
e(W[1])\leq -5/3 \mbox{ or } e(W[5])\leq -5/3. 
\end{align*}
\end{corollary}

\begin{proof}
This is a direct consequence of Propositions \ref{prop3xhlambda(c2)}, \ref{prop3xh2^k1},\ref{prop3xh_32^(k-2)1^2}.  
\end{proof}

\begin{corollary}\label{cor{(3k_2+2)}{(3k_1+1)}}
There exists no word $W\in \Wordsdhwvarshort{2}{(3k_2+2)}{(3k_1+1)}$ with $e(W)=5/3$.
\end{corollary}

\begin{proof}
We know from Proposition \ref{prop4xh4/3a} that the only possible excess of $W\DI{1,4}$ is $e(W\DI{1,4})=1/3$.
Using symmetry of $W$, we know that there exists an integer $i<k_2$ such that $e(W\DI{1,3i+1})=4/3$
Assume without loss of generality that the integer $i$ is minimal such that $e(W\DI{1,3i+1})=1/3$ and $e(W\DI{1,3i+4})=4/3$. 
Then we have $e(W[3i+2])=e(W[3i+3])=e(W[3i+4])=1/3$ which implies from Corollary \ref{cor5xh_e(c1)e(c5)} that 
$e(W[3i+5])\leq -5/3$ or   $e(W[3i+1])\leq -5/3$ in contradiction with Proposition \ref{prop{3k_2+2}{3k_1+1}-5/3}
We have thus proved that $e(W\DI{1,4})=1/3$ is impossible and that  $e(W)=5/3$ is also impossible.
\end{proof}


\end{document}

%% file: macros.tex
\usepackage{enumerate}
\usepackage{multirow}
\usepackage{blkarray}
\usepackage{graphicx}
\usepackage{amssymb}
\usepackage{MnSymbol}
\usepackage{amsmath}
\usepackage{mathtools}
\usepackage{stmaryrd}
\usepackage{todonotes}
\usepackage{cancel}
\usepackage{tikz}
\usepackage{comment}
\usetikzlibrary{arrows}
\usetikzlibrary{decorations.pathmorphing}

\usepackage{subcaption}
\usepackage{cases}
\newcommand{\Z}{\mathbb{Z}}
\newcommand{\Zp}{\mathbb{Z}_{>0}}
\newcommand{\Q}{\mathbb{Q}}
\newcommand{\DI}[1]{\llbracket #1 \rrbracket}
\newcommand{\DIhw}{\DI{1, h} \times \DI{1, w}}
\newcommand{\bd}{\partial}

\newenvironment{wcells}[1]
  {\def\tabcolsep{0.01mm}\tabular{#1}}
  {\endtabular}

\newcommand{\e}{\square}
\renewcommand{\o}{\blacksquare}
 \newcommand{\purple}{\tikz{\path[draw=black,fill=violet] (0,0) rectangle (.22cm,.22cm);}}
\newcommand\WordshwA{\mathcal{W}_{h \times w}(A)}
\newcommand\WordshwAvar[2]{\mathcal{W}_{#1 \times #2}(A)}
\newcommand\Wordshw{\mathcal{W}_{h \times w}(\{\e,\o\})}
\newcommand\Wordshwshort{\mathcal{W}_{h \times w}}

\newcommand\Wordsdhwshort{\mathcal{W}^{\leq d}_{h \times w}}

\newcommand\Wordszhwshort{\mathcal{W}^{\leq 0}_{h \times w}}

\newcommand\Wordsohwshort{\mathcal{W}^{\leq 1}_{h \times w}}

\newcommand\Wordstwhwshort{\mathcal{W}^{\leq 2}_{h \times w}}

\newcommand\Wordsthhw{\mathcal{W}^{\leq 3}_{h \times w}(\{\e,\o\})}
\newcommand\Wordsthhwshort{\mathcal{W}^{\leq 3}_{h \times w}}
\newcommand{\emax}{\mathbf{\hat{e}_{max}}}

\newcommand\Wordsfhwshort{\mathcal{W}^{\leq 4}_{h \times w}}

\newcommand{\maxdhw}[3]{\mathrm{max}_{\leq #1}(#2,#3)}
\newcommand{\mdhw}[3]{m_{\leq #1}(#2,#3)}

\newcommand{\maxshw}[2]{\mathrm{max}_{s}(#1,#2)}

\newcommand{\Wordsdhwvar}[3]{\mathcal{W}^{\leq #1}_{#2\times #3}(\{\e,\o\})}
\newcommand{\Wordsdhwvarshort}[3]{\mathcal{W}^{\leq #1}_{#2\times #3}}
\newcommand{\MWordsdhwvarshort}[3]{\mathcal{MW}^{\leq #1}_{#2\times #3}}

\newcommand{\indic}{\mathbb{I}}

\newcommand{\hcat}{\overt}
\newcommand{\vcat}{\ominus}

\renewcommand{\bmod}{\ \mathrm{mod}\ }
\newcommand{\Mod}[1]{\ (\mathrm{mod}\ #1)}

\newcommand\caseLouis[3]{
	\draw [draw=black, fill=#3] (#1, #2) -- (#1 +1, #2) -- (#1 +1, #2 +1) -- (#1, #2 +1) -- cycle;
}
\newcommand\lcase[4]{
	\draw [draw=black, fill=#3] (#1, #2) -- (#1 +1, #2) -- (#1 +1, #2 +1) -- (#1, #2 +1) -- cycle;
	\node[] at (#1+0.5,#2+0.5){#4};
}
\newcommand\ccase[2]{
	\caseLouis{#1}{#2}{cyan}
}

\newcommand\grid[4]{
	\filldraw[step=1cm,gray,very thin,fill=white] (#3,#4) grid (#3+#1,#4+#2);
}

